\documentclass[11pt,hidelinks]{article}

\input{MyMacros.sty} 
\hypersetup{
    colorlinks = true,
    allcolors = blue
}
\usepackage{authblk}
\usepackage{mathtools}
\usepackage{mleftright}
\usepackage{caption}
\usepackage{placeins}
\newcolumntype{L}[1]{>{\raggedright\arraybackslash}p{#1}}

\usepackage{calrsfs}
\DeclareMathAlphabet{\pazocal}{OMS}{zplm}{m}{n}
\renewcommand{\cV}{\pazocal{V}}
\renewcommand{\cN}{\pazocal{N}}
\renewcommand{\cA}{\pazocal{A}}
\renewcommand{\cT}{\pazocal{T}}
\newcommand{\cTN}{{\cT\!\cN}}

\newcommand{\tu}{\tilde{u}}
\newcommand{\tU}{\tilde{U}}

\newcommand{\tv}{\tilde{v}}
\newcommand{\tV}{\tilde{V}}

\newcommand{\Var}{\mathrm{Var}}

\newcommand{\stdunif}{\mathrm{Unif}(0,1)}

\newcommand{\imi}{\mathrm{i}}

\renewcommand{\left}{\mleft}
\renewcommand{\right}{\mright}

\usepackage{comment}
\newcommand{\bdelta}{\bm{\delta}}

\newcommand\nnfootnote[1]{%
  \begin{NoHyper}
  \renewcommand\thefootnote{}\footnote{#1}%
  \addtocounter{footnote}{-1}%
  \end{NoHyper}
}

\title{A Tractable Family of Smooth Copulas with Rotational Dependence: Properties, Inference, and Application}

\author[1]{Micha\"{e}l Lalancette$^*$}
\author[2]{Robert Zimmerman$^*$}
\affil[1]{Département de mathématiques, Université du Québec à Montréal}
\affil[2]{Department of Mathematics, Imperial College London}

\date{\today}

\begin{document}

\maketitle
\nnfootnote{$^*$Equal contribution}

\begin{abstract}
    We introduce a new family of copula densities constructed from univariate distributions on $[0,1]$. Although our construction is structurally simple, the resulting family is versatile: it includes both smooth and irregular examples, and reveals clear links between properties of the underlying univariate distribution and the strength, direction, and form of multivariate dependence. The framework brings with it a range of explicit mathematical properties, including interpretable characterizations of dependence and transparent descriptions of how rotational forms arise. We propose model selection and inference methods in parametric and nonparametric settings, supported by asymptotic theory that reduces multivariate estimation to well-studied univariate problems. Simulation studies confirm the reliable recovery of structural features, and an application involving neural connectivity data illustrates how the family can yield a better fit than existing models.
\end{abstract}

\tableofcontents

\section{Introduction}

A \emph{copula} is a $d$-dimensional distribution function $C:[0,1]^d \to [0,1]$ with uniform marginal distributions: $C(1, \dots, 1, u_j, 1, \dots, 1) = u_j$ for $u_j \in [0,1]$ and $j \in \{1, \dots, d\}$. Copulas provide a rigorous way to isolate and model statistical dependence; by Sklar's theorem \citep{Sklar:1959}, any multivariate distribution can be decomposed into its margins and a copula which captures the dependence structure of the distribution independently of the margins. An extensive literature \citep{schweizerThirtyYearsCopulas1991,joeMultivariateModelsMultivariate1997,nelsenIntroductionCopulas2006,durantePrinciplesCopulaTheory2015} reveals a vast array of dependence structures that can be described explicitly using copulas, which have been used in countless applied fields including risk management \citep{embrechts2001modelling}, finance \citep{cherubiniCopulaMethodsFinance2004}, hydrology \citep{genest2007metaelliptical}, environmental science \citep{salvadoriExtremesNatureApproach2007}, and epidemiology \citep{shihInferencesAssociationParameter1995}, to name only a few. A central challenge is to construct copula families that are both flexible enough to capture a wide range of dependence structures and tractable enough to allow rigorous statistical analysis.

One such structure is rotational dependence, which is more naturally described in terms of the dependence structure of angular (or circular) random variables. Examples arise naturally in meteorology, including wind direction or temporal variables such as time of day or time of year; similar situations abound in the earth sciences, physics, and psychology, among other fields \citep[for a broad overview of applications and methods in multivariate directional statistics, see][]{mardiaDirectionalStatistics2009}. Related work has introduced \emph{circulas}, the analogue of copulas for circular data  \citep{jonesClassCirculasCopulas2015a,juppCopulaeProductsCompact2015,ameijeiras-alonsoSmoothedCirculasNonparametric2024}; these provide a copula-like decomposition on the torus with circular uniform marginals and have been studied in both parametric and nonparametric regimes. They allow one to directly model dependence structure in multivariate angular data. Another surprising appearance of rotational dependence is in neuroscience, where oscillatory phases across different brain regions may tend to advance or recede together. \citet{kleinTorusGraphsMultivariate2020} introduced a graphical model for multivariate phase data in order to capture such dependence, motivated by local field potential recordings from the hippocampus and prefrontal cortex of rhesus monkeys \citep{brincatFrequencyspecificHippocampalprefrontalInteractions2015}, and found that their data is well explained by a tree-structured version of their \emph{phase difference model} with uniform marginals. By results in the present paper, this model can be formulated as follows: for each adjacent pair $(j, k)$ of locations, two standardized phase angles $U_j,U_k \sim \stdunif$ are coupled such that independently of $U_j$, $U_j - U_k \bmod{1}$ follows a von Mises distribution, the parameters of which control the strength of angular dependence. The bivariate distributions of the adjacent pairs are then assembled by \citet{kleinTorusGraphsMultivariate2020} into a tree graphical model \citep{maathuisHandbookGraphicalModels2018}.

In this paper, we present a new class $\cC$ of absolutely continuous copulas on $[0,1]^d$ that reduces the multivariate dependence problem to a univariate construction, thereby achieving both tractability and flexibility. Each copula in $\cC$ is ``generated'' by an underlying univariate density $f$ on $[0,1]$ whose smoothness and dispersion determine both the smoothness of the corresponding copula density and the strength of dependence between the marginal components. Despite the apparent simplicity of its construction, the family $\cC$ is very rich: in the bivariate case, for example, it generalizes the aforementioned \citet{kleinTorusGraphsMultivariate2020} model by allowing $U_j - U_k \bmod{1}$ to follow \emph{any} continuous distribution on $[0,1]$, rather than being restricted to the parametric class of von Mises distributions. This generalization allows far more flexibility in the form of angular dependence. The class contains both statistically motivated distributions and highly pathological examples, can exhibit extreme positive and negative dependence, and admits transparent expressions for concordance measures such as Kendall's tau and Spearman's rho. In arbitrary dimensions, a wide array of attractive mathematical properties make $\cC$ an interesting addition to the copula literature.

In order to allow for negative rotational dependence between two variables (termed ``reflectional'', as opposed to rotational, dependence in \cite{kleinTorusGraphsMultivariate2020}), we close the family $\cC$ under reflection of any axis. In addition to the density $f$, this adds a discrete parameter which essentially contains information about which of the $d$ axes are reflected. We refer to the task of learning the discrete parameter as model selection, and we introduce a nonparametric method to solve that problem which is shown to be consistent. We further discuss parametric inference within the class $\cC$, where it becomes clear that maximum likelihood estimation in flexible parametric families poses no challenge beyond the univariate case. On the nonparametric side, we analyze kernel density estimators and establish asymptotic properties using empirical process techniques. We illustrate our methodology through an extensive simulation study, and we then revisit the neural connectivity data of \citet{kleinTorusGraphsMultivariate2020}, showing that our approach not only confirms the findings in that paper, but also offers more flexibility in modelling the pairwise interactions therein. These illustrations demonstrate that the proposed copula family is able to capture rotational dependence patterns while retaining ease of interpretation and offering rigorous statistical guarantees.

The remainder of the paper is organized as follows. \cref{sec:construction} formally constructs the family $\cC$ and details many of its interesting properties. \cref{sec:examples} illustrates the breadth of the family with concrete examples, both simple and highly pathological. \cref{sec:inference} develops model selection as well as parametric and nonparametric parameter estimation. \cref{sec:simsandapps} presents the simulation study and neural connectivity application. \cref{sec:discussion} concludes with a discussion and directions for future work.

\subsection{Notation and conventions}

We now define some notation and conventions which will be used repeatedly in the sequel.

Throughout, we fix an integer dimension $d \geq 2$ whose value is clear from context. Bold elements always refer to vectors (written in lower case) or random vectors (written in upper case). For $\bx = (x_1, \ldots, x_d) \in \R^d$ and $j \in \{1,\ldots, d\}$, we write $\bx_{-j} = (x_1,\ldots,x_{j-1},x_{j+1},\ldots,x_d) \in \R^{d-1}$ for $\bx$ with its $j$th element removed, and for $J \subseteq \{1,\ldots,d\}$ we write $\bx_J \in \R^{|J|}$ for the subvector of $\bx$ indexed by the elements of $J$ and $\bx_{-J}$ for $\bx_{J^c}$, the complement being taken with respect to $\{1, \dots, d\}$. We slightly abuse notation and write $(\bx_J, \bx_{-J})$ for the vector with entries $\bx_J$ in positions $J$ and entries $\bx_{-J}$ in positions $J^c$; the same conventions apply to random vectors. The $j$th canonical basis vector in $\R^d$ is denoted $\be_j$, and the zero and all-ones vectors in $\R^d$ are denoted by $\bzero$ and $\bone$ respectively. Inequalities between vectors are taken componentwise: $\bx \leq \by$ means $x_j \leq y_j$ for all $j \in \{1,\ldots,d\}$, and $[\bx, \by]$ refers to the set $\{\bu \in \R^d: \bx \leq \bu \leq \by\}$. By $x \bmod{1}$ we mean the fractional part of $x \in \R$, given by $x - \lfloor x \rfloor$, where $\lfloor x \rfloor := \sup \{m \in \Z: m \leq x\}$ denotes the largest integer upper bounded by $x$. We denote the sum of $x_1$ and $x_2$ modulo $1$ --- that is, the quantity $(x_1 + x_2) \bmod{1}$ --- as $x_1 \oplus x_2$, and similarly $\bigoplus_{j=1}^d x_j := \left( \sum_{j=1}^d x_j \right) \bmod{1}$; we refer to the latter quantity as the \emph{wrapped sum} of $\bx$. We also refer to $x_1 - x_2 \bmod{1}$ as a \emph{wrapped difference}.

Unless noted otherwise, all densities and integrals are taken with respect to Lebesgue measure on $\R^m$, with $m$ clear from context. We define $\cF_{[0,1]}$ as the set of univariate densities on $[0,1]$ --- i.e., any $f \in \cF_{[0,1]}$ is measurable, non-negative and has a Lebesgue integral of $1$ over $[0,1]$ --- and for any fixed $f \in \cF_{[0,1]}$ under consideration, we write $F$ for its corresponding cdf. $\Pi(\bu) := \prod_{j=1}^d u_j$ denotes the $d$-dimensional independence copula. When they exist, the derivatives of a $d$-dimensional copula $C$ are notated as $\partial_j C(\bu):=\partial C(\bu)/\partial u_j$ and $\partial^2_{jk} C(\bu):=\partial^2 C(\bu)/(\partial u_j\,\partial u_k)$ for $j,k \in \{1,\ldots,d\}$. We index coordinates by $j,k,l \in \{1,\ldots,d\}$ and observations by $i \in \{1,\ldots,n\}$, where $n$ usually refers to a sample size. A hat over a function or parameter (such as $\hat C$ or $\hat \theta$) denotes an estimator based on a random sample. We often employ the shorthand $g(x):= \mathrm{expression}(x)$ to mean that the function $g$ is defined pointwise by the mapping $x \mapsto \mathrm{expression}(x)$, with its domain and codomain usually clear from context. For a proposition $P(x)$, we write the indicator function for $P(x)$ as $\One\{P(x)\} := \begin{cases}
    1, & P(x) \text{ is true} \\ 0, & P(x) \text{ is false}
\end{cases}$.

For $1 \leq p < \infty$ and a set $A$, we write $L^p(A)$ and $L^\infty(A)$ for the sets of $p$-integrable and bounded real-valued functions on $A$, respectively. For such functions, $\|\cdot\|_p$ and $\|\cdot\|_\infty$ denote the $p$ and supremum norm, respectively. We denote the circle group as $\T := \R/\Z$ with addition $\oplus$, which we identify with $[0,1)$ via the canonical projection $x \mapsto x \bmod{1}$. Under this identification, Haar measure on $\T$ coincides with Lebesgue measure $\lambda$ on $[0,1]$, with the points $0$ and $1$ identified. For this reason, we write Haar integrals as Lebesgue integrals over $[0,1]$ and $[0,1]^d$ (or subsets thereof); when the group structure is relevant, we write $\T$ explicitly. Finally, when complex numbers arise, we write $\imi \in \C$ for the imaginary unit.

\section{Construction of \texorpdfstring{$\cC$}{Cf} and basic properties}\label{sec:construction}

\subsection{The construction}\label{sub:construction}

Our construction is very simple and works in any dimension $d \geq 2$. Rather than explicitly defining a copula in $\cC$, we specify a copula density, which in turn determines the copula itself. Let $f$ be a measurable function on $[0,1]$, and define
\begin{equation}\label{eq:ourcopulaORIG}
    c_{f}(\bu) := f\left( \bigoplus_{j=1}^d u_j \right), \quad \bu \in [0, 1]^d.
\end{equation}
This definition immediately generalizes to the case in which the $u_j$ in the wrapped sum above are given arbitrary signs; thus, for any $\bs \in \{0,1\}^d$, we define\footnote{Note that an equivalent definition is given by $c^{\bs}_{f}(\bu) := f\left( \bigoplus_{j=1}^d (-1)^{s_j} u_j \right)$, which, while slightly more concise, suppresses the fact that $c^{\bs}_f$ is obtained by reflecting $c^{\bzero}_f$ around the hyperplanes $\{\bu \in [0,1]^d : u_j = \tfrac{1}{2}\}$ for each $j$ such that $s_j=1$.}
\begin{equation}\label{eq:ourcopula}
    c^{\bs}_{f}(\bu) := f\left( \bigoplus_{j=1}^d u_j^{1-s_j} (1 - u_j)^{s_j} \right), \quad \bu \in [0,1]^d,
\end{equation}
so that $c_f = c^{\bzero}_{f}$. For ease of notation, we often write $\tu_j := u_j^{1-s_j} (1 - u_j)^{s_j}$ and similarly for other variables, where in general a tilde on a variable $u \in [0,1]$ denotes $u^{1-s}(1-u)^{s}$, with the signature component $s \in \{0,1\}$ clear from context; the same convention holds for random variables $U$. Thus, we have the equivalent representation 
\[
    c^{\bs}_{f}(\bu) 
    :=
    f\left( \bigoplus_{j=1}^d \tu_j \right), \quad \bu \in [0,1]^d.
\]
Our first result shows that \cref{eq:ourcopula} yields a genuine copula density precisely when $f$ is itself a univariate density on $[0,1]$:

\begin{prop}\label{prop:Ccharacteriation}
    \cref{eq:ourcopula} defines a copula density if and only if $f \in \cF_{[0,1]}$. In that case, for fixed $\bs$ the copula is uniquely defined (up to almost everywhere equivalence) by $f$.
\end{prop}

The proof of \cref{prop:Ccharacteriation} depends on the following lemma, which essentially states that the pushforward of Haar measure on the direct product $\T^p$ by an affine wrapped sum is Haar measure on a lower-dimensional subspace of $\T^p$. More precisely, let $\Ab \in \Z^{m \times p}$ and $\bb \in \R^m$, with $m \leq p$. We define the map $\Psi_{\Ab,\bb}(\bu) := (\Ab \bu) \oplus \bb$ --- to be used again in the sequel --- where $\Ab \bu$ is computed in $\R^m$ and then reduced modulo $1$ componentwise.

\begin{lemm}\label{lem:Lebesguepushforward}
    If $\bU \sim \mathrm{Unif}([0,1]^p)$ and $\mathrm{rank}(\Ab) = m$, then $\Psi_{\Ab,\bb}(\bU) \sim \mathrm{Unif}([0,1]^m)$.
\end{lemm}

\begin{proof}
    Let $\bV \sim \mathrm{Unif}([0,1]^m)$. Because both $\Psi_{\Ab,\bb}(\bU)$ and $\bV$ are supported on $[0,1]^m$, it suffices to verify that their respective characteristic functions $\varphi_{\Psi_{\Ab,\bb}(\bU)}$ and $\varphi_{\bV}$ are equal at the integer frequencies $2 \pi \bk$, for $\bk \in \Z^m$ \citep[see, e.g.,][Chapter XIX, Section 6, Theorem 1]{fellerIntroductionProbabilityTheory1968}. Indeed,
    \[
        \varphi_{\Psi_{\Ab,\bb}(\bU)}(2 \pi \bk)
        =
        \E\left[e^{2 \pi \imi \Psi_{\Ab,\bb}(\bU)^\top \bk}  \right]
        =
        e^{2 \pi \imi \bb^\top \bk} \E\left[e^{2 \pi \imi (\Ab \bU)^\top \bk}  \right],
    \]
    using the fact that $e^{2\pi \imi k(\bigoplus_{j=1}^m x_j)} = \prod_{j=1}^m e^{2\pi \imi k x_j}$ for any $\bx \in \R^m$ (since $\bigoplus_{j=1}^m x_j = \sum_{j=1}^m x_j - q$ for some $q \in \Z$ and $e^{2\pi \imi k q} = 1$). In turn, by Fubini's theorem the above display is equal to
    \[
        e^{2 \pi \imi \bb^\top \bk} \prod_{j=1}^p \int_0^1 e^{2 \pi \imi (\Ab^ \top \bk)_j u_j} \dif u_j
        =
        e^{2 \pi \imi \bb^\top \bk} \prod_{j=1}^p \One\{(\Ab^ \top \bk)_j = 0\}
        =
        e^{2 \pi \imi \bb^\top \bk} \One\{\Ab^\top \bk = \bzero\}
        =
        \One\{\bk = \bzero\}
    \]
    where the first equality holds because $\Ab^\top \bk \in \Z^p$ and $\int_0^1 e^{2 \pi \imi k u} \dif u = 0$ for $k \in \Z \setminus \{0\}$, and the last equality holds because $\Ab$ is of full row rank (i.e., $\Ab^\top \bk = \bzero$ implies $\bk = \bzero$). This is exactly the characteristic function of $\bV$ evaluated at $2 \pi \bk$, which proves the result.
\end{proof}

\begin{rem}
    The \emph{characters} of the group $\T^m$ are the continuous homomorphisms $\chi_{\bk}(\bu) := e^{2 \pi \imi \bk^\top \bu}$ into $S^1$, for $\bk \in \Z^m$. Haar measure $\lambda^{\otimes m}$ on $\T^m$ is characterized by the orthogonality relations $\int_{\T^m} \chi_{\bk}(\bu) \dif \lambda^{\otimes m}(\bu) = \One\{\bk = \bzero\}$ for all $\bk \in \Z^m$, and by the uniqueness theorem for Fourier-Stieltjes transforms on locally compact abelian groups \citep[][Theorem 1.3.6]{rudinFourierAnalysisGroups1962}, any probability measure $\mu$ on $\T^m$ is uniquely determined by the Fourier coefficients $\int_{\T^m} \chi_{\bk}(\bu) \dif \mu(\bu)$. Since $\Psi_{\Ab, \bb}(\bU)$ is $\T^m$-valued, its Fourier coefficients are exactly the quantities $\E[\chi_{\bk}(\Psi_{\Ab,\bb}(\bU))] = \varphi_{\Psi_{\Ab,\bb}(\bU)}(2 \pi \bk)$ computed above; hence, the law of $\Psi_{\Ab,\bb}(\bU)$ is $\lambda^{\otimes m}$.
\end{rem}

\begin{proof}[Proof of \cref{prop:Ccharacteriation}]
    Let $\bU \sim \mathrm{Unif}([0,1]^d)$. Suppose first that $f \in \cF_{[0,1]}$. Fix an index $j \in \{1,\ldots,d\}$ and $u_j \in [0,1]$, and define $\Psi_{j,u_j}^{\bs}(\bu_{-j}) := \left( \bigoplus_{k \neq j} \tu_k \right) \oplus \tu_j$. By \cref{lem:Lebesguepushforward}, $\Psi_{j,u_j}^{\bs}(\bU_{-j}) \sim \stdunif$, so
    \[
        \int_{[0,1]^{d-1}} c_f^{\bs}(\bu) \dif \bu_{-j}
        =
        \int_{[0,1]^{d-1}} f\left( \Psi_{j,u_j}^{\bs}(\bu_{-j})\right) \dif \bu_{-j}
        =
        \E\left[f\left( \Phi_{j,u_j}^{\bs}(\bU_{-j}) \right) \right]
        =
        \int_0^1 f(x) \dif x
        =
        1.
    \]
    That is, $c_f^{\bs}$ --- which is non-negative and measurable because $f$ itself is --- also has uniform margins and is therefore a copula density. Conversely, suppose that $c_f^{\bs}$ is a copula density. Observe that $\bigoplus_{j=1}^d \tU_j \sim \stdunif$ by \cref{lem:Lebesguepushforward}, so
    \[
        1
        =
        \int_{[0,1]^d} c_f^{\bs}(\bu) \dif \bu
        =
        \int_{[0,1]^d} f\left( \bigoplus_{j=1}^d \tu_j \right) \dif \bu
        =
        \E\left[f\left(\bigoplus_{j=1}^d \tU_j\right)\right]
        =
        \int_0^1 f(x) \dif x.
    \]
    Moreover, $f$ is non-negative and measurable because $c^{\bs}_f$ is; thus $f \in \cF_{[0,1]}$. For uniqueness, suppose there exist $f, g \in \cF_{[0,1]}$ such that $c^{\bs}_f(\bu) = c^{\bs}_g(\bu)$ for almost every $\bu \in [0,1]^d$, and let $D := \{x \in [0,1]: f(x) \neq g(x)\}$. Again using the fact that $\bigoplus_{j=1}^d \tU_j \sim \stdunif$, we have
    \[
        0 
        =
        \Prob\left(c^{\bs}_f(\bU) \neq c^{\bs}_g(\bU)
        \right)
        =
        \Prob\left(f\left(\bigoplus_{j=1}^d \tU_j\right) \neq g\left(\bigoplus_{j=1}^d \tU_j\right)\right) 
        = 
        \int_0^1 \One\{x \in D\} \dif x 
        =
        \lambda(D).
    \]
    That is, $D$ is a set of measure zero and so $f = g$ almost everywhere, which proves the result.
\end{proof}

We therefore define our class $\cC$ as the set of $d$-dimensional copulas which have a density in $\{c^{\bs}_{f}: f \in \cF_{[0, 1]}, \bs \in \{0,1\}^d\}$, and we write $C^{\bs}_{f}$ for the copula whose density is given by $c^{\bs}_f$; thus,
\[
    C_f^{\bs}(\bu) 
    :=
    \int_{[\bzero, \bu]} c_f^{\bs}(\bv) \dif \bv
    =
    \int_{[\bzero, \bu]} f\left( \bigoplus_{j=1}^d \tv_j \right) \dif \bv, \quad \bu \in [0, 1]^d.
\]
We refer to $f$ as the \emph{generator} of $C^{\bs}_{f}$ and to $\bs$ as the \emph{signature} of $C^{\bs}_{f}$. Strictly speaking, this class is over-parameterized since reflecting the generator $f$ around $1/2$ has the same effect as changing $\bs$ for $\bone - \bs$, in the sense that
\[
    c^{\bs}_f = c^{\bone - \bs}_g, \qquad \text{where} \qquad g(x):= f(1-x).
\]
This issue can be resolved by fixing one entry of $\bs$; we will return to this point in our discussion of statistical inference in \cref{sec:inference}. We will, on occasion, focus on the subclass $\{c^{\bzero}_f : f \in \cF_{[0,1]}\}$ specifically, because if $(U_1, \ldots, U_d) \sim C_f^{\bzero}$ then $(\tilde{U}_1, \ldots \tilde{U}_d) \sim C_f^{\bs}$; when we do, it will be convenient to retain the original notation of \cref{eq:ourcopulaORIG}.

\subsection{Some properties of the copulas \texorpdfstring{$C^{\bs}_{f}$}{Cfi}}\label{sub:properties}

It is interesting to note that while they can have extremely irregular densities (see \cref{sub:univariate}), the copulas $C^{\bs}_{f}$ enjoy a number of attractive properties. Several related to the signature are immediate consequences of \cref{eq:ourcopula}. For example, if $\bU \sim C^{\bs}_{f}$ and $\bs \in \{\bzero, \bone\}$, then $\bU$ is exchangeable. If $\bs \neq \bzero$, then $\bU$ can be recovered from $\bV \sim C^{\bzero}_f$ by defining $U_j = V_j^{1-s_j} (1-V_j)^{s_j}$ for $j \in \{1,\ldots,d\}$. Moreover, the relation $\sim$ on $\cC$ defined by $C^{\bs_1}_f \sim C^{\bs_2}_g$ if and only if $\bs_1 = \bs_2$ is an equivalence relation that partitions $\cC$ into equivalence classes of the form $\{C^{\bs}_f: f \in \cF_{[0,1]} \}$ for each $\bs \in \{0,1\}^d$, and by Tonelli's theorem each equivalence class is closed under arbitrary mixtures.

\subsubsection{Distributional properties}

We now list more important properties of the copulas as a series of propositions. The first involves stochastic representations which characterize our copulas and demonstrates the power of a misspecified signature, which will play important roles in sampling and the model selection procedure proposed in \cref{sec:inference}.

\begin{prop}\label{prop:basicproperties}
    
    If $\bU \sim C^{\bs}_{f}$ for some $f \in \cF_{[0, 1]}$ and $\bs \in \{0,1\}^d$, then the following hold:
    \begin{enumerate}
        \item \label{res:fandPi}For each $j \in \{1, \dots, d\}$,\begin{equation}\label{eq:oplusf}
        \bigoplus_{k=1}^d \tU_k \sim f, \quad \bU_{-j} \sim \Pi,
        \end{equation}
        and moreover $\bigoplus_{k=1}^d \tU_k$ and $\bU_{-j}$ are independent. As a consequence, $\bigoplus_{k=1}^d \tU_k \ \big|\ \bU_{-j} \sim f$.

        \item \label{res:wrongsig} For any $\bt \not \in \{\bs, \bone - \bs\}$,
        \begin{equation}\label{eq:wrongsig}
            \bigoplus_{j=1}^d U_j^{1-t_j} (1-U_j)^{t_j} \sim \stdunif.
        \end{equation}
    \end{enumerate}
    
\end{prop}

\begin{proof}

    We prove the two parts separately.
    
    \noindent\textbf{Proof of Part \ref{res:fandPi}.} Let $\cV = \{\bu \in [0,1]^d: \bigoplus_{j=1}^d \tu_j \neq 0\}$, and let $\xi^{\bs}_j(\bu) := (\bu_{-j}, \bigoplus_{k=1}^d \tu_k)$. The mapping $\xi^{\bs}_j$ is a smooth bijection from $\cV$ to itself, on which the determinant of its Jacobian is trivially equal to $(-1)^{s_j}$, and its inverse is given by
    \[
        \bv \mapsto \left(\bv_{-j}, (-1)^{s_j} \left(v_j - \sum_{k \neq j} \tv_k \right) \bmod{1} \right).
    \]
    Let $g^{\bs}_j$ be the density of $\xi^{\bs}_j(\bU)$. By the standard change-of-variables formula for vector-valued mappings, for $(\bv_{-j}, v) \in \cV$ we have
    \begin{align*}
        g^{\bs}_j(\bv_{-j}, v) 
        &= 
        c_f^{\bs}\left(\bv_{-j}, (-1)^{s_j} \left(v - \sum_{k \neq j} \tv_k \right) \bmod{1} \right) |(-1)^{s_j}|\\
        &=
        f\left( \sum_{k \neq j} \tv_k  + (-1)^{2s_j} \left(v - \sum_{k \neq j} \tv_k \right) \bmod{1} \right) \\
        &= 
        f(v).
    \end{align*}

    Since $\{\bu \in [0,1]^d: \bigoplus_{j=1}^d \tu_j = 0\}$ is a set of measure zero, the claimed distributions of $\bigoplus_{k=1}^d \tU_k$ and $\bU_{-j}$, as well as their independence, follow immediately. Note that the independence of the random vector $\bU_{-j}$ could have also been established on its own from \cref{lem:Lebesguepushforward}.

    \noindent\textbf{Proof of Part \ref{res:wrongsig}.}
    Let $R^{\bt} := \bigoplus_{j=1}^d U_j^{1-t_j}(1-U_j)^{t_j} = \bigoplus_{j=1}^d (-1)^{t_j}U_j$, and let $\kappa_j := (-1)^{t_j} - (-1)^{s_j}$ for each $j \in \{1, \ldots, d\}$. By assumption, there exists at least one $k \in \{1,\ldots,d\}$ such that $t_k = s_k$, and at least one $l \in \{1,\ldots,d\}$ such that $t_l \neq s_l$. We have that
    \[
        R^{\bt} 
        =
        R^{\bs} \oplus \sum_{j \neq k} \kappa_j U_j 
        =
        R^{\bs} \oplus \, \kappa_l U_l \, \oplus \sum_{j \neq k,l} \kappa_j U_j.
    \]
    Now, it follows from the result in Part~\ref{res:fandPi} that the three components of the wrapped sum above are independent. Moreover, each $\kappa_j \in \{-2,0,2\}$, and in particular, $\kappa_k = 0$ and $\kappa_l \in \{-2,2\}$. Thus
    \begin{equation} \label{eq:RtgivenRs}
        R^{\bt} \mid \Big\{ R^{\bs} = r, \sum_{j \neq k,l} \kappa_j U_j = u \Big\} 
        \stackrel{d}{=}
        \kappa_l V \oplus (r + u),
    \end{equation}
    where $V \sim \stdunif$, which itself follows a $\stdunif$ distribution by \cref{lem:Lebesguepushforward}. That $R^{\bt} \sim \stdunif$ then follows, since the conditional distribution in \cref{eq:RtgivenRs} does not depend on $r$ and $u$.
    \end{proof}

The stochastic representations given in Part~\ref{res:fandPi} of \cref{prop:basicproperties}  prompt a discussion of sampling. Generally speaking, one can use the inverse Rosenblatt transform (also called the \emph{conditional distribution method}) \citep{rosenblattRemarksMultivariateTransformation1952} to sample a random vector $\bU$ distributed according to a general copula $C$; however, this technique requires sampling from the conditional distribution of $U_j \mid (U_1, \ldots, U_{j-1})$ for each $j \in \{2,\ldots,d\}$. For a general copula $C$, this is typically done via the inverse probability transform method, using calculations which involve all $d$ partial derivatives of $C$ and can be extremely tedious; moreover, closed forms may not even exist.

Fortunately, Part~\ref{res:fandPi} of \cref{prop:basicproperties} implies a simple form for those conditional distributions for a copula $C^{\bs}_f$: when $j < d$, the marginal component $U_j$ is independent of $(U_1, \dots, U_{j-1})$, and $U_d \mid \bU_{-d}$ is distributed according to a rotated version of $f$. This yields a method to sample from any $C^{\bs}_f \in \cC$. The method, given by \cref{algo:sampler} below, is straightforward to implement, and its time complexity is linear in $d$; we only need to be able to sample from $f$, and no functional inverses of any kind are otherwise required.

\begin{algo}[A sampler for $C^{\bs}_f$]\label{algo:sampler} \phantom{a}
    \begin{enumerate}[(1), topsep=0pt]
    \item Sample $U_1, \ldots, U_{d-1} \iid \stdunif$.
    \item Sample $X \sim f$.
    \item \label{algo:step3} Calculate $U_d := (-1)^{s_d}\left( X - \sum_{j=1}^{d-1} \tilde{U}_j \right)\bmod{1}$. 
    \item Return $(U_1, \ldots, U_{d-1}, U_d)$.
\end{enumerate}

\end{algo}

The correctness of \cref{algo:sampler} rests upon the definition of $U_d$ in Step~\ref{algo:step3}, which is due to the fact that $\tU_d = X \, \oplus \, \bigoplus_{j=1}^{d-1} - \tU_j$ is equivalent to $X = \tU_d \, \oplus \, \bigoplus_{j=1}^{d-1} \tU_j = \bigoplus_{j=1}^d \tU_j$. This follows because the circle group $\T$ is abelian, so that by the cancellation law for groups we have $x = y \oplus z$ if and only if $y = x \oplus -z$ \citep[see, e.g.,][Section 1.1, Proposition 2]{dummitAbstractAlgebra2003}.

An immediate consequence of \cref{prop:basicproperties} is that if a random vector $\bX$ has copula $C^{\bs}_f$, then the marginal vector $\bX_{-j}$ has the independence copula for any $j \in \{1,\ldots,d\}$. That is, any subset of at most $d-1$ of the variables $X_1, \dots, X_d$ are mutually independent. As a corollary, we find that the intersection of $\cC$ and the class of Archimedean copulas is trivial, so no confusion arises from our use of the word ``generator'' and its use for Archimedean generators. 

\begin{coro}

    The only Archimedean copula in $\cC$ is $\Pi$. 

\end{coro}

\begin{proof}
Suppose $\bU \sim C_f^{\bs}$ has an Archimedean copula, so that $C^{\bs}_f(\bu) = \psi\left( \sum_{j=1}^d \psi^{-1}(u_j)\right)$ for some Archimedean generator $\psi$. Such a function has a continuous, decreasing inverse on $[0,1]$ which satisfies $\psi^{-1}(1) =  0$. For any $j \in \{1,\ldots,d\}$, the distribution of the marginal vector $\bU_{-j}$ is given by 
\[
    C^{\bs}_f(u_1,\ldots,u_{j-1},1,u_{j+1},\ldots, u_d) 
    = 
    \psi\left( \sum_{k \neq j} \psi^{-1}(u_k)\right) = \prod_{k \neq j} u_k,
\]
where the second equality follows from Part~\ref{res:fandPi} of \cref{prop:basicproperties}. Thus $\sum_{k \neq j} \psi^{-1}(u_k) = \psi^{-1}\left( \prod_{k \neq j} u_k\right)$, so that $\psi^{-1}$ satisfies Cauchy's logarithmic functional equation, the continuous and decreasing solutions of which are of the form $\psi^{-1}(t) = -c \log(t)$ for $c > 0$ \citep{aczelLecturesFunctionalEquations1966}. Thus $\psi(t) = \exp(-t/c)$, and it follows that $C^{\bs}_f(\bu) = \prod_{j=1}^d u_j$. That is, $C^{\bs}_f = \Pi$.
\end{proof}

Characteristic functions provide insight into the effect of the generator $f$ on the construction of $C_f^{\bs}$ under the modest assumption that $f \in L^p([0,1])$ for some $p > 1$. It is not hard to derive the characteristic function $\varphi_{C^{\bs}_f}$ of $C^{\bs}_f$ in terms of the characteristic function $\varphi_f$ of $f$:

\begin{prop}\label{prop:characteristicfx}
    Let $\bU \sim C^{\bs}_f$. If $f \in L^p([0,1])$ for some $p > 1$, then
    \begin{equation}\label{eq:chfunction}
    \varphi_{\bU}(\bt) 
    =
    \sum_{k = -\infty}^\infty \varphi_f(-2\pi k)  \prod_{j=1}^d \varphi_1(t_j + (-1)^{s_j}2\pi k),
    \end{equation}
    where $\varphi_1(t) := (e^{\imi t} - 1)/\imi t$ is the characteristic function of the $\stdunif$ distribution.
\end{prop}

\begin{rem}
    \cref{eq:chfunction} decomposes the dependence structure of $C_f$ into a superposition of simpler, separable components. The foundational component, corresponding to $k = 0$, is precisely the characteristic function of $\Pi$, and its weight is $\varphi_f(0) = 1$. The remaining terms (i.e., those with $k \neq 0$) induce dependence by coupling the marginals $U_1, \ldots, U_d$ through shared harmonic shifts, with the generator $f$ determining the strength of each coupling mode via its sequence of Fourier coefficients $\{\varphi_f(-2\pi k)\}$. Moreover, it is easy to see that $|\varphi_f(-2\pi k)| < 1$ for $k \neq 0$, which shows that dependence within $C^{\bs}_f$ necessarily arises from the full Fourier expansion of $f$. In particular, no single term $\varphi_f(-2\pi k)  \prod_{j=1}^d \varphi_1(t_j + (-1)^{s_j}2\pi k)$ can by itself reproduce the copula structure unless $k = 0$ and $C^{\bs}_f = \Pi$. Thus, the analytic properties of the generator are closely linked to the nature of the dependence: a smooth generator whose Fourier coefficients decay rapidly produces simple dependence structures dominated by the independence component, while a generator with more slowly decaying coefficients gives rise to a richer and more complex dependence by incorporating more harmonic interactions.
\end{rem}

\begin{proof}[Proof of \cref{prop:characteristicfx}]
    Write $c_k := \varphi_f(-2\pi k) = \int_0^1 e^{-2\pi \imi k x} f(x) \dif x$ and $S_n[f](x) := \sum_{k=-n}^n c_k e^{2 \pi \imi k x}$ for the $n$th partial sum of the Fourier series of $f$. It is a standard result in harmonic analysis that as $n \to \infty$, $S_n[f] \to f$ in $L^p([0,1])$ \citep[see, e.g.,][Chapter 2, Section 1]{katznelsonIntroductionHarmonicAnalysis2004}. Put $G_n(x) := S_n[f](x) -f(x)$. For any fixed $\bt \in \R^d$, H\"older's inequality gives
    \begin{equation*}\label{eq:CFbound}
        \left|\int_{[0,1]^d} e^{\imi \bt^\top \bu} G_n\Big( \bigoplus_{j=1}^d \tu_j \Big) \dif \bu  \right|
        \leq
        \bigg( \int_{[0,1]^d} \Big|G_n\Big( \bigoplus_{j=1}^d \tu_j \Big)\Big|^p \dif \bu \bigg)^{1/p} \bigg( \int_{[0,1]^d} \Big|e^{\imi \bt^\top \bu}\Big|^{\tfrac{p}{p-1}} \dif \bu \bigg)^{\tfrac{p-1}{p}}.
    \end{equation*}
    Since $|e^{\imi x}| = 1$ for any $x \in \R$, the second term on the right hand side in the above display is equal to $1$, leaving only the first term. Given $V, V_1, \ldots, V_d \iid \stdunif$, this term can be rewritten as
    \[
        \E\bigg[ \Big| G_n\Big( \bigoplus_{j=1}^d \tV_j \Big) \Big|^p \bigg]^{1/p}
        =
        \E\left[|G_n(V)|^p\right]^{1/p}
        =
        \bigg( \int_0^1 |G_n(x)|^p \dif x \bigg)^{1/p}
        =
        ||S_n[f] - f||_p
        \too
        0
    \]
    as $n \to \infty$, where the equality in law of $\bigoplus_{j=1}^d \tV_j$ and $V$ is due to \cref{lem:Lebesguepushforward}. We therefore have
    \begin{equation*}\label{eq:Fourierseries}
        \varphi_{\bU}(\bt) 
        =
        \lim_{n \to \infty} \int_{[0,1]^d} e^{\imi \bt^\top \bu} S_n[f] \left( \bigoplus_{j=1}^d \tu_j \right) \dif \bu 
        =
        \lim_{n \to \infty} \sum_{k = -n}^n c_k \prod_{j=1}^d \int_0^1 e^{(\imi t_j + (-1)^{s_j} 2 \pi \imi k) u_j} \dif u_j,
    \end{equation*}
    which is equivalent to \cref{eq:chfunction}.
\end{proof}

\subsubsection{Measures of concordance}

We now examine three measures of concordance for bivariate copulas, beginning with Spearman's rho. First we show that by choosing an appropriate generator $f$ and signature $\bs$, we can attain any value of Spearman's rho within $(-1,1)$, and hence the family $\cC$ is comprehensive. Note that any measure of concordance $\kappa_{(U,V)}$ --- of which Spearman's rho and Kendall's tau are famous examples --- satisfies $\kappa_{(U,V)} = -\kappa_{(U,1-V)}$ \citep{scarsini1984measures}, so by the discussion in \cref{sub:construction} it suffices to treat the $\bs = \bzero$ case in the proofs of the following results.

\begin{prop}\label{prop:spearman}
    
    If $d=2$, then for every $f \in \cF_{[0, 1]}$ and $\bs \in \{0,1\}^2$, the copula $C^{\bs}_f$ has a Spearman's rho of
    \[
        \rho_{C^{\bs}_f} = (-1)^{s_1 + s_2}\left(6\EE\left[X(1-X)\right] - 1\right),
    \]
    where $X \sim f$. In particular, for a fixed signature $\bs$, the range of Spearman's rho is given by
    \[
        \{ \rho_{C^{\bs}_f} : f \in \cF_{[0,1]} \} =
        \begin{cases}
            (-1, 1/2), \quad &s_1 = s_2
            \\
            (-1/2, 1), \quad &s_1 \neq s_2
        \end{cases}.
    \]
    
\end{prop}

\begin{proof}
    According to Theorem 5.1.6 in \citet{nelsenIntroductionCopulas2006},
    \begin{equation} \label{eq:rho}
        \rho_{C_f} = 12 \int_{[0, 1]^2} u_1 u_2 c_f(\bu) \dif \bu - 3.
    \end{equation}
    The integral in \cref{eq:rho} can be written as
    \begin{align}
        \int_0^1 \int_0^1 u_1 u_2 f(u_1 \oplus u_2) \dif u_2 \dif u_1 
        &= \int_0^1 u_1 \left\{ \int_0^{1-u_1} u_2 f(u_1 + u_2) \dif u_2 + \int_{1-u_1}^1 u_2 f(u_1 + u_2 - 1) \dif u_2 \right\} \dif u_1 \nonumber
        \\
        &= \int_0^1 u_1 \left\{ \int_{u_1}^1 (u_2 - u_1) f(u_2) \dif u_2 + \int_0^{u_1} (u_2 - u_1 + 1) f(u_2) \dif u_2 \right\} \dif u_1 \nonumber
        \\
        &= \int_0^1 u_1 \dif u_1 \int_0^1 u_2 f(u_2) \dif u_2 - \int_0^1 u_1^2 \dif u_1 \int_0^1 f(u_2) \dif u_2 + \int_0^1 u_1 F(u_1) \dif u_1 \nonumber
        \\
        &= \frac{1}{2} \int_0^1 u f(u) \dif u - \frac{1}{2} \int_0^1 u^2 f(u) \dif u + \frac{1}{6} \label{eq:rhointegral}
    \end{align}
    where we have used the fact that for a positive integer $p$,
    \begin{equation}\label{eq:darthvader}
        p \int_0^1 u^{p-1} (1 - F(u)) \dif u = \int_0^1 u^p f(u) \dif u.
    \end{equation}
    Inserting \cref{eq:rhointegral} into \cref{eq:rho} yields
    \[
        \rho_{C_f} = 6\int_0^1 u(1-u) f(u) \dif u - 1,
    \]
    which is equivalent to the desired expression.
    
    For the bounds, we note that $0 \leq x(1-x) \leq 1/4$ when $x \in [0,1]$, and that these bounds are attained at $x \in \{0, 1\}$ and $x=1/2$, respectively. Letting the density $f$ put all of its mass arbitrarily close to $0$ or $1$ will yield $\E\left[X(1-X)\right]$ arbitrarily small, whereas if $f$ has all of its mass arbitrarily close to $1/2$, that quantity will approach its maximum of $1/4$. By a linearity argument, allowing $f$ to approach a two-point distribution on $\{0, 1/2\}$ with suitable weights shows that the expected value can attain any value in $(0, 1/4)$. Hence $\rho_{C_f}$ can attain any value between its two bounds.
\end{proof}

We next turn to Kendall's tau.

\begin{prop}\label{prop:kendall}

    If $d = 2$, then for every $f \in \cF_{[0, 1]}$ and $\bs \in \{0,1\}^2$, the copula $C^{\bs}_f$ has a Kendall's tau of
    \begin{equation}\label{eq:tauorig}
        \tau_{C^{\bs}_f} 
        = (-1)^{s_1 + s_2} \left(4\EE\left[X(1-X)\right] + 2\EE\left[|X-X'|\right] - 4\Var(X) - 1\right),
    \end{equation}
    where $X,X' \iid f$. 
\end{prop}

\begin{rem}
    
    The expression $\EE\left[|X-X'|\right]$ is the \emph{mean absolute difference} of the distribution $F$ and is equal to the expected value of \emph{Gini's mean difference}, given by
    \[
        \frac{2}{n(n-1)} \sum_{i=1}^n \sum_{j>i} |X_i - X_j|
    \]
    for a sample $X_1, \ldots, X_n \iid f$ \citep{lomnicki1952standard}. Among several other representations, one can also write \cref{eq:tauorig} as 
    \[
        \tau_{C^{\bs}_f} 
            = 4 (-1)^{s_1 + s_2}\left( \frac{1}{2} \EE\left[|X-X'|\right] - \EE\left[(X - \EE[X])^2\right] - \EE\left[(X - 1/2)^2\right] \right),
    \]
    which is essentially a linear combination of three measures of statistical variability for the distribution $F$: its mean absolute difference, its variance, and its mean squared difference from $1/2$. A sharp upper bound for Kendall's tau is substantially more complicated to develop than that for Spearman's rho; we postpone an analysis until after the proof of \cref{prop:kendall}.

\end{rem}

The proof of \cref{prop:kendall} relies on a simple lemma related to antiderivatives of generators. We write $F^{(-m)}$ for the $(m+1)$th antiderivative of $f$, in the sense that $F^{(-m)}(x) := \int^x_0 F^{(-m+1)}(t) \dif t$, with $F^{(0)}(x) := \int_0^x f(t) \dif t = F(x)$.

\begin{lemm}\label{lemma:momentstoprimitives}
    If $f \in \cF_{[0,1]}$ and $X \sim f$, then
    \[
        F^{(-m)}(1) = \sum_{j=0}^m \frac{(-1)^j}{j! (m-j)!}\E\left[X^j\right], \quad m \in \N.
    \]
\end{lemm}
\begin{proof}
From Cauchy's formula for repeated integration and the binomial theorem, we have that for any $x \in [0,1]$,
\begin{align*}
    F^{(-m)}(x) &= \frac{1}{m!} \int_0^x (x - t)^m  f(t) \dif t \\
    &= \frac{1}{m!} \int_0^x \sum_{j=0}^m \binom{m}{j} (-1)^j  t^j  x^{m-j}  f(t) \dif t \\
    &= \frac{1}{m!} \sum_{j=0}^m \binom{m}{j} (-1)^j  x^{m-j} \int_0^x t^j   f(t) \dif t \\
    &= \sum_{j=0}^m \frac{(-1)^j}{j!  (m-j)!}  x^{m-j}  \E\left[ X^j  \One\{X \leq x\}\right].
\end{align*}
Taking $x = 1$ concludes the proof.\footnote{With substantially more work, the following ``converse'' expression for $\E[X^m]$ in terms of the antiderivatives $\{F^{(-j)}(1)\}$ can also be shown: $\E[X^m] =  \sum_{j=0}^m (-1)^j m^{\underline{j}} F^{(j)}(1)$, where $m^{\underline{j}} =  m(m-1)(m-2) \cdots (m-j+1)$ is the $j$th falling factorial of $m$. See Section 2.3 of \citet{zimmermanCopulasNewTheory2025} for details and further insights into the antiderivatives $\{F^{(-j)}(1)\}$, as well as a generalization of \cref{lemma:momentstoprimitives} to any $m \in [0,\infty)$.}
\end{proof}

\begin{proof}[Proof of \cref{prop:kendall}]
We begin with the well-known identity given in Equation (5.1.12) in \citet{nelsenIntroductionCopulas2006}:
\begin{equation}\label{eq:nelsentau}
    \tau_{C_f} = 1 - 4 \int_{[0,1]^2} \partial_1 C_f(\bu) \cdot \partial_2 C_f(\bu) \dif \bu.
\end{equation}

Let $w(\bu) := F(u_1 \oplus u_2) + \One\{u_1 + u_2 > 1\}$. For $j \in \{1,2\}$, we have
\begin{align}
    \partial_j C_f(\bu) 
    &= \int_0^{u_{-j}} f(u_j \oplus v) \dif v \nonumber \\
    &= \int_{u_j}^{u_1 + u_2} f(v \bmod 1) \dif v \nonumber\\
    &= \left(\int_{u_j}^{u_1 + u_2} f(v) \dif v \right) \One\{u_1 + u_2 \leq 1\} + \left(\int_{u_j}^1 f(v) \dif v + \int_1^{u_1 + u_2} f(v-1) \dif v \right) \One\{u_1 + u_2 > 1\} \nonumber\\
    &= \left(F(u_1 + u_2) - F(u_j) \right) \One\{u_1 + u_2 \leq 1\} \nonumber \\
    &\hspace{1in}+ \left(F(1) - F(u_j) + F(u_1 + u_2 - 1) - F(0) \right) \One\{u_1 + u_2 > 1\} \nonumber\\
    &= w(\bu) - F(u_j). \label{eq:partialjC}
\end{align}
Thus
\begin{align*}
    \partial_1 C(\bu) \cdot \partial_2 C(\bu)
    &= w(\bu)^2 - w(\bu)(F(u_1) + F(u_2)) + F(u_1) F(u_2),
\end{align*}
so by symmetry, the integral in \cref{eq:nelsentau} is equal to
\begin{equation} \label{eq:intpd}
    \int_{[0,1]^2} w(\bu)^2 \dif\bu - 2 \int_{[0,1]^2} w(\bu) F(u_1) \dif\bu + F^{(-1)}(1)^2.
\end{equation}

To handle the three terms in \cref{eq:intpd}, we first note that \cref{lemma:momentstoprimitives} implies the identities
\begin{equation}\label{eq:lemma1identity1}
    F^{(-1)}(1) = 1 - \E[X]
\end{equation} and
\begin{equation}\label{eq:lemma1identity2}
    F^{(-2)}(1) = \frac{1}{2} - \E[X] + \frac{1}{2} \E[X^2].
\end{equation}

After expanding $w(\bu)^2$ in \cref{eq:intpd}, and after several changes of variables to remove wrapped sums, we first find that
\begin{align}
    \int_{[0,1]^2} w(\bu)^2 \dif\bu &= \int_0^1 \int_0^1 F(u_1 \oplus u_2)^2 \dif u_2 \dif u_1 + 2\int_0^1 \int_{1-u_1}^1 F(u_1 \oplus u_2) \dif u_2 \dif u_1 + \int_0^1 \int_{1-u_1}^1 \dif u_2 \dif u_1 \nonumber
    \\
    &= \int_0^1 \int_0^1 F(u_1)^2 \dif u_2 \dif u_1 + 2\int_0^1 \int_0^{u_1} F(u_2) \dif u_2 \dif u_1 + \int_0^1 u \dif u \nonumber
    \\
    &= \int_0^1 F(u)^2 \dif u + 2\int_0^1 F^{(-1)}(u) \dif u + \frac{1}{2}. \label{eq:intwsquared}
\end{align}
Using integration by parts, \cref{eq:lemma1identity1}, and \cref{lemma:momentstoprimitives}, the first integral in \cref{eq:intwsquared} is
\begin{align}
    \int_0^1 F(u)^2 \dif u 
    &= F^{(-1)}(1) - \int_0^1 F^{(-1)}(u) f(u) \dif u
    \nonumber\\
    &= 1 - \E[X] - \E[F^{(-1)}(X)] 
    \nonumber\\
    &= 1 - \E[X] - \left( \E[X \E[\One\{X' \leq X\} \mid X] - \E[X' \One\{X' \leq X\} \mid X]] \right)
    \nonumber\\
    &= 1 - \E[X] - \E[\E[(X - X') \One\{X' \leq X\} \mid X]]
    \nonumber\\
    &= 1 - \E[X] - \frac{1}{2} \E[|X - X'|], \label{eq:intF2}
\end{align}
while the second integral in \cref{eq:intwsquared} is by definition equal to $F^{(-2)}(1)$, which itself is equal to $1/2 - \E[X] + (1/2) \E[X^2]$ by \cref{eq:lemma1identity2}. Putting the pieces together, the first term in \cref{eq:intpd} is
\begin{equation} \label{eq:intw2}
    \int_{[0,1]^2} w(\bu)^2 \dif\bu = -\frac{1}{2} \E[|X - X'|] - 3 \E[X] + \E[X^2] + \frac{5}{2}.
\end{equation}

Via similar changes of variables and \cref{eq:lemma1identity1,eq:lemma1identity2},
\begin{align}
    \int_{[0,1]^2} w(\bu) F(u_1) \dif\bu &= \int_0^1 \int_0^1 F(u_1 \oplus u_2) F(u_1) \dif u_2 \dif u_1 + \int_0^1 \int_{1-u_2}^1 F(u_1) \dif u_1 \dif u_2 \nonumber
    \\
    &= \int_0^1 \int_0^1 F(u_1) F(u_2) \dif u_2 \dif u_1 + \int_0^1 (F^{(-1)}(1) - F^{(-1)}(1 - u)) \dif u \nonumber
    \\
    &= F^{(-1)}(1)^2 + F^{(-1)}(1) - F^{(-2)}(1) \nonumber
    \\
    &= -2\E[X] + \E[X]^2 - \frac{1}{2}\E[X^2] + \frac{3}{2} \label{eq:intwF}
\end{align}

Inserting \cref{eq:lemma1identity1,eq:intwF,eq:intw2} into \cref{eq:intpd} and rearranging, we finally conclude that
\[
    \int_{[0,1]^2} \partial_1 C_f(\bu) \cdot \partial_2 C_f(\bu) \dif \bu = \E[X(X-1)] -\frac{1}{2} \E[|X - X'|] + \Var(X) + \frac{1}{2},
\]
which by \cref{eq:nelsentau} is equivalent to the claimed result.
\end{proof}

We now directly compare the ranges of Spearman's rho and Kendall's tau for copulas of the form $C_f$, an analysis which easily extends to copulas $C^{\bs}_f$ with signatures $\bs \in \{(0,1), (1,0), \bone\}$. As in the case of Spearman's rho, the lower bound of $-1$ for Kendall's tau is attained only by the countermonotonicity copula \citep{embrechts2001modelling}. However, when $f$ puts all of its mass arbitrarily close to $1/2$, we obtain $\tau_{C_f} = 0$, which is an upper bound for generators with a single point mass. Nevertheless, $0$ is not an upper bound in general. For example, if $\{f_n\}$ is a sequence in $\cF_{[0,1]}$ that approaches a $2$-point distribution placing mass $p$ at $a \in [0,1]$ and mass $1-p$ at $b \in [0,1]$, then $\tau_{C_{f_n}}$ approaches a maximum when $p = (a+b)/2$, and maximizing over $0<a<b<1$ yields $\tau_{C_{f_n}} \to 1/8$ when $a = 3/8$ and $b=5/8$. In fact, the tightest upper bound for $\tau_{C_f}$ among all $f \in \cF_{[0,1]}$ is only slightly higher:

\begin{prop}\label{prop:tauupperbound}
    For a fixed signature $\bs$, the range of Kendall's tau is given by
    \[
        \{ \tau_{C^{\bs}_f} : f \in \cF_{[0,1]} \} =
        \begin{cases}
            (-1, 1/6], \quad &s_1 = s_2
            \\
            [-1/6, 1), \quad &s_1 \neq s_2
        \end{cases}.
    \]
\end{prop}

It is interesting to note that, in contrast to Spearman's rho, the upper bound on Kendall's tau is attained by a specific copula, which is explicitly constructed in the following proof.

\begin{proof}[Proof of \cref{prop:tauupperbound}]
    To begin with, one can easily show using integration by parts \citep[see, e.g.,][]{ceroneBoundsGiniMean2005} that
    \[
        \E[|X - X'|] = 4 \E[X F(X)] - 2 \E[X] = 2 \int_0^1 x(2 F(x) - 1) f(x) \dif x.
    \]
    Define the functional
    \begin{equation}\label{eq:Phifunctional}
        \Phi[G] := 8 \int_0^1 x G(x) G'(x) \dif x - 8 \int_0^1 x^2 G'(x) \dif x + 4 \left( \int_0^1 x G'(x) \dif x \right)^2 - 1
    \end{equation}
    on the convex set of continuous and almost everywhere differentiable cdfs $G$ on $[0,1]$, so that $\Phi[F] = \tau_{C_f}$. We maximize $\Phi$ through a variational approach. To determine stationary points, perturb $G$ as $G + \eps h$, where $h$ is any test function such that for sufficiently small $\eps > 0$, $G + \eps h$ remains in the domain of $\Phi$. Such a test function must be non-constant, continuous and almost everywhere differentiable on $[0,1]$; moreover, it must satisfy $h(0) = h(1) = 0$, $h(x) \geq 0$ whenever $G(x) = 0$, and $h(x) \leq 0$ whenever $G(x) = 1$.

    First, using integration by parts, one can show that the Gateaux derivative of the first term in \cref{eq:Phifunctional} is given by
    \[
        \lim_{\epsilon \to 0} \frac{\dif}{\dif \epsilon} 8 \int_0^1 x (G(x) + \eps h(x)) (G'(x) + \eps h'(x)) \dif x = -8 \int_0^1 G(x) h(x) \dif x.
    \]
    Similarly, the second term in \cref{eq:Phifunctional} has Gateaux derivative
    \[
        \lim_{\epsilon \to 0} \frac{\dif}{\dif \epsilon} 8 \int_0^1 x^2 (G'(x) + \eps h'(x)) \dif x = -16 \int_0^1 x h(x) \dif x,
    \]
    and the third term has Gateaux derivative 
    \[
        \lim_{\epsilon \to 0} \frac{\dif}{\dif \epsilon} \left[4 \left( \int_0^1 x (G'(x) + \eps h'(x)) \dif x \right)^2 \right]
        = -8 \E[X_G] \int_0^1 h(x) \dif x,
    \]
    where $X_G \sim G$. Therefore, the Gateaux derivative of $\Phi$ itself at $G$ in the direction $h$ is
    \begin{equation*}\label{eq:Gateaux1}
        \delta \Phi[G; h] := \int_0^1 \left(-8 G(x) + 16 x - 8 \E[X_G] \right) h(x) \dif x.
    \end{equation*}
    
    Consider the candidate cdf $G^*$ defined by 
    \[
        G^*(x) 
        := 
        \begin{cases}
            0, & x \leq 1/4 \\
            2x - 1/2, & 1/4 < x < 3/4 \\
            1, & x \geq 3/4.
        \end{cases}
    \]
    That is, $G^*$ is the cdf of the $\mathrm{Unif}(1/4, 3/4)$ distribution. We claim that $\delta\Phi[G^*; h] \leq 0$ for every admissible test function $h$. Indeed,
    \begin{equation}\label{eq:gateauxatGstar}
        \delta\Phi[G^*; h] 
        =
        \int_0^{1/4} (16x - 4) h(x) \dif x
        + \int_{1/4}^{3/4} (16x - 4 - 8G^*(x)) h(x) \dif x + \int_{3/4}^1 (16x - 12) h(x) \dif x. 
    \end{equation}
    The first integral in \cref{eq:gateauxatGstar} is non-positive because the admissible test function $h$ must satisfy $h(x) \geq 0$ at points $x$ where $G^*(x) = 0$. The second integral vanishes by definition of $G^*$, and the third integral is non-positive because $h(x) \leq 0$ at points $x$ such that $G^*(x) = 1$. It follows that $\delta\Phi[G^*; h] \leq 0$.

    To see that $G^*$ is a global maximizer of $\Phi$, it suffices to verify that $\Phi$ is strictly concave over its domain. Observe that the second-order Gateaux derivative of $\Phi$ at any $G$ and in any admissible direction $h$ is given by
    \begin{align*}
        \delta^2 \Phi[G; h] &:= \lim_{\eps \to 0} \frac{\dif^2}{\dif\eps^2} \Phi(G + \eps h)
        \\
        &= 16 \int_0^1 x h(x) h'(x) \dif x + 8 \bigg( \int_0^1 x h'(x) \dif x \bigg)^2
        \\
        &= -8 \int_0^1 h(x)^2 \dif x + 8 \left( \int_0^1 h(x) \dif x \right)^2,
    \end{align*}
    where in the third equality we have once again used integration by parts. By the Cauchy-Schwarz inequality, 
    \[
        \left( \int_0^1 h(x) \dif x \right)^2 
        \leq 
        \int_0^1 h(x)^2 \dif x,
    \]
    with equality if and only if $h$ is constant almost everywhere. Because the admissible direction $h$ must satisfy $h(0) = h(1) = 0$ and be continuous, we have $\delta^2 \Phi[G; h] \leq 0$, with equality if and only if $h \equiv 0$. That is, $\delta^2 \Phi[G; h] < 0$ for any non-zero admissible direction $h$, and therefore the function $G \mapsto \delta^2 \Phi[G; h]$ is negative definite for all such directions $h$. Thus $\Phi$ is strictly concave over its domain, which establishes that the $\mathrm{Unif}(1/4, 3/4)$ cdf $G^*$ is the \emph{unique} global maximizer of $\Phi$. A routine computation then gives $\Phi[G^*] = 1/6$.
\end{proof}

\begin{rem}
    The proof of \cref{prop:tauupperbound} relies on a representation of $\tau_{C_f}$ as a functional of the cdf associated to the generator $f$. Alternatively, $\tau_{C_f}$ can be written as a functional of the associated probability measure:
    \[
        T[\mu] := 4 \int_{[0,1]} x(1-x) \dif \mu(x) + 2 \int_{[0,1]^2} |x - x'| \dif (\mu \otimes \mu) (x, x') - 4 \int_{[0,1]} \left(x - \int_{[0,1]} x \dif \mu(x) \right)^2 \dif \mu(x) - 1.
    \]
    It is interesting to note that, since $\mu_n \wc \mu$ implies $\mu_n \otimes \mu_n \wc \mu \otimes \mu$, the functional $T$ is continuous in the space of probability measures on $\R$ equipped with the topology of weak convergence. Moreover, absolutely continuous distributions form a dense subset of that topological space.\footnote{Indeed, for any $\eps>0$ and any cdf $G$, partition $[0,1]$ into subintervals of length at most $\eps$, and let $G_\eps$ be the function obtained by linearly interpolating the values of $G$ at the partition points. It is easy to verify that $G_\eps$ has a density and is within L\'evy distance $\eps$ from $G$. Thus, the subset of cdfs with densities is dense with respect to the topology induced by the L\'evy metric, which is well-known to be equivalent to the topology of weak convergence \citep[see, e.g.,][Exercise 14.5]{billingsleyProbabilityMeasure1995}.} Therefore, the global maximizer of $T$ over absolutely continuous probability measures that was identified in \cref{prop:tauupperbound} --- i.e., the $\mathrm{Unif}(1/4, 3/4)$ measure --- is, in fact, a global maximizer of $T$ over \emph{all} probability measures. Furthermore, using a careful smoothing argument, one can show that no purely atomic probability measure can attain the upper bound for Kendall's tau. This result contrasts starkly with the case of Spearman's rho, for which the upper bound behaves in precisely the opposite fashion: while the global maximizer of $T$ is a measure with an associated density in $\cF_{[0,1]}$, the global maximizer of the corresponding functional $\mu \mapsto 6\int_{[0,1]} x(1-x) \dif \mu(x) - 1$ associated with Spearman's rho resides in the space of purely atomic measures, as shown in the proof of \cref{prop:spearman}.
\end{rem}

We now turn to the $[0,1]$-valued functional $\xi_{C} := 6 ||\partial_1 C||_2^2 - 2$, which was introduced by \citet{detteCopulaBasedNonparametricMeasure2013} as a copula-based measure of dependence and is now commonly referred to as the \emph{Dette--Siburg--Stoimenov coefficient}. If $(X,Y)$ is a pair of continuous random variables with copula $C$, then $\xi_{C} = 0$ if and only if $X$ and $Y$ are independent, and $\xi_C = 1$ if and only if $Y$ is completely dependent on $X$ (i.e., $\Prob(Y = g(X)) = 1$ for some deterministic measurable function $g$ --- see \citet{lancasterCorrelationCompleteDependence1963}). This characterization of perfect functional (but not necessarily monotonic) dependence distinguishes this coefficients from other traditional measures of concordance. More recently, \citet{chatterjeeNewCoefficientCorrelation2021} proposed a rank-based sample correlation coefficient $\xi_n(X,Y)$ that has quickly gained considerable popularity, and showed that $\xi_n(X,Y)$ is a consistent estimator of $\xi_{C}$, sparking renewed interest in the latter quantity. 
\begin{prop}\label{prop:chatterjee}
    
    If $d=2$, then for every $f \in \cF_{[0, 1]}$ and $\bs \in \{0,1\}^2$, the copula $C^{\bs}_f$ has a Dette--Siburg--Stoimenov coefficient of
    \begin{equation}\label{eq:chatterjee}
        \xi_{C_f^{\bs}}
        =
        12 \Var(X) - 6\E[|X-X'|] + 1,
    \end{equation}
    where $X,X' \iid f$.
    In particular, the range of the Dette--Siburg--Stoimenov coefficient is given by
    \[
        \{\xi_{C_f^{\bs}}: f \in \cF_{[0,1]}\}  =
        [0,1).
    \]
    
\end{prop}
\begin{proof}
    From \cref{eq:partialjC}, we have
    \[
        ||\partial_1 C_f^{\bs}||_2^2
        =
        \int_{[0,1]^2} (w(\bu) - F(u_1))^2 \dif \bu
        =
        \int_{[0,1]^2} w(\bu)^2 \dif \bu - 2\int_0^1 \int_0^1 w(\bu) F(u_1) \dif \bu + \int_0^1 F(u_1)^2 \dif u_1.
    \]
    Inserting \cref{eq:intw2,eq:intwF,eq:intF2} and rearranging, the expression above reduces to
    \[
        2\Var(X) -\E[|X-X'|] + \frac{1}{2},
    \]
    which implies \cref{eq:chatterjee}.
    
    Concerning the bounds, our family includes the independence copula, which is the only copula $C$ achieving $\xi_C = 0$. As for the upper bound of $1$, if $\{f_n\}$ is a sequence of generators in $\cF_{[0,1]}$ converging weakly to a point mass at any point $x_0 \in [0,1]$, then by \cref{eq:chatterjee} and a continuity argument, $\xi_{C_{f_n}} \to 1$. On the other hand, the upper limit of 1 is not attainable in $\cC$ since it corresponds to perfect functional dependence, which would contradict the existence of a bivariate density.
\end{proof}

\begin{rem}
    The extremal cases shown in the proof of \cref{prop:chatterjee} are highly intuitive. Indeed, suppose that $(U_1, U_2) \sim C_f^{\bs}$. When $f$ is constant on $[0,1]$, \cref{algo:sampler} shows that $U_2 = (-1)^{s_2}(U - \tU_1) \bmod{1}$ with $U \sim \stdunif$, so that adding a uniform shift modulo $1$ randomizes $\tU_1$ completely, thereby rendering $U_1$ and $U_2$ independent. At the other extreme, when $f$ tends to concentrate mass at some $x_0 \in [0,1]$, the shift in \cref{algo:sampler} becomes deterministic in the limit: $U_2$ converges pointwise to $(-1)^{s_2}(x_0 - \tU_1) \bmod{1}$, whence $U_1$ and $U_2$ become functionally dependent, and $C_f^{\bs}$ collapses to a wrapped version of the upper or lower Fr\'echet--Hoeffding bound (depending on $\bs$).
\end{rem}

To conclude our investigations into measures of concordance, we note that \cref{eq:chatterjee}, in combination with \cref{prop:spearman,prop:kendall}, yields the elegant identity
\begin{equation*}\label{eq:chatspearmankendall}
        \xi_{C_f^{\bs}}
        =
        (-1)^{s_1 + s_2}(2 \rho_{C_f^{\bs}} - 3 \tau_{C_f^{\bs}}),
    \end{equation*}
and from the preceding discussion, we obtain the following characterizations of both independence and complete dependence for bivariate copulas in $\cC$:
    \begin{coro}

        Let $f \in \cF_{[0,1]}$ and let $(U_1, U_2) \sim C_f^{\bs}$. Then, $U_1$ and $U_2$ are independent (i.e., $f$ is the $\stdunif$ density) if and only if $2 \rho_{C_f^{\bs}} =  3\tau_{C_f^{\bs}}$. Conversely, let $\{f_n\} \subset \cF_{[0,1]}$ and let $(U_1^{(n)}, U_2^{(n)}) \sim C_{f_n}^{\bs}$. Then, $(U_1^{(n)}, U_2^{(n)})$ converges weakly to $(U_1, (-1)^{s_2}(x_0 - U_1) \bmod 1)$ for some $x_0 \in [0,1]$ (i.e., the distribution implied by $f_n$ converges weakly to the Dirac measure $\delta_{x_0}$) if and only if $2 \rho_{C_{f_n}^{\bs}} - 3\tau_{C_{f_n}^{\bs}} \to (-1)^{s_1 + s_2}$.
    \end{coro}

\subsubsection{Other properties}

We now turn away from measures of concordance to briefly note an interesting algebraic property of $\cC$. \citet{darsowCopulasMarkovProcesses1992} define the \emph{$*$-product} of two bivariate copulas $C_1$ and $C_2$ by
\[
    (C_1 * C_2)(u_1,u_2) 
    :=
    \int_0^1 \partial_1 C_1(v, u_1) \cdot \partial_2 C_2(u_2,v) \dif v,
\]
which is itself a bivariate copula; moreover, the $*$-product induces a semigroup structure on the space of bivariate copulas (i.e., it is associative and $\Pi$ acts as an identity element). It is not hard to show that when $d=2$, $\cC$ is closed under this operation: by Tonelli's theorem, for $f, g \in \cF_{[0,1]}$ and $\br, \bs \in \{0,1\}^2$,
\begin{align*}
    (C^{\br}_f * C^{\bs}_g)(u,v) &= \int_0^u \int_0^v \int_0^1 f( (-1)^{r_1} z \oplus (-1)^{r_2} x) \cdot g( (-1)^{s_1} y \oplus (-1)^{s_2} z) \dif z \dif y \dif x \\
    &= \int_0^u \int_0^v \int_0^1 f( (-1)^{r_2}x \oplus (-1)^{r_1 + s_1 + s_2 + 1} y \oplus (-1)^{r_1 + s_2} z ) g(z) \dif z \dif y \dif x.
\end{align*}
Therefore,
\[
    C^{\br}_f * C^{\bs}_g = C^{\bt}_h,
\]
where $h(x):= \int_0^1 f((-1)^{r_1} x \oplus (-1)^{r_1 + s_2} y) g(y) \dif y$ (i.e., $h$ is the density of $(-1)^{r_1}X \oplus (-1)^{r_1 + s_2 + 1}Y$, where $X \sim f$ and $Y \sim g$ are independent) and $\bt = (r_1 + r_2 \bmod{2}, \, s_1 + s_2 + 1 \bmod{2})$. Thus, $(\cC, *)$ is a subsemigroup of the aforementioned semigroup.

We conclude this section with a brief investigation of tail dependence. In particular, we find that if $\bU \sim C^{\bs}_f$, then no subvector of $\bU$ has positive (upper or lower) tail dependence. Below, we write $\bar{C}_f^{\bs}$ for the survival function of the copula $C_f^{\bs}$, defined as $\bar{C}_f^{\bs}(\bu) := \Prob(\bU > \bu)$.

\begin{prop} \label{prop:tail}
    For every generator $f \in \cF_{[0, 1]}$ and signature $\bs$, we have $C_f^{\bs}(t, \dots, t) = o(t)$ as $t \downarrow 0$ and $\bar{C}_f^{\bs}(t, \ldots, t) = o(1-t)$ as $t \uparrow 1$.
\end{prop}

\begin{proof}
    We first show the lower tail dependence result. Recall that by \cref{prop:basicproperties}, if $d \geq 3$ then the $(d-1)$-marginals of $C_f^{\bs}$ are equal to the $(d-1)$-dimensional independence copula. Thus $C_f^{\bs}(t, \dots, t) \leq t^{d-1}$.
    
    If $d = 2$, suppose without loss of generality that $s_1 = 0$, so that $c_f^{\bs}(u, v) = f(u \oplus v^{1-s_2} (1-v)^{s_2} )$. We may also assume without loss of generality that $t < 1/2$, since we are interested in the limit as $t \downarrow 0$. The form of $C_f^{\bs}(t, t)$ depends on the signature component $s_2$. If $s_2=0$,
    \[
        C_f^{\bs}(t, t) = \int_0^t \int_0^t f(u + v) \dif v \dif u = \int_0^t \big( F(u+t) - F(u) \big) \dif u \leq t F(2t) = o(t)
    \]
    since $F$ is a continuous cdf with $F(0) = 0$. Similarly, if $s_2=1$,
    \begin{align*}
        C_f^{\bs}(t, t)
        &= \int_0^t \int_0^t f(u \oplus 1-v) \dif u \dif v
        \\
        &= \int_0^t \bigg\{ \int_0^u f(u-v) \dif v + \int_u^t f(1+u-v) \dif v \bigg\} \dif u
        \\
        &= \int_0^t \{ F(u) + F(1) - F(1+u-t) \} \dif u
        \\
        &\leq t F(t) + t(F(1) - F(1-t))
        \\
        &= o(t).
    \end{align*}
    Thus, $C_f^{\bs}(t, \dots, t) = o(t)$ as $t \downarrow 0$.
    
    As for upper tail dependence, a simple calculation first shows that if $\bU \sim C_f^{\bs}$, then $\bone - \bU \sim C_{g}^{\bs}$, where $g(x) := f(1-x) \in \cF_{[0,1]}$; equivalently, $\bone - \bU \sim C_f^{\bone - \bs}$ (see \cref{sub:construction} for details). Therefore,
    \[
        \bar{C}_f^{\bs}(t, \ldots, t)
        =
        \Prob\left(1 - U_1 \leq 1-t, \ldots, 1 - U_d \leq 1-t\right)
        =
        C_g^{\bs}(1-t, \ldots, 1-t).
    \]
    Since the above result for lower tail dependence holds for any generator, it holds for the generator $g$ as well. Thus, the right-hand side of the above display is $o(1-t)$ as $1-t \downarrow 0$, as required.
\end{proof}

\section{Examples}\label{sec:examples}

\subsection{Simple examples}\label{sub:simpleexamples}

To begin with, it is obvious from \cref{eq:ourcopula} that by choosing the generator $f_1 := 1$, we obtain the density $c_{f_1} = 1$ of $\Pi$, as shown in the left panel of \cref{fig:cop_pw}. More generally, we can define a sequence $\{f_n\}$ in $\cF_{[0,1]}$ by
\[
    f_n(x) := \sum_{j=1}^{n} \frac{2j-1}{n} \One\{x \in \left[(j-1)/n, j/n \right]\}.
\]
Each $f_n$ is piecewise constant, placing uniform mass within each interval $[(j-1)/n, j/n]$. The center panel of \cref{fig:cop_pw} shows $c_{f_{10}}$ for the bivariate case. Generally, when $n \geq 2$ the piecewise constancy of $f_n$ manifests itself in $c_{f_n}$ as $n$ pairs of isosceles trapezoids, each perpendicular to the $z$-axis and separated from one another by distance $1/n$.

\begin{figure}[ht]
    \centering
    \includegraphics[width=0.32\textwidth]{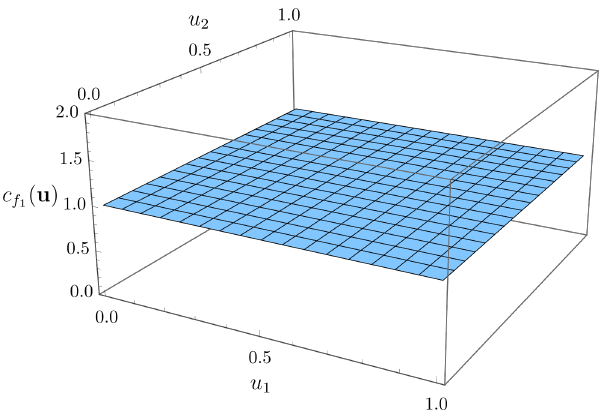}
    \includegraphics[width=0.32\textwidth]{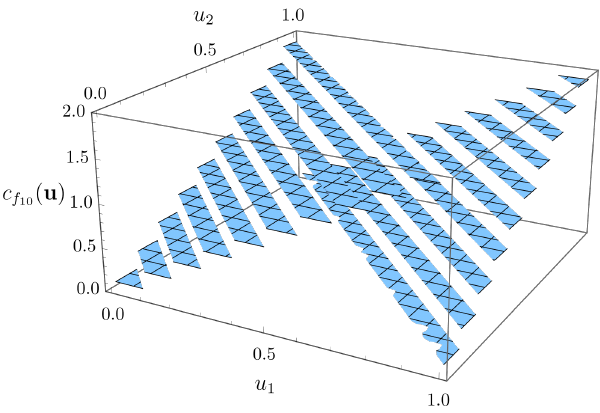}
    \includegraphics[width=0.32\textwidth]{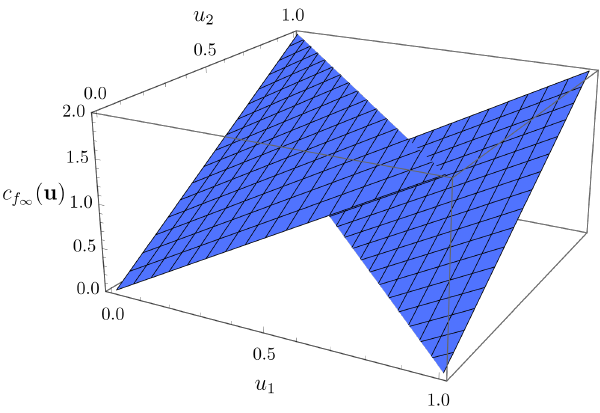}
    \caption{Bivariate copula densities corresponding to the piecewise constant generators $f_1$ and $f_{10}$, and the limiting triangular distribution $f_\infty.$}
    \label{fig:cop_pw}
\end{figure}

Note that $c_{f_n} \to c_{f_\infty}$ pointwise, where $f_\infty(x) := \lim_{n \to \infty} f_n(x) = 2x$, a special case of the triangular distribution density. By Scheff\'{e}'s lemma, the sequence of copulas $C_{f_n}$ converges weakly to $C_{f_\infty}$. The density $c_{f_{\infty}}$ is shown in the right panel of \cref{fig:cop_pw}.

For examples of more ``typical'' generators, we consider the $\text{Beta}(\alpha,\beta)$ distribution for several pairs of parameters $(\alpha, \beta)$. When $\alpha = \beta = 3/2$, the generator is
\[
    f_{\frac{3}{2},\frac{3}{2}}(x) := \frac{8\sqrt{x(1-x)}}{\pi}.
\]
This density has zeros at $0$ and $1$, and so the corresponding copula density $c_{f_{\frac{3}{2},\frac{3}{2}}}$ is zero on hyperplanes of the form $\{\bu \in [0, 1]^d: \sum_{j=1}^m u_j = m\}$, where $m \in \{0,1,\ldots,d\}$. In the bivariate case, this corresponds to zeros along the line $u_1 + u_2 = 1$ as well as at the two isolated points $\bzero$ and $\bone$, as shown in the top left panel of \cref{fig:cop_beta}. On the other hand, with $\alpha = \beta = 1/2$, the generator is
\[
    f_{\frac{1}{2},\frac{1}{2}}(x) := \frac{1}{\pi \sqrt{x(1-x)}}.
\]
The density $c_{f_{\frac{1}{2},\frac{1}{2}}}$ is essentially the reciprocal of the previous situation, with singularities along the same line and isolated points in the bivariate case, as shown in the top center panel of \cref{fig:cop_beta}. Finally, we can introduce asymmetry into the copula density by choosing any $\alpha \neq \beta$. For example, $\alpha = 1/2$ and $\beta = 3/2$ yield the generator
\[
    f_{\frac{1}{2},\frac{3}{2}}(x) := \frac{2}{\pi}\sqrt{\frac{1-x}{x}}.
\]
and $c_{f_{\frac{1}{2},\frac{3}{2}}}$ is shown in the top right panel of \cref{fig:cop_beta}.

\begin{figure}[ht]
    \centering
    \includegraphics[width=0.32\textwidth]{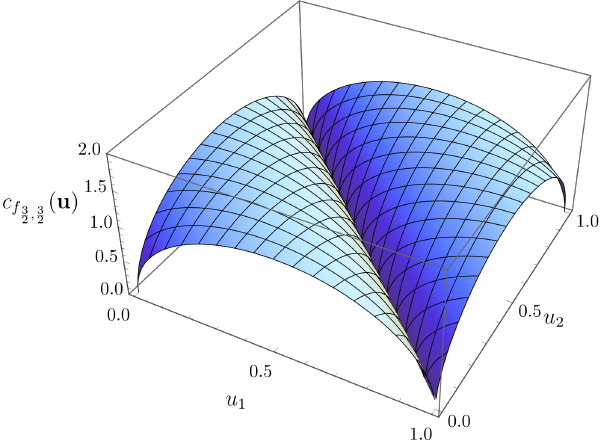}
    \includegraphics[width=0.32\textwidth]{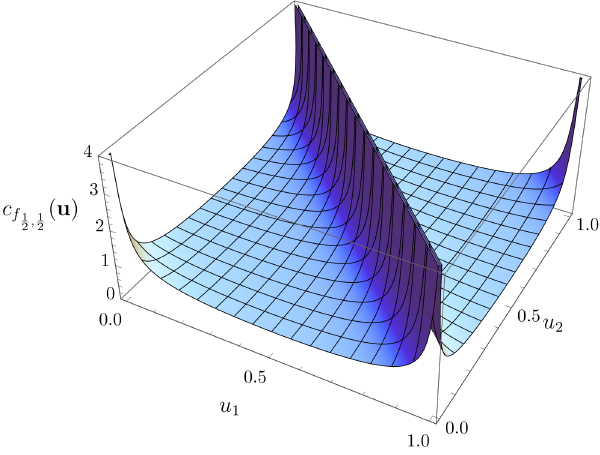}
    \includegraphics[width=0.32\textwidth]{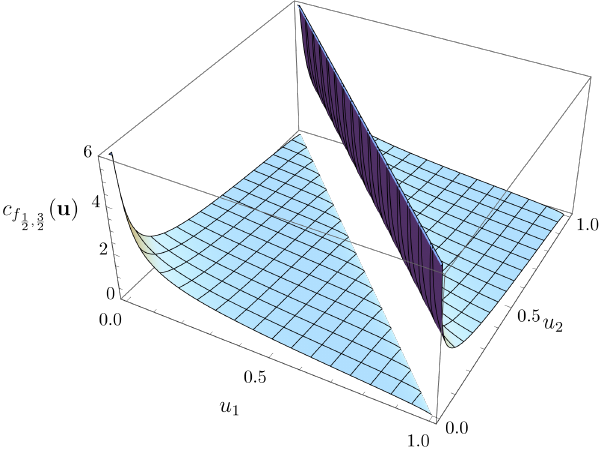}
    \newline
    \includegraphics[width=0.32\textwidth]{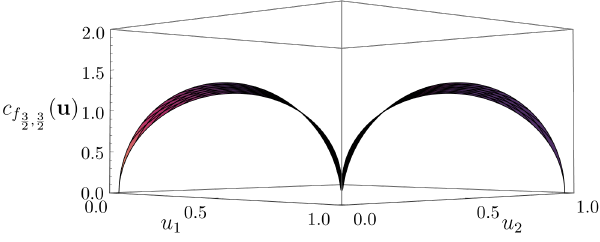}
    \includegraphics[width=0.32\textwidth]{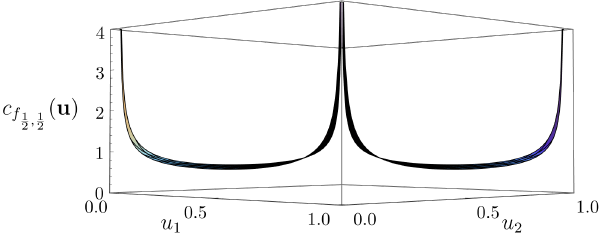}
    \includegraphics[width=0.32\textwidth]{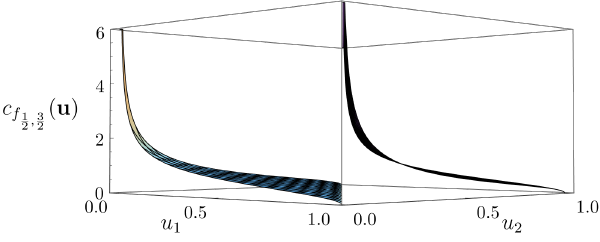}
    \caption{Top row: bivariate copula densities corresponding to the beta generators $f_{\frac{3}{2},\frac{3}{2}}$, $f_{\frac{1}{2},\frac{1}{2}}$, and $f_{\frac{1}{2},\frac{3}{2}}$. Bottom row: different views of the same graphs.}
    \label{fig:cop_beta}
\end{figure}

Next, consider
\[
    f(x) := \frac{5}{2}\frac{\varphi(10x - 5/2)}{\Phi(15/2) - \Phi(-5/2) } + \frac{15}{2} \frac{\varphi(10x - 15/2)}{\Phi(5/2) - \Phi(-15/2) }.
\]
This is the density of an unequal mixture of a $\cTN_{[0,1]}(1/4, 1/100)$ distribution and a $\cTN_{[0,1]}(3/4, 1/100)$ distribution, where $\cTN_{[a,b]}(\mu, \sigma^2)$ refers to the $\cN(\mu, \sigma^2)$ distribution truncated to the interval $[a,b]$. To illustrate the effect of changing the signature $\bs$, we show bivariate copula densities generated by $f$ with all four possible signatures in \cref{fig:cop_mixturenorm}. 

\begin{figure}[ht]
    \centering
    \includegraphics[width=0.32\textwidth]{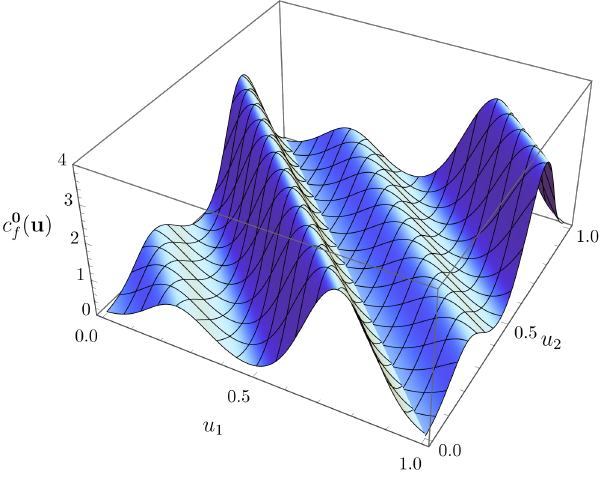}
    \includegraphics[width=0.32\textwidth]{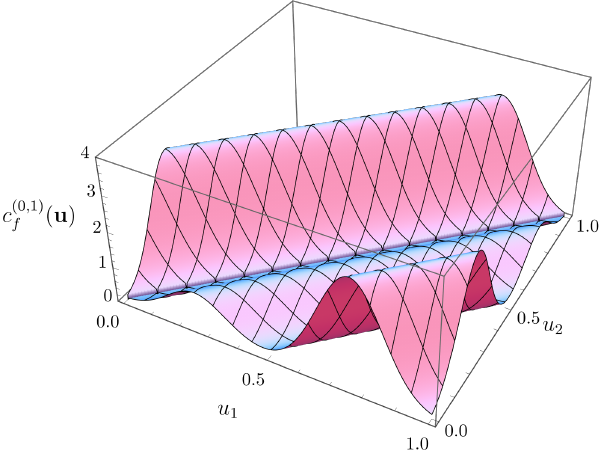}\\
    \includegraphics[width=0.32\textwidth]{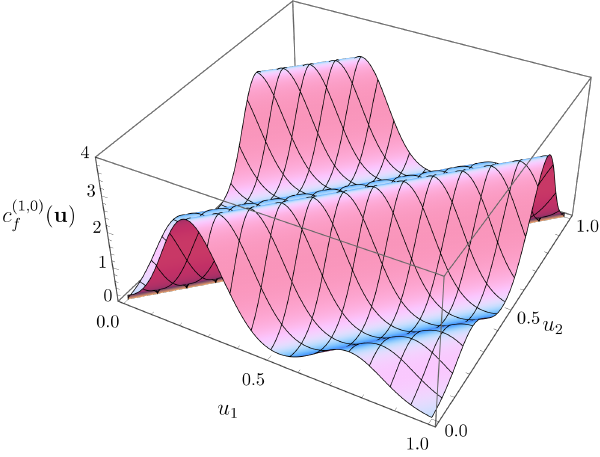}
    \includegraphics[width=0.32\textwidth]{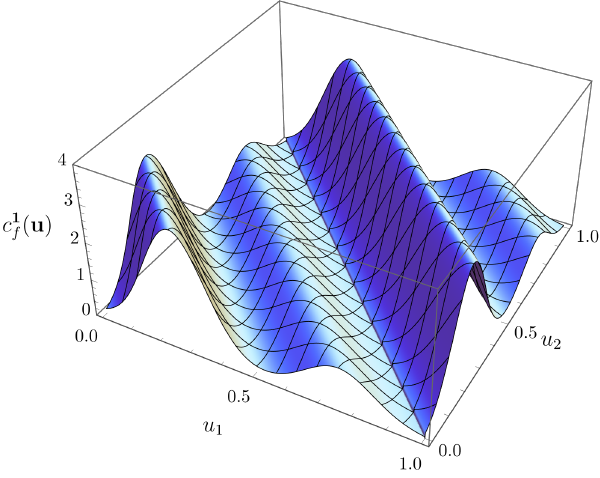}
    \caption{Bivariate copula densities corresponding to the mixture of truncated normals generator with signatures $\bzero$, $(0,1)$, $(1,0)$, and $\bone$.}
    \label{fig:cop_mixturenorm}
\end{figure}

Inspecting the plots in \cref{fig:cop_pw,fig:cop_beta,fig:cop_mixturenorm} reveals additional ``graphical'' properties of our bivariate densities $c_f^{\bs}$ that readily extend to arbitrary dimensions $d \geq 3$. It is apparent from \cref{eq:ourcopula} that $c_f^{\bs}$ is constant and equal to $f(r \bmod{1})$ on any hyperplane of the form $H_r := \{\bu \in [0,1]^d: \sum_{j=1}^d \tu_j = r\}$, where $r \in \R$. Moreover, the graph $\{(\bu, c_f^{\bs}(\bu)): \bu \in (0,1)^d\} \subseteq \R^{d+1}$ admits a natural $2$-dimensional visualization via orthogonal projection onto the subspace spanned by the vector $((-1)^{s_1}, \ldots, (-1)^{s_d}, 0)$ and the vertical axis. Indeed, this projection maps each $\bu$ to a point whose horizontal coordinate is given by $(1/\sqrt{d}) \sum_{j=1}^d \tu_j$ and whose vertical coordinate is $c^{\bs}_f(\bu)$. Thus, every hyperplane $H_r$ is projected to the point $(r/\sqrt{d}, f(r \bmod{1}))$, and as $r$ ranges over $[0, d]$, the image consists of $d$ horizontally compressed copies of the graph of $f$ over $[0,1]$, placed side by side on the intervals $[j/\sqrt{d}, (j+1)/\sqrt{d})$ for each $j \in \{0,1,\ldots,d-1\}$, with possible reflection depending on $\bs$. When $d = 2$, this projection can be visualized by a rotation of the ($3$-dimensional) graph, as illustrated in the second row of \cref{fig:cop_beta}.

\subsection{A poorly behaved univariate density}\label{sub:univariate}

The properties satisfied by a copula are stronger than the characterizing properties of multivariate distribution functions (e.g., monotonicity in each input, boundedness, and so on). For example, it is not hard to see that copulas are $1$-Lipschitz with respect to the supremum norm. One might therefore suspect that densities of absolutely continuous copulas cannot be too poorly-behaved. In practice, one often encounters diverging behavior on the boundary of the support or on sets of measure zero. For example, the density of the Clayton copula with parameter $\theta$ is unbounded along the curve $\sum_{j=1}^d u_j^{-\theta} = d-1$ when $\theta < -1/d$, while the density of the Gaussian copula with correlation matrix $\Rb$ is unbounded in some regions, including around the corners $\bzero$ and $\bone$ when all entries of $\Rb$ are positive. Other examples include the Gumbel and $t$ families \citep[see, e.g.,][]{joe2014dependence}. Such behavior is inconvenient but usually not overly restrictive. However, as we shall demonstrate, the highly regular copulas in $\cC$ can have densities that are poorly behaved on their entire domain.

We now construct a generator $f^* \in \cF_{[0,1]}$ which, in contrast with the examples of \cref{sub:simpleexamples}, is quite pathological. To this end, let $\QQ_1 := \QQ \cap (0, 1)$ and $q_1, q_2,\ldots$ be an arbitrary enumeration of $\QQ_1$. For each $q \in \QQ_1$, we define densities $f_q^+$ and $f_q^-$ that diverge at $1-q$ by
\begin{equation}
	f_q^\pm(u) := \begin{cases}\dfrac{1}{2\sqrt{(\pm(u + q))\bmod{1}}}, & u \neq 1-q \\
	1/2, & u = 1-q\\ \end{cases}
\label{eq:f_q}\end{equation}
and $f_q := (f_q^+ + f_q^-)/2$. By a simple change of variable, we see that
\begin{equation} \label{eq:integralfq}
	\int_0^1 f_q^+(u) \dif u = \int_0^{1-q} \frac{1}{2\sqrt{u + q}} \dif u + \int_{1-q}^1 \frac{1}{2\sqrt{u + q - 1}} \dif u = \int_q^1 \frac{1}{2\sqrt{u}} \dif u + \int_0^q \frac{1}{2\sqrt{u}} \dif u = 1,
\end{equation}
and similarly $\int_0^1 f_q^-(u) \dif u = 1$, so $f_q$ is a proper density. Now, let positive weights $\{w_q\}$ be given such that $\sum_{q \in \QQ_1} w_q = 1$. Valid choices include, for example, $w_{q_m}^{(1)} := 6/(\pi^2 m^2)$ and $w_{q_m}^{(2)} := 2^{-m}$, or indeed, any normalized and positive summable sequence. Define
\[
	h^*(x) := \sum_{q \in \QQ_1} w_q f_q(x).
\]
To provide intuition, we show several normalized partial sums of $h^*$ in \cref{fig:fplot}. Each partial sum $h_m^*(x) := \sum_{j=1}^m w_{q_j} f_{q_j}(x)$ is constructed by choosing for $q_1,\ldots,q_m$ the $m$ rationals evenly spaced within the closed interval $[1/(m+1), m/(m+1)]$. The weights are chosen as $w_{q_j} \propto 1.1^{-\pi(j)}$, $j=1,\ldots, m$ for a random permutation $\pi$ of $(1,\ldots,m)$. 

\begin{figure}[ht]
    \centering
    \includegraphics[width=0.32\textwidth]{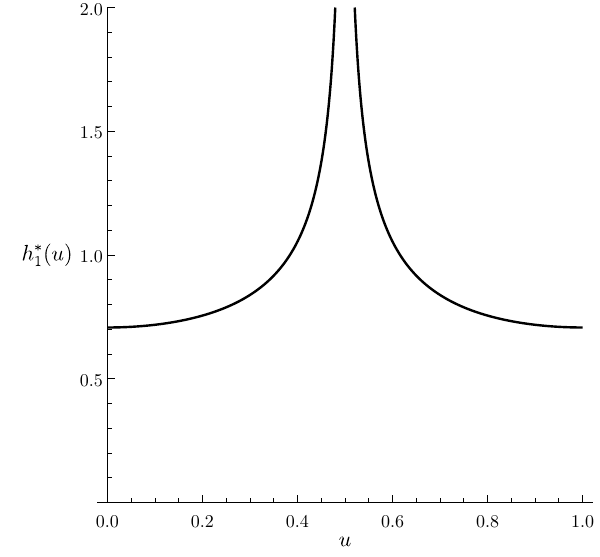}
    \includegraphics[width=0.32\textwidth]{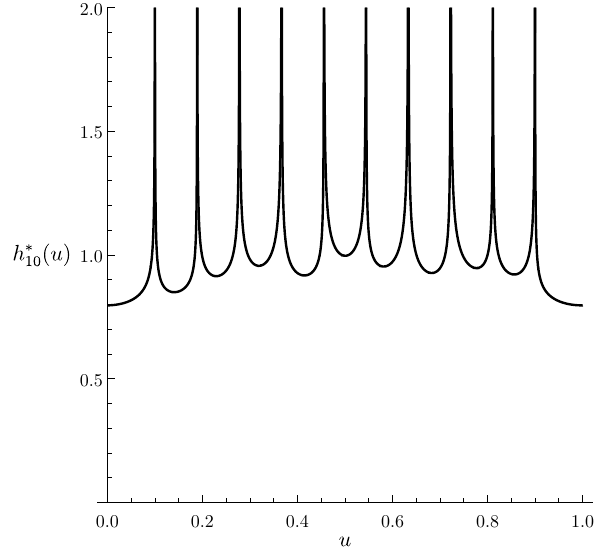}
    \includegraphics[width=0.32\textwidth]{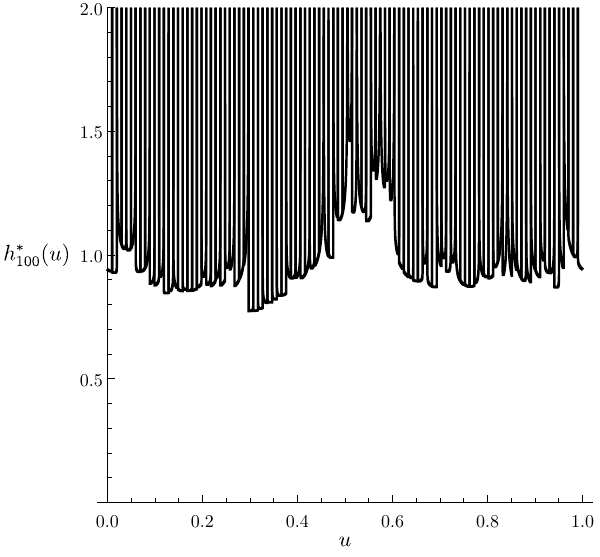}
    \caption{From left to right, the normalized partial sums $h_1^*$, $h_{10}^*$ and $h_{100}^*$.}
    \label{fig:fplot}
\end{figure}
The function $h^*$ is non-negative and measurable, and by Tonelli's theorem it has Lebesgue integral over $[0,1]$ equal to $1$. It is thus a proper density in $\cF_{[0,1]}$, although it is not Riemann integrable (as we shall see shortly). Witness to the irregularity of the associated copula density is the set of singular points of $h^*$, which we investigate here. Consider the set $A := \{x \in [0,1]: h^*(x)= \infty\}$. While $\QQ_1$ is clearly contained in $A$, the latter set is in fact much larger:

\begin{prop}
$A$ is uncountable.
\end{prop}

\begin{proof}
    For $k,m \in \N$, let $A_{k,m} := \{x \in [0,1]: h^*_m(x) > k\}$ and $A_k := \cup_{m \geq 1} A_{k,m} = \{x \in [0,1]: h^*(x) > k\}$, and note that $A = \cap_{k \geq 1} A_k$. First, observe that each function $w_q f_q$ is lower semi-continuous on $[0,1]$, which implies that the superlevel set $A_{k,m}$ is open, and therefore $A_k$ is a union of open sets, and hence also open in $[0,1]$. Next, for any non-empty interval $(a,b) \subseteq [0,1]$, there exists some $q' \in \Q_1 \cap (a,b)$. We have $h^*(x) > w_{q'}f^+_{q'}(x) \to \infty$ as $x \to q'$, and we can thus find $x' \in (a,b)$ such that $h^*(x') > k$; therefore, $x' \in A_k$, and so $(a,b) \cap A_k \neq \emptyset$. Thus $A_k$ is also dense in $[0,1]$.

    We have shown that $A$ is a countable intersection of open dense sets. Therefore, its complement $A^c$ is a countable union of closed nowhere dense sets. If $A$ were countable --- say $A = \{x_1, x_2, \ldots\}$ --- then $[0,1] =  A \cup A^c = \cup_{k \geq 1} (\{x_k\} \cup  A_k^c)$ would also be a countable union of closed nowhere dense sets (i.e., $[0,1]$ would be meagre), but this would contradict the Baire category theorem \citep[see, e.g.,][Theorem 3.6]{levyBasicSetTheory2002}. Hence $A$ is uncountable.
\end{proof}

Constructions of functions similar to $h^*$ are studied in real analysis \citep[see, e.g.,][Section 2.3]{follandRealAnalysisModern2013}. This is, of course, not the only way to construct a univariate density which is unbounded at uncountably many points. Given any uncountable set $D \subset [0,1]$ of measure zero --- of which the Cantor set is the classic example, although unlike $A$, it is nowhere dense in $(0,1)$ --- one can somewhat artificially define a function on $[0,1]$, such as $h(x) := 1/\One\{x \in D^c\}$, which is infinite on $D$ but remains constant at 1 on $D^c$. While any such function is a density in $\cF_{[0,1]}$, it produces only the standard uniform distribution upon integration, and hence its use as a generator simply returns $C^{\bs}_h = \Pi$ for any given signature $\bs$. More generally, one can produce a density in $\cF_{[0,1]}$ which is well-behaved on $D^c$ by choosing, say, some continuous $f \in \cF_{[0,1]}$ and then defining $h_f(x) := f/\One\{x \in D^c\}$. The copulas generated by such $h_f$ will not differ from the kinds discussed in \cref{sub:simpleexamples}; their copula densities $c^{\bs}_{h_f}$ will not exhibit divergent behavior anywhere. In contrast, the density of $C^{\bs}_{h^*}$ is fundamentally irregular, as it \emph{diverges} in any neighborhood: one can find a sequence converging to any point in $[0, 1]^d$ along which $c^{\bs}_{h^*}$ is finite but $c^{\bs}_{h^*} \to \infty$.

To avoid generators taking on the value $\infty$, we define 
\[
    f^*(x) := h^*(x) \One\{x \in A^c\}.
\]
While this modification restores finiteness at every point of $[0,1]$, it does not tame the underlying irregularity of $h^*$ in any open set, which carries on directly to the density $c^{\bs}_{f^*}$. Indeed, since the wrapped sum operation $\bu \mapsto \bigoplus_{j=1}^d \tu_j$ is an open map, the image of every open subset of $[0,1]^d$ under $c^{\bs}_{f^*}$ corresponds to the image of an open set of $[0,1]$ under $f^*$. Hence, while $c^{\bs}_{f^*}$ is everywhere finite, in any nonempty open $O \subset [0, 1]^d$ one can still find $\bu \in O$ with $c^{\bs}_{f^*}(\bu)$ arbitrarily large. In this sense, the density $c^{\bs}_{f^*}$ inherits the category-level pathology of its generator --- a concrete instance where integrability coexists with extreme pointwise divergence. 

Our construction showcases a fundamental decoupling between copulas and their densities. As we will discuss in \cref{sec:inference}, a certain degree of smoothness is required of any copula $C$ for elegant statistical properties to hold, particularly with regard to large sample theory, when observations are drawn from a distribution with copula $C$. We will show in \cref{sub:empiricalcopula} that \emph{all} copulas in $\cC$ --- including $C^{\bs}_{f^*}$ --- satisfy these properties. Thus, for the purposes of statistical inference, it can be restrictive to consider only copulas with ``traditional" densities which are typically continuous and locally bounded on $(0,1)^d$.

\section{Statistical inference}\label{sec:inference}

Throughout this section, $\bX_1, \dots, \bX_n$, $\bX_i := (X_{i1}, \dots, X_{id})$ represent iid observations of a random vector $\bX := (X_1, \dots, X_d)$ with continuous marginal cdfs $F_1, \dots, F_d$ and copula $C^{\bs}_f$. Therefore, $\bU := (F_1(X_1), \dots, F_d(X_d)) \sim C^{\bs}_f$. The generator $f \in \cF_{[0, 1]}$ and signature $\bs$ are arbitrary. The \emph{rank-based pseudo-observations} $\hat\bU_1, \dots, \hat\bU_n$ from the copula are defined by $\hat\bU_i := (\hat U_{i1}, \dots, \hat U_{id})$, where $\hat U_{ij} := \hat F_{nj}(U_{ij})$ and
\[
    \hat F_{nj}(u) := \frac{1}{n} \sum_{i=1}^n \Ind{U_{ij} \leq u}
\]
denotes the $j$th marginal empirical distribution function. Note that $\hat U_{ij}$ can also be defined in terms of the marginal ranks of the original sample $\bX_1, \dots, \bX_n$ as
\[
    \hat{U}_{ij} 
    =
    \frac{R_{ij}}{n},
\]
where $R_{ij}$ denotes the rank of $X_{ij}$ among the observations $X_{1j},\ldots,X_{nj}$ (in practice, one usually replaces the denominator $n$ with $n+1$ to force $\hat{U}_{ij} \in (0,1)$). Thus, the rank-based pseudo-observations can be computed even if the $U_{ij}$ themselves are not observed. Finally, we define the shorthand notation $\hat\bF_n(\bu) := (\hat F_{n1}(u_1), \dots, \hat F_{nd}(u_d))$ so that $\hat\bU_i = \hat\bF_n(\bU_i)$.

\subsection{The empirical copula}\label{sub:empiricalcopula}

The \emph{empirical copula} is the canonical nonparametric copula estimator, and is defined as an empirical distribution based on the pseudo-observations previously defined:
\[
    \hat C_n(\bu) := \frac{1}{n} \sum_{i=1}^n \Ind{\hat\bU_i \leq \bu}.
\]
It is well known that at a point $\bu \in (0, 1)^d$, the empirical copula $\hat C_n(\bu)$ is asymptotically normal if and only if the underlying true copula $C$ has continuous partial derivatives at $\bu$. In fact, \citet{segersAsymptoticsEmpiricalCopula2012} proved that the \emph{empirical copula process} $\sqrt{n}(\hat C_n - C)$ converges in distribution to a Gaussian process in $L^\infty([0, 1]^d)$ if and only if each partial derivative $\partial_j C$ exists and is continuous at points $\bu$ where $0 < u_j < 1$ (see Proposition 3.1 therein). The same condition appears in the study of empirical copula processes indexed by functions of bounded variation \citep[][Theorem 5]{radulovicWeakConvergenceEmpirical2017}. \citet{berghausWeakConvergenceEmpirical2017} show that in addition, if the second-order partial derivatives of $C$ also exist, are continuous, and satisfy a certain bound, then empirical copula process converges in distribution to a tight Gaussian process in weighted metrics, strengthening the result of \citet{segersAsymptoticsEmpiricalCopula2012}; see their Theorem 2.2.

\citet{segersEmpiricalBetaCopula2017} introduce smoothed versions of the empirical copula. They show that asymptotic normality of the so-called \emph{empirical beta copula} and \emph{empirical Bernstein copula} processes holds if the first-order partial derivatives of $C$ exist, are continuous everywhere, and satisfy a local Lipschitz property (see Theorem 3.6 therein). The latter holds, for instance, if the second-order partial derivatives exist and are bounded everywhere. \citet{berghausWeakConvergenceWeighted2018} further show that under the conditions imposed in \citet{berghausWeakConvergenceEmpirical2017}, the empirical beta copula process weakly converges in weighted metrics.

Our next result states that any copula in $\cC$ has \emph{continuous and bounded} first- and second-order partial derivatives, thus guaranteeing that all of the aforementioned results hold when the data is distributed according to such a copula. In particular, they hold for the copula $C^{\bs}_{f^*}$ with nowhere continuous and nowhere bounded density, as constructed in \cref{sub:univariate}. Interestingly, multiple copulas appear smoother than $C^{\bs}_{f^*}$ at first sight, but do not have continuous partial derivatives. For example, the \emph{checkerboard copula} on $[0, 1]^2$ has density $\bu \mapsto 2\One\{\bu \in [0, 1/2]^2 \cup [1/2, 1]^2\}$, which is bounded and almost everywhere continuous. However, continuity of its partial derivatives fails along the lines $\{1/2\} \times [0,1]$ and $[0,1] \times \{1/2\}$, as does the asymptotic normality of the associated empirical copula process.

\begin{prop} \label{prop:pd}
    
    For every $f \in \cF_{[0, 1]}$ and each $j \in \{1, \dots, d\}$, the $j$th partial derivative $\partial_j C^{\bs}_f$ of $C^{\bs}_f$ exists and is continuous on the set $\cV_j := \{\bu \in [0, 1]^d: 0 < u_j < 1\}$. If $d \geq 3$, then for each pair $(j,k) \in \{1, \dots, d\}^2$, the second-order partial derivative $\partial_j \partial_k C^{\bs}_f$ exists and is continuous on $\cV_j \cap \cV_k$, and $|\partial_j \partial_k C^{\bs}_f(\bu)| \leq 1$.
    
\end{prop}

\begin{rem}
It is a standard fact that wherever they exist, the first-order partial derivatives of any copula are bounded by $1$. That is, for any copula $C$, we have $0 \leq \partial_j C \leq 1$ for all $j \in \{1,\ldots,d\}$, which follows directly from the monotonicity and 1-Lipschitz properties. However, the \emph{second-order} partial derivatives of $C$ --- assuming they exist --- are not bounded in general. Consider, for example, a trivariate extension of entry 4.2.9 in the well-known list of one-parameter Archimedean copulas found in Table 4.1 of \citet{nelsenIntroductionCopulas2006}, given by $C_{\theta}(\bu) := u_1 u_2 u_3 \exp\left(-\theta \log(u_1)\log(u_2)\log(u_3)\right)$ where $\theta \in (0,1]$. This is a valid copula when $\theta \leq (3-\sqrt{5})/2$, for then its Archimedean generator satisfies the $3$-monotonicity property \citep[][Theorem 2.2]{mcneilMultivariateArchimedeanCopulas2009}. An easy calculation shows that for fixed $u_2, u_3 > 0$,
\[
    \partial^2_{11} C_\theta(\bu) \propto u_1^{-1 - \theta\log(u_2)\log(u_3)}
\]
which is unbounded around $u_1=0$.
\end{rem}

The key to proving \cref{prop:pd} is the following lemma:

\begin{lemm} \label{lemm:continuity}
    
    Let $J \subseteq \{1, \dots, d\}$ be nonempty. Then for any $f \in \cF_{[0, 1]}$, any $\bs \in \{0,1\}^d$, and any values of $\bu_J \in \R^{|J|}$ and $\bv_{-J} \in \R^{d - |J|}$, the function
    \begin{equation}\label{eq:continuityofintegral}
        \int_{[0, \bu_J]} c^{\bs}_f(\bv) \dif \bv_J = \int_{[0, \bu_J]} f\left( \sum_{j \in J} \tv_j \oplus \sum_{k \notin J} \tv_k \right) \dif \bv_J
    \end{equation}
    is continuous, both in $\bu_J$ and in $\bv_{-J}$.
    
\end{lemm}

\begin{proof}

    Let $K \subset \R^d$ be compact and let $\bsigma := \left((-1)^{s_1}, \ldots, (-1)^{s_d}\right)^\top$. Choose a finite set $M \subset \Z^d$ with $K \subset \cup_{\bbm \in M} (\bbm + [0,1]^d)$, which is possible by compactness. By a change of variable, the fact that $\Psi_{\bsigma, \bzero}(\bv + \bbm) = \Psi_{\bsigma, \bzero}(\bv)$ for all $\bv \in [0,1]^d$ and $\bbm \in \Z^d$ (see \cref{sub:construction}), and \cref{lem:Lebesguepushforward}, we have that
    \[
        \int_K c^{\bs}_f(\bv) \dif \bv
        \leq
        \sum_{\bbm \in M} \int_{\bbm + [0,1]^d} f\left( \Psi_{\bsigma, \bzero}(\bv) \right) \dif \bv
        =
        \sum_{\bbm \in M} \int_{[0,1]^d} f\left( \Psi_{\bsigma, \bzero}(\bv) \right) \dif \bv
        = 
        \sum_{\bbm \in M} \int_0^1 f(x) \dif x 
        =
        |M|
        < \infty.
    \]
    Thus, when extended to a function on $\R^d$, $c^{\bs}_f$ is integrable on arbitrary compact sets.
    Continuity in $\bu_J$ of the function in \cref{eq:continuityofintegral} thus follows from continuity from above and below of Lebesgue measure.
    
    To establish continuity in the direction of $\bv_{-J}$, we show that an infinitesimal change in $\bv_{-J}$ is equivalent to an infinitesimal change in some component of $\bu_J$; we then simply apply the first part of the present result. Indeed, for a small vector $\bdelta \in \R^d$, adding $\bdelta_{-J}$ to $\bv_{-J}$ shifts $\sum_{j=1}^d \tv_j$ to
    \[
        \sum_{j \notin J} (\tv_j + \delta_j^{1-s_j} (1 - \delta_j)^{s_j}) + \sum_{j \in J} \tv_j = \sum_{j \neq j^*} \tv_j + (\tv_{j^*} + \delta),
    \]
    for some arbitrary $j^* \in J$ and for $\delta := \sum_{j \notin J} \delta_j^{1-s_j}(1-\delta_j)^{s_j}$. The integral of interest itself is then shifted to
    \begin{align*}
        \int_{[\bzero, \bu_J]} f\left( \sum_{j \notin J} (\tv_j + \delta_j^{1-s_j}(1-\delta_j)^{s_j}) \oplus \sum_{j \in J} \tv_j \right) \dif \bv_J 
        &= \int_{[\bzero, \bu_J]} f\left( \sum_{j \neq j^*} \tv_j \oplus (\tv_{j^*} + \delta) \right) \dif \bv_J
        \\
        &= \int_{[(-1)^{s_{j^*}} \delta \be_{j^*}, \bu_J + (-1)^{s_{j^*}} \delta \be_{j^*}]} c^{\bs}_f(\bv) \dif \bv_J
        \\
        &\too \int_{[\bzero, \bu_J]} c^{\bs}_f(\bv) \dif \bv_J
    \end{align*}
    as $\delta \to 0$, by the continuity established above.
\end{proof}

\begin{proof}[Proof of \cref{prop:pd}]We consider the first- and second-order partial derivatives separately.
    
    \noindent\textbf{First-order partial derivatives.} Let $j \in \{1,\ldots,d\}$ and $\bu \in \cV_j$. By Tonelli's theorem and the fundamental theorem of calculus,
    \begin{equation}\label{eq:FOPD}
        \partial_j C^{\bs}_f(\bu) 
        = 
        \frac{\partial}{\partial u_j} \int_0^{u_j} \int_{[\bzero, \bu_{-j}]} c^{\bs}_{f}(v_j, \bv_{-j}) \dif \bv_{-j} \dif v_j
        = 
        \int_{[\bzero, \bu_{-j}]} c^{\bs}_{f}(u_j, \bv_{-j}) \dif \bv_{-j},
    \end{equation}
    where we abuse notation and write $c^{\bs}_{f}(u_j, \bv_{-j})$ for the function $c^{\bs}_{f}$ applied to the vector $(u_j, \bv_{-j})$. Note that the application of the fundamental theorem of calculus in the second equality above requires continuity of the right-hand side in $u_j$. This is a consequence of \cref{lemm:continuity}, which also asserts that $\partial_j C^{\bs}_f$ is continuous at $\bu$.

    \noindent\textbf{Second-order partial derivatives.} Now suppose $d \geq 3$. Let $j,k \in \{1,\ldots,d\}$ with $j \neq k$. We treat $\partial_j \partial_k C^{\bs}_f$ and $\partial_j^2 C^{\bs}_f$ separately. Using the expression for $\partial_j C^{\bs}_f$ in \cref{eq:FOPD} and employing a similar abuse of notation,
    \[
        \partial_j \partial_k C^{\bs}_f(\bu) 
        = \frac{\partial}{\partial u_k} \int_0^{u_k} \int_{\left[\bzero, \bu_{-\{j, k\}} \right]} c^{\bs}_{f}(u_j, v_k, \bv_{-\{j, k\}}) \dif \bv_{-\{j, k\}} \dif v_k = \int_{\left[\bzero, \bu_{-\{j, k\}} \right]} c^{\bs}_{f}(u_j, u_k, \bv_{-\{j, k\}}) \dif \bv_{-\{j, k\}},
    \]
    which is continuous in $\bu$. As before, we use Tonelli's theorem, the fundamental theorem of calculus and \cref{lemm:continuity}.
    
    We now evaluate $\partial_j^2 C^{\bs}_f$. Let $\bu \in \cV_j$ and $\delta \in \R$. We have
    \begin{align*}
        \partial_j C^{\bs}_f(\bu + \delta \be_j) &= \int_0^{u_k} \int_{\left[\bzero, \bu_{-\{j, k\}} \right]} c^{\bs}_f(u_j + \delta, v_k, \bv_{-\{j, k\}}) \dif \bv_{-\{j, k\}} \dif v_k
        \\
        &= \int_0^{u_k} \int_{\left[\bzero, \bu_{-\{j, k\}} \right]} c^{\bs}_f(u_j, v_k + (-1)^{s_j + s_k} \delta, \bv_{-\{j, k\}}) \dif \bv_{-\{j, k\}} \dif v_k
        \\
        &= \int_{(-1)^{s_j + s_k} \delta}^{u_k + (-1)^{s_j + s_k} \delta} \int_{\left[\bzero, \bu_{-\{j, k\}} \right]} c^{\bs}_f(u_j, v_k, \bv_{-\{j, k\}}) \dif \bv_{-\{j, k\}} \dif v_k,
    \end{align*}
    where we note that the resulting integral is independent of the choice of $k$. In cases where $(-1)^{s_j + s_k} \delta < 0$ or $u_j + (-1)^{s_j + s_k} \delta > 1$, the function $c^{\bs}_f$ is extended outside $[0,1]^d$ as in the proof of \cref{lemm:continuity}. We deduce that $\partial_j^2 C^{\bs}_f(\bu)$ exists and is given by
    \[
        \lim_{\delta \to 0} \frac{\partial_j C^{\bs}_f(\bu + \delta \be_j) - \partial_j C^{\bs}_f(\bu)}{\delta}
        = (-1)^{s_j + s_k} \int_{\left[\bzero, \bu_{-\{j, k\}} \right]} \left\{c^{\bs}_f(u_j, u_k, \bv_{-\{j, k\}}) - c^{\bs}_f(u_j, 0, \bv_{-\{j, k\}}) \right\} \dif \bv_{-\{j, k\}},
    \]
    which is continuous in $\bu$ by \cref{lemm:continuity}. Once again note that this limit is independent of the choice of $k$. Now, putting $a := (-1)^{s_j}u_j \oplus (-1)^{s_k} u_k$ and $\bsigma := \left((-1)^{s_1}, \ldots, (-1)^{s_d}\right)^\top$, we have $c^{\bs}_f(u_j, u_k, \bv_{-\{j,k\}}) = f\left( \bPsi_{\bsigma_{-\{j,k\}},a}(\bv_{-\{j,k\}})\right)$. \cref{lem:Lebesguepushforward} then yields
    \[
        \int_{[0,1]^{d-2}} c^{\bs}_f(u_j, u_k, \bv_{-\{j,k\}}) \dif \bv_{-\{j,k\}}
        = 
        1,
    \]
    and so the function $\bv_{-\{j,k\}} \mapsto c^{\bs}_f(u_j, u_k, \bv_{-\{j,k\}})$ is a density on $[0,1]^{d-2}$. It follows that $\partial_j \partial_k C^{\bs}_f(\bu) \in [0,1]$ and $\partial_j^2 C^{\bs}_f(\bu) \in [-1,1]$, which completes the proof.
\end{proof}

\begin{rem}
In proving \cref{prop:pd}, we find that no matter how irregular $c^{\bs}_{f}$ is, as long as it is integrated with respect to at least one variable, the resulting antiderivative is continuous. Since differentiating the copula essentially removes one integral, the resulting derivative exists and is continuous as long as at least one integral remains. By iterating this process, one can show that all the partial derivatives of $C^{\bs}_f$ of order up to $d-1$ exist and that they are continuous and bounded on suitable domains.
\end{rem}

\subsection{Data-driven selection of the signature}\label{sub:sigselection}

We now illustrate how the signature $\bs$ can be learned consistently from data in a fully nonparametric and $f$-agnostic fashion. As noted in \cref{sub:construction}, in practice we can implicitly resolve the identifiability issue of $\bs$ by imposing $s_1 = 0$.

Since the random vector $\bU$ is distributed according to $C^{\bs}_f$, according to \cref{prop:basicproperties}, for any $\bt := (t_1, \dots, t_d) \in \{0, 1\}^d$ we have
\begin{equation}\label{eq:Ytdistribution}
    Y^{\bt} := \bigoplus_{j=1}^d U_j^{1-t_j} (1 - U_j)^{t_j} \sim
    \begin{cases}
        f, \quad &\bt \in \{\bs, \bone - \bs\}
        \\
        \stdunif, \quad &\bt \notin \{\bs, \bone - \bs\}
    \end{cases}.
\end{equation}
\cref{eq:Ytdistribution} then allows us to identify the signature as the unique $\bt$ for which $Y^{\bt}$ is \emph{not} uniformly distributed on $(0,1)$. This suggests a simple yet robust and fully nonparametric method to estimate the signature from the rank-based pseudo-observations $\hat\bU_1, \ldots, \hat\bU_n$: for every possible signature $\bt$, first compute the univariate pseudo-observations $\hat Y_1^{\bt}, \dots, \hat Y_n^{\bt}$ of the wrapped sum $Y^{\bt}$, where
\[
    \hat Y_i^{\bt} := \bigoplus_{j=1}^d \hat U_{ij}^{1-t_j} (1 - \hat U_{ij})^{t_j}, \quad i \in \{1,\ldots,n\},
\]
and then identify the signature $\bt$ for which the sample $\hat Y_1^{\bt}, \dots, \hat Y_n^{\bt}$ deviates most from the uniform distribution. We propose to measure this deviation using either a Kolmogorov--Smirnov or Cram\'er--von Mises distance, leading to the estimators
\[
    \hat\bs^{\rm{KS}}_n := \mathop{\argmax}_{\bt \in \{0, 1\}^d} \sup_{u \in [0, 1]} \big| \hat F_{n}^{\bt}(u) - u \big|
\]
and
\[
    \hat\bs^{\rm{CvM}}_n := \mathop{\argmax}_{\bt \in \{0, 1\}^d} \int_0^1 (\hat F_n^{\bt}(u) - u)^2 \dif u,
\]
respectively, where $\hat F_n^{\bt}$ denotes the empirical distribution function of the sample $\hat Y_1^{\bt}, \dots, \hat Y_n^{\bt}$. In practice, these estimators can be computed explicitly by using the known formulas
\[
    \sup_{u \in [0, 1]} \big| \hat F_{n}^{\bt}(u) - u \big| = \max_{i=1,\ldots,n} \max\left\{ \left| \frac{i-1}{n} - \hat Y_{(i)}^{\bt} \right|, \left| \frac{i}{n} - \hat Y_{(i)}^{\bt} \right| \right\}
\]
and
\[
    \int_0^1 \big(\hat F_n^{\bt}(u) - u\big)^2 \dif u = \frac{1}{n} \sum_{i=1}^n \left( \hat Y^{\bt}_{(i)} - \frac{i}{n} + \frac{1}{2n} \right)^2 + \frac{1}{12n^2}
\]
for the Kolmogorov--Smirnov and Cram\'er--von Mises test statistics respectively, where $\hat Y_{(i)}^{\bt}$ denotes the $i$th order statistic of the sample $(\hat Y_1^{\bt}, \dots, \hat Y_n^{\bt})$ \citep[see, e.g.,][Chapter 7.3]{davisonStatisticalModels2003}.

\begin{prop} \label{prop:signature}
    As long as $f$ is not the $\stdunif$ density, $\hat\bs^{\rm{KS}}_n$ and $\hat\bs^{\rm{CvM}}_n$ are consistent estimators of $\bs$.
\end{prop}

\begin{proof}
    For every $\bt \in \{0,1\}^d$ and $u \in [0,1]$, let $A_u^{\bt} = \{\bu \in [0,1]^d: \bigoplus_{j=1}^d u_j^{1-t_j} (1-u_j)^{t_j} \leq u\}$. Define the empirical copula measure, empirical measure and copula measure by
    \[
        \hat P_n(A) := \frac{1}{n} \sum_{i=1}^n \delta_{\hat \bU_i}(A), 
        \quad 
        P_n(A) := \frac{1}{n} \sum_{i=1}^n \delta_{\bU_i}(A), \quad 
        \text{and} \quad
        P(A) := \Prob(\bU \in A),
    \]
    respectively. Notice that $\hat P_n(A_u^{\bt}) = \hat F_n^{\bt}(u)$, and similarly
    \[
        P(A_u^{\bt}) = F^{\bt}(u) :=
        \begin{cases}
            F(u), \quad &\bt \in \{\bs, \bone - \bs\}
            \\
            u, \quad &\bt \notin \{\bs, \bone - \bs\}
        \end{cases},
    \]
    which is the cdf of $Y^{\bt}$ by \cref{prop:basicproperties}. Therefore,
    \[
        |\hat F_n^{\bt}(u) - F^{\bt}(u)| = |\hat P_n(A_u^{\bt}) - P(A_u^{\bt})| = |P_n(\hat A_u^{\bt}) - P(A_u^{\bt})|,
    \]
    where we define the randomly perturbed set
    \[
        \hat A_u^{\bt} := \{\bu \in [0,1]^d: \hat\bF_n(\bu) \in A_u^{\bt}\}.
    \]
    Thus,
    \begin{equation} \label{eq:absdiff}
        |\hat F_n^{\bt}(u) - F^{\bt}(u)| \leq |P_n(\hat A_u^{\bt}) - P_n(A_u^{\bt})| + |P_n(A_u^{\bt}) - P(A_u^{\bt})|.
    \end{equation}

    In order to bound the first term on the left hand side of \cref{eq:absdiff}, we analyze the symmetric difference $A_u^{\bt} \bigtriangleup \hat A_u^{\bt}$. A point $\bu$ can belong to this set either because $\bu \in A_u^{\bt}$ and the transformation by $\hat\bF_n$ makes it exit $A_u^{\bt}$, or because $\bu \notin A_u^{\bt}$ and the transformation by $\hat\bF_n$ makes it enter $A_u^{\bt}$. Neither scenario is possible unless $\bu$ is within Euclidean distance $\epsilon_n := \sup_{\bv \in [0,1]^d} \|\hat\bF_n(\bv) - \bv\|_2$ from the boundary of $A_u^{\bt}$. Since $\partial A_u^{\bt} = \{\bu \in [0,1]^d: \bigoplus_{j=1}^d u_j^{1-t_j} (1 - u_j)^{t_j} \in \{0, u\}\}$, we thus have
    \begin{align*}
        |P_n(\hat A_u^{\bt}) - P_n(A_u^{\bt})|
        &\leq P_n(A_u^{\bt} \bigtriangleup \hat A_u^{\bt})
        \\
        &\leq P_n\Big( \Big\{ \bu \in [0,1]^d: \bigoplus_{j=1}^d u_j^{1-t_j} (1 - u_j)^{t_j} \in [0, \epsilon_n] \cup [u-\epsilon_n, u+\epsilon_n] \cup [1-\epsilon_n, 1] \Big\} \Big)
        \\
        &= P_n^{\bt}([0, \epsilon_n] \cup [u-\epsilon_n, u+\epsilon_n] \cup [1-\epsilon_n, 1]),
    \end{align*}
    where $P_n^{\bt}$ is the empirical measure of the (unobservable) independent copies
    \begin{equation*}
        Y_i^{\bt} := \bigoplus_{j=1}^d U_{ij}^{1-t_j} (1 - U_{ij})^{t_j}, \quad i \in \{1,\ldots,n\}
    \end{equation*}
    of $Y^{\bt}$. By the Glivenko--Cantelli theorem, this empirical measure converges in probability to the distribution of $Y^{\bt}$ uniformly over all intervals of $[0,1]$. Therefore,
    \[
        P_n^{\bt}([0, \epsilon_n] \cup [u-\epsilon_n, u+\epsilon_n] \cup [1-\epsilon_n, 1]) = F^{\bt}(\epsilon_n) + (F^{\bt}(u+\epsilon_n) - F^{\bt}(u-\epsilon_n)) + (1 - F^{\bt}(1-\epsilon_n)) + \op(1).
    \]
    By the multivariate Glivenko--Cantelli theorem, $\epsilon_n \to 0$ in probability. By continuity of $F^{\bt}$ --- which is either $F$ itself or the identity on $[0,1]$ --- we conclude that the first term on the right in \cref{eq:absdiff} vanishes in probability uniformly in $u$. The second term can be written as
    \[
        |P_n^{\bt}([0,u]) - F^{\bt}(u)| = \op(1)
    \]
    uniformly in $u$, again by the standard Glivenko--Cantelli theorem. We have therefore established that for each $\bt \in \{0,1\}^d$, as $n \to \infty$,
    \begin{equation} \label{eq:GC}
        \sup_{u \in [0,1]} |\hat F_n^{\bt}(u) - F^{\bt}(u)| \too 0
    \end{equation}
    in probability; that is, $\hat F_n^{\bt} \to F^{\bt}$ uniformly in probability.

    To conclude, note that both signature estimators are of the form
    \[
        \mathop{\argmax}_{\bt \in \{0, 1\}^d} Q(\hat F_n^{\bt}),
    \]
    where the functional $Q$ mapping the set of distribution functions on $[0,1]$ to $[0, \infty)$ is defined by either
    \[
        Q(H) := \sup_{u \in [0,1]} |H(u) - u| \quad \text{or} \quad Q(H) := \int_0^1 (H(u) - u)^2 \dif u.
    \]
    Since in both cases the functional $Q$ has a unique zero at the identity function and is continuous in the topology of uniform convergence, the continuous mapping theorem ensures that
    \[
        Q(\hat F_n^{\bt}) \too
        \begin{cases}
            Q(F), \quad &\bt \in \{\bs, \bone - \bs\}
            \\
            0, \quad &\bt \notin \{\bs, \bone - \bs\}
        \end{cases}.
    \]
    in probability. Because $F$ is assumed not to be the identity function on $[0,1]$, we conclude that as $n \to \infty$,
    \[
        \Prob(\hat\bs_n \in \{\bs, \bone-\bs\}) \too 1
    \]
    where $\hat\bs_n$ stands for either $\hat\bs^{\rm{KS}}_n$ or $\hat\bs^{\rm{CvM}}_n$.
\end{proof}

\begin{rem}
Our proof reveals that \emph{any} estimator of the signature based on maximizing a functional of $\hat F_n^{\bt}$ that is continuous and uniquely minimized at the identity function is consistent. For example, one can easily implement an estimator based on the Anderson--Darling test statistic which follows the same recipe. We leave a thorough comparison of these tests for potential future work.

A natural alternative approach to signature selection exploits the partial exchangeability properties implied by the signature. Specifically, if $\bU \sim C_f^{\bs}$ and $s_j = s_k$ for some $j \neq k$, then swapping $U_j$ and $U_k$ leaves the distribution unchanged; that is, $(U_1, \ldots, U_{j-1}, U_k, U_{j+1}, \ldots, U_{k-1}, U_j, U_{k+1}, \ldots, U_d) \sim C_f^{\bs}$ as well. This suggests a selection procedure based on testing whether such pairs are exchangeable using two-sample versions of the Kolmogorov--Smirnov or Cram\'er--von Mises tests, wherein one could, in principle, select the signature based on the resulting $p$-values. Since there are only $d(d-1)/2$ such swaps, one might expect this method to be more scalable than the uniform comparison method described above, which exhaustively examines all $2^{d-1}$ possible signatures. However, in our experiments, the need to compute the empirical copula for each candidate swap created a significant computational bottleneck, making this method far slower than our exhaustive method for all dimensions $d \leq 10$ tested. More critically, its accuracy in recovering the true signature was consistently and substantially worse than that of the uniform comparison method.
\end{rem}

For the remainder of \cref{sec:inference}, inference for the generator $f$ is discussed under the assumption that the true signature $\bs$ is known or has been learned.

\subsection{Parametric estimation of the generator}\label{sub:parametric}

One appealing feature of the class $\cC$ is that likelihood inference is straightforward for any subclass of copulas whose generators arise from a parametric family, assuming the signature is known. This contrasts with the case for general parametric copulas, where the computational cost of exact inference typically scales poorly with dimension, mainly because computing score functions is difficult \citep{hofertLikelihoodInferenceArchimedean2012}. To be more specific, suppose $\bU_1, \bU_2, \ldots, \bU_n \iid C^{\bs}_{f_\theta}$, where $f_\btheta$ belongs to some parametric family $\cF_\Theta := \{f_\btheta: \btheta \in \Theta\} \subset \cF_{[0, 1]}$. Writing $Y_i = \bigoplus_{j=1}^d \tU_{ij}$, the log-likelihood function of $\btheta$ is given by
\begin{equation}\label{eq:loglikelihood}
    \ell(\btheta \mid \bU_1, \bU_2, \ldots, \bU_n) 
    = 
    \sum_{i=1}^n \log{c^{\bs}_{f_\btheta}(\bU_i)} 
    =
    \sum_{i=1}^n \log{f_\btheta(Y_i)},
\end{equation}
which is simply the standard log-likelihood based on the iid observations $Y_1, \ldots, Y_n$ from $f_\btheta$, which follows from \cref{prop:basicproperties}. It follows immediately that the maximum likelihood estimator (MLE) $\hat{\btheta}^{(n)}_\mathrm{MLE}$ of $\btheta$ based on $\bU_1, \bU_2, \ldots, \bU_n$ as observations from $C_{f_\btheta}^{\bs}$ coincides with the MLE of $\btheta$ based on $Y_1, \ldots, Y_n$ as observations from $f_\btheta$.

In general, there is a one-to-one correspondence between statistics based on $\bU_1, \dots, \bU_n$ and statistics based on the wrapped sums $Y_1, \ldots, Y_n$, due to the facts that $Y_1, \ldots, Y_n$ behave as iid observations from $f_\btheta$ and that for a fixed signature $\bs$, the generator $f$ uniquely determines $C_f^{\bs}$ by \cref{prop:Ccharacteriation}. This correspondence naturally preserves classical properties such as sufficiency, minimal sufficiency, and completeness, as well as asymptotic normality and efficiency of the MLE. Thus, for example,
\[
    \sqrt{n}\big(\hat{\btheta}^{(n)}_\mathrm{MLE} - \btheta_0 \big) \stackrel{d}{\longrightarrow} \cN\left(\bzero, \boldsymbol{\pazocal{I}}^{-1}(\btheta_0)\right)
\]
under suitable regularity conditions on the statistical model $\cF_\Theta$ (such as identifiability, smoothness, and finiteness of the Fisher information matrix $\boldsymbol{\pazocal{I}}(\btheta_0)$), where $\btheta_0$ is the true data-generating parameter. In a similar vein, hypothesis tests and confidence intervals based on the univariate observations $Y_1, \dots, Y_n$ carry directly over to those based on the multivariate copula observations, as does Bayesian inference on $\btheta$.

In practice, however, the margins are rarely known \emph{a priori}, so two broad strategies arise depending on whether one is willing to model them parametrically. When a fully parametric model is adopted for the margins as well, a natural two–step procedure is the \emph{inference functions for margins} approach of \citet{joe1996estimation}. In the first step, the marginal parameters $\bpsi_1, \ldots, \bpsi_d$ are estimated separately by maximizing the univariate log-likelihoods, yielding the marginal MLEs $\hat{\bpsi}_1, \ldots, \hat{\bpsi}_d$, and the plug-in pseudo-observations $\breve{U}_{ij} := F_{h, \hat{\bpsi}_j}(X_{ij})$ are formed using the fitted marginal cdfs $F_{h, \hat{\bpsi}_j}$. In the second step, \cref{eq:loglikelihood} is maximized --- with $\breve{\bU}_i := (\breve{U}_{i1}, \ldots, \breve{U}_{id})$ in place of $\bU_i$ --- to obtain the copula parameter estimator $\hat{\btheta}^{(n)}_\mathrm{IFM}$. Under standard regularity conditions \citep{joe1996estimation,joe2014dependence} the estimator $\hat{\btheta}^{(n)}_\mathrm{IFM}$ is $\sqrt{n}$-consistent and asymptotically normal. 

Alternatively, if one prefers to remain agnostic about the marginal laws, the classical alternative is to replace $\bU_1, \ldots, \bU_n$ with the rank-based pseudo-observations $\hat{\bU}_1, \ldots, \hat{\bU}_n$. Plugging these into \cref{eq:loglikelihood} yields the pseudo-log-likelihood, whose maximizer $\hat{\btheta}^{(n)}_\mathrm{PL}$ remains $\sqrt{n}$-consistent and asymptotically normal \citep{genestSemiparametricEstimationProcedure1995}. Because it relies only on ranks, the estimator $\hat{\btheta}^{(n)}_\mathrm{PL}$ remains consistent for the copula parameter under any continuous marginals, making it a robust alternative when one is unwilling or unable to model the univariate margins.

\subsection{Nonparametric estimation of the generator}\label{sub:nonpara}

\cref{sub:parametric} demonstrates a significant advantage of our framework: in parametric inference, inference for $C_f^{\bs}$ reduces to a more tractable, one-dimensional problem involving the generator $f$. This advantage also carries over to nonparametric estimation. In a scenario in which the marginal distributions of the data are known \emph{a priori}, we can compute the ``true'' wrapped sum $Y_i$ for each observation $\bU_i$. The task of estimating $f$ nonparametrically then becomes a standard univariate kernel density estimation procedure based on the iid sample $Y_1, \ldots, Y_n$. In this case, classical asymptotic results for kernel density estimation apply directly, such as establishing the optimal bandwidth and the asymptotic normality of the estimator \citep[see, e.g.,][]{wandKernelSmoothing1994}.

In the more common (and substantially more challenging) situation, the margins are unknown. In this case, we must rely on the rank-based pseudo-observations $\hat Y_1, \dots, \hat Y_n$, which we recall are defined as
\begin{equation} \label{eq:Yihat}
    \hat Y_i := \bigoplus_{j=1}^d \hat U_{ij}^{1-s_j} (1 - \hat U_{ij})^{s_j}, \quad i \in \{1,\ldots,n\}.
\end{equation}
Given a kernel $K$ and some bandwidth $h_n > 0$, the fully nonparametric kernel density estimator (KDE) of $f$ based on this sample is
\[
    \hat f_n(x) := \frac{1}{n h_n} \sum_{i=1}^n K\Big( \frac{x - \hat Y_i}{h_n} \Big), \quad x \in (0, 1).
\]

In order to analyze the asymptotic behavior of $\hat f_n$, we note that
\begin{equation} \label{eq:kde}
    \hat f_n(x) = \frac{1}{n h_n} \sum_{i=1}^n K\Big( \frac{x - \hat F_n^{-1}(i/n)}{h_n} \Big),
\end{equation}
where
\[
    \hat F_n(u) := \frac{1}{n} \sum_{i=1}^n \One\{\hat Y_i \leq u\}, \quad u \in [0, 1],
\]
is the empirical distribution function based on the pseudo-observations and $\hat F_n^{-1}(y) := \inf\{u \in [0, 1]: \hat F_n(u) > y\}$ is its generalized inverse; by convention, we set $\hat F_n^{-1}(1) = 1$.

Analyzing the asymptotic behavior of \cref{eq:kde} requires results from the theory of empirical processes. The key difficulty is that the pseudo-observations $\hat Y_1, \ldots, \hat Y_n$ are not independent, since each $\hat Y_i$ depends on the full set of empirical marginal distributions of the original sample. The standard asymptotic analysis of the kernel density estimator, which relies on the weak law of large numbers and the Lindeberg--Feller central limit theorem, therefore does not apply. We state asymptotic properties of the estimator under a high-level conjecture concerning the asymptotic distribution of $\hat F_n$, and make some remarks on the state of this conjecture below.

\begin{conj} \label{conj:ecdf}
    As $n \to \infty$, the empirical process $B_n := \sqrt{n} (\hat F_n - F)$ converges weakly in $L^\infty([0, 1])$ to $B := W \circ F + V$, where $W$ is a standard Wiener process and $V$ is a stochastic process on $[0,1]$ whose sample paths are almost surely locally Lipschitz continuous at $x$. Furthermore, on a possibly enriched probability space containing the data as well as $B$, there exists $\alpha>0$ such that $\|B_n - B\|_\infty = \Op(n^{-\frac{1+\alpha}{10}})$.
\end{conj}

\begin{rem}
    In addition to the uniform asymptotic normality of the empirical process $B_n$ --- a standard property in asymptotic statistics --- \cref{conj:ecdf} asserts both a certain structure for the limiting Gaussian process, and a bound on the rate of convergence. We provide some comments on these two assertions:
    \begin{enumerate}
        
        \item The assumption on $V$ can be reformulated as the existence of a random variable $\Delta \in (0, \min\{x, 1-x\}]$ such that $V(y) - V(x - \Delta) = \int_{x-\Delta}^y V'(z) \dif z$ for every $y \in [x - \Delta, x + \Delta]$, where $V'$ is a stochastic process on $[0,1]$ whose sample paths are almost surely bounded \citep[see, e.g.,][Section 5.8, Theorem 4]{evansPartialDifferentialEquations2010}.

        Note that we perform stochastic calculus (by applying the stochastic integration by parts formula and the It\^{o} isometry) only on $W$, not on $V$. We therefore do not require that $V$ be a (semi-)martingale, nor that it be adapted to the natural filtration of $W$.

        \item Suppose the marginal distributions were known, so that instead of the pseudo-observations $\hat Y_i$ in \cref{eq:Yihat}, we had exact iid observations $Y_i \sim f$, obtained by replacing $\hat U_{ij}$ by $U_{ij}$. The empirical process in \cref{conj:ecdf} would then be a standard univariate empirical process, converging weakly at a rate of $n^{-1/2}$ (up to a logarithmic factor --- see \citealp{komlosApproximationPartialSums1975} and \citealp{bretagnolleHungarianConstructionsNonasymptotic1989}) to the $F$-Brownian bridge $W \circ F - W(1) F$. Because $F$ is differentiable at $x$, this Gaussian process naturally satisfies our assumption. Note that the Lipschitz part $W(1) F$ is not an adapted process.

        We can use the lens of empirical copula processes to analyze the effect of replacing the exact marginal distributions in the definition of $Y_i$ by empirical ones to obtain $\hat Y_i$. Indeed, by the arguments in \cref{sub:sigselection}, $B_n$ can be expressed as the empirical copula process (based on the rank-based pseudo-observations $\hat\bU_1, \dots, \hat\bU_n$) evaluated at the set $A_u^{\bs}$ defined at the beginning of the proof of \cref{prop:signature}. Its asymptotic distribution can thus be evaluated if that of the set-valued empirical copula process indexed by $\{A_u^{\bs}: u \in [0,1]\}$ is known. Based on preliminary results obtained by one of the authors (ML) in joint work with Axel Bücher, Johan Segers and Stanislav Volgushev, the process $B_n$ can be shown to converge to $W \circ F - \xi f$, where $\xi$ is a centered normal random variable whose variance depends on the copula and the dimension $d$ but not on $x$; the rate of convergence of the set-valued empirical copula process can also be assessed by carefully consulting the proofs in this joint work and using the aforementioned rate of convergence of the usual empirical process. The term $\xi f$ is obviously locally Lipschitz continuous at $x$ as long as $f$ itself is, which is assumed in \cref{prop:kde} below.
        
    \end{enumerate}
\end{rem}

Provided that \cref{conj:ecdf} is true, we can characterize the limiting distribution of $\hat f_n(x)$ for any $x \in (0,1)$ under standard smoothness assumptions on $f$ and a suitable choice of kernel:

\begin{prop} \label{prop:kde}
    Suppose \cref{conj:ecdf} holds and assume the following regularity conditions:
    \begin{enumerate}
        \item\label{ass:regularity-generator} The generator $f$ is twice differentiable at $x \in (0,1)$ and satisfies $f(x) > 0$.
        \item\label{ass:regularity-kernel} The kernel $K$ is a symmetric, unimodal and twice-differentiable probability density on $\R$ with bounded first and second derivatives such that as $y \to \infty$, for some $\beta > 0$,
        \[
            K(y) = O(y^{-(3+\beta)}), \quad K'(y) = o(y^{-3/2}), \quad K''(y) = O(y^{-\beta}).
        \]
        \item\label{ass:regularity-bandwidths} The sequence of bandwidths $\{h_n\}$ satisfies $n h_n^5 \to 0$ and $n h_n^\gamma \to \infty$ as $n \to \infty$, where $\gamma := \max\{5/(1+\alpha), 5 - 2\beta/(1+\beta)\} < 5$, for $\alpha$ as in \cref{conj:ecdf} and $\beta$ as in Assumption~\ref{ass:regularity-kernel}.
    \end{enumerate}
    Then
    \[
        \sqrt{n h_n} (\hat f_n(x) - f(x)) \wc \cN(0, \sigma^2)
    \]
    as $n \to \infty$, where $\sigma^2 := f(x) \int_{-\infty}^\infty K(z)^2 \dif z$.
\end{prop}

\begin{proof}
    Note that Assumption~\ref{ass:regularity-kernel} implies that $K$ and $K'$ are both bounded and absolutely integrable, and also that $K$ has a finite second moment $\int_{-\infty}^\infty y^2 K(y) \dif y < \infty$. Let $L := \max\{|f''(x)|, ||K'||_\infty, ||K''||_\infty\}$, which is finite by Assumptions~\ref{ass:regularity-generator} and~\ref{ass:regularity-kernel}. As a first step, let
    \[
        \psi_n(y) := \frac{1}{h_n} K\Big( \frac{x - \hat F_n^{-1}(y)}{h_n} \Big),
    \]
    so that $\hat f_n(x) = (1/n) \sum_{i=1}^n \psi_n(i/n)$. This is a step function with jumps at $1/n, \dots, (n-1)/n$, and therefore its total variation in the sense of Hardy and Krause is
    \[
        V_0^1(\psi_n) = \sum_{i=1}^{n-1} |\psi_n(i/n^+) - \psi_n(i/n^-)| \leq \frac{L}{h_n} \sum_{i=1}^{n-1} \Big( \frac{\hat F_n^{-1}(i/n)}{h_n} - \frac{\hat F_n^{-1}((i-1)/n)}{h_n} \Big) \leq \frac{L}{h_n^2},
    \]
    since $L$ is an upper bound for $K'$. Note also that the uniform grid $\{1/n, 2/n, \dots, (n-1)/n, 1\}$ has star discrepancy $1/n$, so
    \begin{equation} \label{eq:KHineq}
        \hat f_n(x) = \int_0^1 \psi_n(y) \dif y + O\Big( \frac{1}{n h_n^2} \Big)
    \end{equation}
    by the Koksma--Hlawka inequality \citep[see, e.g.,][Chapter 5]{lemieuxMonteCarloQuasiMonte2009}.

    By a second-order Taylor expansion of the function $t \mapsto K((x-t)/{h_n})$ around $F^{-1}(y)$, the leading term in \cref{eq:KHineq} can be expressed as
    \begin{multline}\label{eq:KDEthreeterms}
        \frac{1}{h_n} \int_0^1 K\Big( \frac{x - F^{-1}(y)}{h_n} \Big) \dif y - \frac{1}{h_n^2} \int_0^1 K'\Big( \frac{x - F^{-1}(y)}{h_n} \Big) \big( \hat F_n^{-1}(y) - F^{-1}(y) \big) \dif y
        \\
        + \frac{1}{2h_n^3} \int_0^1 K''\Big( \frac{x - F^{-1}(y) + u_n(y)}{h_n} \Big) \big( \hat F_n^{-1}(y) - F^{-1}(y) \big)^2 \dif y,
    \end{multline}
    where $u_n(y)$ lies between 0 and $F^{-1}(y) - \hat F_n^{-1}(y)$. Denote the three terms in \cref{eq:KDEthreeterms} as $I_n^{(0)}$, $I_n^{(1)}$ and $I_n^{(2)}$, respectively. We will prove that
    \begin{equation} \label{eq:toprove0}
        I_n^{(0)} = f(x) + O(h_n^2),
    \end{equation}
    \begin{equation} \label{eq:toprove1}
        \sqrt{n h_n} I_n^{(1)} \wc \cN(0, \sigma^2),
    \end{equation}
    and
    \begin{equation} \label{eq:toprove2}
        I_n^{(2)} = \Op\big( n^{-1} h_n^{-3 + \frac{\beta}{1+\beta}} \big).
    \end{equation}
    By Assumption~\ref{ass:regularity-bandwidths}, both $\{h_n^2\}$ and $\{n^{-1} h_n^{-3 + \frac{\beta}{1+\beta}}\}$ are $o(1/\sqrt{nh_n})$, since
    \[
        h_n^2 \sqrt{n h_n} = \sqrt{n h_n^5} \too 0
    \]
    and
    \[
        n^{-1} h_n^{-3 + \frac{\beta}{1+\beta}} \sqrt{n h_n} = \big( n h_n^{5 - \frac{2\beta}{1+\beta}} \big)^{-1/2} \too 0.
    \]
    Similarly, the error term in \cref{eq:KHineq} is also $o(1/\sqrt{nh_n})$. Therefore, \cref{eq:KHineq,eq:KDEthreeterms,eq:toprove0,eq:toprove1,eq:toprove2} directly imply the desired result.

    \noindent\textbf{Proof of \cref{eq:toprove0}.}
    First, by a change of variable and the fact that $f$ is supported on a subset of $[0, 1]$,
    \[
        I_n^{(0)} = \frac{1}{h_n} \int_0^1 K\Big( \frac{x - y}{h_n} \Big) f(y) \dif y = \frac{1}{h_n} \int_{-\infty}^\infty K\Big( \frac{x - y}{h_n} \Big) f(y) \dif y.
    \]
    Next, since $y \mapsto h_n^{-1} K((x - y)/h_n)$ is a probability density and hence integrates to $1$ over $\R$, we may write
    \[
        f(x) = \frac{1}{h_n} \int_{-\infty}^\infty K\Big( \frac{x - y}{h_n} \Big) f(x) \dif y,
    \]
    Let $\delta > 0$ be sufficiently small so that $|f(x) - f(x+y) + f'(x) y| \leq 2L y^2$ for $y \in [-\delta, \delta]$, which is possible since $|f''(x)| \leq L$. Then,
    \begin{align}
        |I_n^{(0)} - f(x)| &= \bigg| \frac{1}{h_n} \int_{-\infty}^\infty K\Big( \frac{x - y}{h_n} \Big) (f(y) - f(x)) \dif y \bigg|
        \notag
        \\
        &= \bigg| \frac{1}{h_n} \int_{-\infty}^\infty K\Big( \frac{y}{h_n} \Big) (f(x) - f(x+y)) \dif y \bigg|
        \notag
        \\
        &\leq \bigg| \frac{1}{h_n} \int_{-\delta}^\delta K\Big( \frac{y}{h_n} \Big) (f(x) - f(x+y)) \dif y \bigg| + \frac{1}{h_n} \int_{\R \setminus [-\delta, \delta]} K\Big( \frac{y}{h_n} \Big) f(x) \dif y
        \notag
        \\
        &\qquad + \frac{1}{h_n} \int_{\R \setminus [-\delta, \delta]} K\Big( \frac{y}{h_n} \Big) f(x+y) \dif y.
        \label{eq:boundI0}
    \end{align}
    By the choice of $\delta$, the first term in~\cref{eq:boundI0} is upper bounded by
    \begin{align*}
        f'(x) \bigg| \frac{1}{h_n} \int_{-\delta}^\delta K\Big( \frac{y}{h_n} \Big) y \dif y \bigg| + 2L \int_{-\delta}^\delta \frac{1}{h_n} K\Big(\frac{y}{h_n}\Big) y^2 \dif y &= 2L \int_{-\delta}^\delta \frac{1}{h_n} K\Big(\frac{y}{h_n}\Big) y^2 \dif y
        \\
        &= 2L h_n^2 \int_{-\delta/h_n}^{\delta/h_n} K(y) y^2 \dif y
        \\
        &= O(h_n^2),
    \end{align*}
    where the first equality is due to the symmetry of $K$ (Assumption~\ref{ass:regularity-kernel}), the second follows from a change of variable, and the third follows because $K$ has a finite second moment.
    
    By the same change of variable and the assumed bound on the tails of $K$ (Assumption~\ref{ass:regularity-kernel}), the second term in~\cref{eq:boundI0} is equal to
    \[
        f(x) \int_{\R \setminus [-\delta/h_n, \delta/h_n]} K(y) \dif y = O\bigg( \int_{\R \setminus [-\delta/h_n, \delta/h_n]} |y|^{-(3 + \beta)} \dif y \bigg) = O(h_n^{2+\beta}).
    \]
    Again by Assumption~\ref{ass:regularity-kernel}, the third term in~\cref{eq:boundI0} is upper bounded by
    \[
        \frac{K(\delta/h_n)}{h_n} \int_{-\infty}^\infty f(x+y) \dif y = O(h_n^2).
    \]
    Thus $|I_n^{(0)} - f(x)| = O(h_n^2)$, which proves \cref{eq:toprove0}.

    \noindent\textbf{Proof of \cref{eq:toprove1}.}
    Let
    \[
        R_n(y) := \big( \hat F_n^{-1}(y) - F^{-1}(y) \big) + \frac{n^{-1/2} B_n(F^{-1}(y))}{f(F^{-1}(y))}, \quad 0 \leq y \leq 1,
    \]
    so that
    \begin{align}
        \sqrt{n h_n} I_n^{(1)} &= -\frac{1}{h_n^{3/2}} \int_0^1 K'\Big( \frac{x - F^{-1}(y)}{h_n} \Big) \sqrt{n} \big( \hat F_n^{-1}(y) - F^{-1}(y) \big) \dif y
        \notag
        \\
        &= \frac{1}{h_n^{3/2}} \int_0^1 K'\Big( \frac{x - F^{-1}(y)}{h_n} \Big) \Big( \frac{B_n(F^{-1}(y))}{f(F^{-1}(y))} + \sqrt{n} R_n(y) \Big) \dif y
        \notag
        \\
        &= \frac{1}{h_n^{3/2}} \int_0^1 K'\Big( \frac{x - y}{h_n} \Big) B_n(y) \dif y + \frac{\sqrt{n}}{h_n^{3/2}} \int_0^1 K'\Big( \frac{x - F^{-1}(y)}{h_n} \Big) R_n(y) \dif y,
        \label{eq:In1dev}
    \end{align}
    where the first term in \cref{eq:In1dev} follows from the change of variable from $y$ to $F(y)$. Furthermore, in this term we may replace $B_n$ by $B$ up to negligible error, since
    \[
        \frac{1}{h_n^{3/2}} \int_0^1 \Big| K'\Big( \frac{x - y}{h_n} \Big) (B_n(y) - B(y)) \Big| \dif y 
        \leq
        \frac{1}{\sqrt{h_n}} \int_{-\infty}^\infty |K'(z)| \dif z \cdot \|B_n - B\|_\infty = \Op(n^{-\frac{1+\alpha}{10}} h_n^{-1/2}) = \op(1)
    \]
    by the absolute integrability of $K'$ together with the assumed rate of convergence of $B_n$ to $B$ and the choice of the sequence $\{h_n$\}; indeed,
    \[
        n^{-\frac{1+\alpha}{10}} h_n^{-\frac{1}{2}} = \big( n h_n^{\frac{5}{1+\alpha}} \big)^{-\frac{1+\alpha}{10}} \too 0
    \]
    by Assumption~\ref{ass:regularity-bandwidths}.

    We now turn to the second term in~\cref{eq:In1dev}. Let $\delta > 0$ be sufficiently small so that on $[x-\delta, x+\delta]$, $f$ is bounded away from $0$ and $\infty$ and is $(1/2)$-H\"older continuous, which is possible by Assumption~\ref{ass:regularity-generator}. The second term can now be decomposed into
    \begin{equation} \label{eq:In1term2}
        \frac{\sqrt{n}}{h_n^{3/2}} \int_{[0,1] \setminus [F(x-\delta), F(x+\delta)]} K'\Big( \frac{x - F^{-1}(y)}{h_n} \Big) R_n(y) \dif y + \frac{\sqrt{n}}{h_n^{3/2}} \int_{F(x-\delta)}^{F(x+\delta)} K'\Big( \frac{x - F^{-1}(y)}{h_n} \Big) R_n(y) \dif y.
    \end{equation}
    Recall that $B_n$ is assumed to be uniformly tight in the topology of uniform convergence on $[0,1]$. By Hadamard differentiability of the inverse mapping and the functional delta method (see, e.g., \citet{vandervaartWeakConvergenceEmpirical2023}, Section~3.10.5.2, or \citet{vervaatFunctionalCentralLimit1972}), $\sqrt{n} \big( \hat F_n^{-1} - F^{-1} \big)$ shares this property, and hence so does $\sqrt{n} R_n$. That is, $\|R_n\|_\infty = \Op(1/\sqrt{n})$ which suffices for the first term in \cref{eq:In1term2}:
    \begin{align*}
        \frac{\sqrt{n}}{h_n^{3/2}} \bigg| \int_{[0,1] \setminus [F(x-\delta), F(x+\delta)]} K'\Big( \frac{x - F^{-1}(y)}{h_n} \Big) R_n(y) \dif y \bigg| &\leq \frac{\Op(1)}{h_n^{3/2}} \int_{[0,1] \setminus [F(x-\delta), F(x+\delta)]} \Big| K'\Big( \frac{x - F^{-1}(y)}{h_n} \Big) \Big| \dif y
        \\
        &\leq \Op(1) \frac{|K'(\delta/h_n)|}{h_n^{3/2}}
        \\
        &= \op(1)
    \end{align*}
    by Assumption~\ref{ass:regularity-kernel}.
    
    However, such a crude bound on $R_n$ does not suffice to control the second term in~\cref{eq:In1term2}. Instead, we apply the Cauchy--Schwarz inequality to bound the integral therein by
    \begin{equation} \label{eq:HolderIneq}
        \bigg( \int_{F(x-\delta)}^{F(x+\delta)} K'\Big( \frac{x - F^{-1}(y)}{h_n} \Big)^2 \dif y \cdot \int_{F(x-\delta)}^{F(x+\delta)} R_n(y)^2 \dif y \bigg)^{1/2}
    \end{equation}
    By two changes of variable, the first integral in~\cref{eq:HolderIneq} is equal to
    \[
        \int_{x-\delta}^{x+\delta} K'\Big( \frac{x - y}{h_n} \Big)^2 f(y) \dif y = h_n \int_{-\delta/h_n}^{\delta/h_n} K'(z)^2 f(x - h_n z) \dif z = O(h_n)
    \]
    because $K'$ is square-integrable (since it is bounded and absolutely integrable) and $f$ is bounded on $[x-\delta, x+\delta]$. As for the second integral in~\cref{eq:HolderIneq}, it can be controlled by using Corollary~2.4 in \citet{dudleyOrderRemainderDerivatives1994}; in the notation of that paper,
    \[
        G(y) := \frac{F(x-\delta + 2\delta y) - F(x-\delta)}{F(x+\delta) - F(x-\delta)}
        \quad
        \text{and}
        \quad
        (G+g)(y) := \frac{\hat F_n(x-\delta + 2\delta y) - F(x-\delta)}{F(x+\delta) - F(x-\delta)}.
    \]
    Indeed, $G$ is then an increasing diffeomorphism of $[0,1]$ whose derivative
    \[
        G'(y) = \frac{2\delta}{F(x+\delta) - F(x-\delta)} f(x-\delta + 2\delta y)
    \]
    is uniformly bounded (above and below) and $(1/2)$-H\"older continuous. The inverses of $G$ and $G + g$ are given by
    \begin{align*}
        G^{-1}(z) &= \frac{F^{-1}(F(x-\delta) + (F(x+\delta) - F(x-\delta)) z) - (x-\delta)}{2\delta},
        \\
        (G+g)^{-1}(z) &= \frac{\hat F^{-1}(F(x-\delta) + (F(x+\delta) - F(x-\delta)) z) - (x-\delta)}{2\delta},
    \end{align*}
    respectively. We then easily obtain that
    \begin{equation} \label{eq:RgRn}
        R_g(z) := (G+g)^{-1}(z) - G^{-1}(z) + \frac{g(G^{-1}(z))}{G'(G^{-1}(z))} = \frac{1}{2\delta} R_n(F(x-\delta) + (F(x+\delta) - F(x-\delta)) z).
    \end{equation}
    The consequence of Corollary~2.4 in \citet{dudleyOrderRemainderDerivatives1994} is that
    \begin{equation}\label{eq:dudley}
        \int_0^1 R_g(z)^2 \dif z = O(\|g\|_{[2]}^3),
    \end{equation}
    where the quadratic variation norm $\|g\|_{[2]} := \|g\|_\infty + \|g\|_{(2)}$ of $g$ is defined as the sum of its supremum norm and the square root of its quadratic variation. By previous arguments on the uniform tightness of $B_n$, the supremum norm satisfies $\|g\|_\infty \leq \|\hat F_n - F\|_\infty / (F(x+\delta) - F(x-\delta)) = \Op(1/(\sqrt{n} (F(x+\delta) - F(x-\delta))))$. As for the quadratic variation of $g$, for any finite collection of points $0 \leq y_0 < y_1 < \dots < y_N \leq 1$, we have
    \begin{align*}
        \sum_{k=1}^N (g(y_k) - g(y_{k-1}))^2 &= \frac{1}{(F(x+\delta) - F(x-\delta))^2} \sum_{k=1}^N \big\{ \hat F_n(x-\delta + 2\delta y_k) - \hat F_n(x-\delta + 2\delta y_{k-1})
        \\
        &\qquad - F(x-\delta + 2\delta y_k) + F(x-\delta + 2\delta y_{k-1}) \big\}^2
        \\
        &\leq \frac{1}{(F(x+\delta) - F(x-\delta))^2} \sum_{k=1}^N \big(\hat F_n(x-\delta + 2\delta y_k) - \hat F_n(x-\delta + 2\delta y_{k-1})\big)^2
        \\
        &\qquad + \frac{1}{(F(x+\delta) - F(x-\delta))^2} \sum_{k=1}^N \big(F(x-\delta + 2\delta y_k) - F(x-\delta + 2\delta y_{k-1})\big)^2.
    \end{align*}
    The second term in the sum above vanishes as the mesh of the partition induced by $y_0, \dots, y_N$ tends to zero since $F$, being a continuous cdf, has zero quadratic variation. As for the first term, it is maximized precisely when $N = n\hat F_n(x+\delta) - n\hat F_n(x-\delta)$ and each plateau of $\hat F_n$ contains a single point, in which case it reduces to
    \[
        \frac{1}{(F(x+\delta) - F(x-\delta))^2} \sum_{k = n\hat F_n(x-\delta) + 1}^{n\hat F_n(x+\delta)} \frac{1}{n^2} = \frac{n\hat F_n(x+\delta) - n\hat F_n(x-\delta)}{n^2 (F(x+\delta) - F(x-\delta))^2} \leq \frac{1}{n (F(x+\delta) - F(x-\delta))^2}.
    \]
    Thus $\|g\|_{(2)} \leq 1/(\sqrt{n} (F(x+\delta) - F(x-\delta)))$, and \cref{eq:dudley} then states that
    \[
        \int_0^1 R_g(z)^2 \dif z = \Op(n^{-3/2}).
    \]
    This implies a bound of the same order for the second integral in \cref{eq:HolderIneq}, since by \cref{eq:RgRn} and a change of variable,
    \[
        \int_{F(x-\delta)}^{F(x+\delta)} R_n(y)^2 \dif y = (2\delta)^2 (F(x+\delta) - F(x-\delta)) \int_0^1 R_g(z)^2 \dif z = \Op(n^{-3/2}).
    \]
    From \cref{eq:HolderIneq}, we finally conclude that the second term in \cref{eq:In1term2} is at most of order
    \[
        \frac{\sqrt{n}}{h_n^{3/2}} (O(h_n) \Op(n^{-3/2}) )^{1/2} = \Op(n^{-1/4} h_n^{-1}) = \op(1)
    \]
    by Assumption~\ref{ass:regularity-bandwidths}. Piecing things together, we have rewritten the entirety of~\cref{eq:In1dev} as
    \begin{align}
        \sqrt{n h_n} I_n^{(1)} &= \frac{1}{h_n^{3/2}} \int_0^1 K'\Big( \frac{x - y}{h_n} \Big) B(y) \dif y + \op(1)
        \notag
        \\
        &= \frac{1}{h_n^{3/2}} \int_{x-\Delta}^{x+\Delta} K'\Big( \frac{x - y}{h_n} \Big) B(y) \dif y + \op(1)
        \notag
        \\
        &= \frac{1}{h_n^{3/2}} \int_{x-\Delta}^{x+\Delta} K'\Big( \frac{x - y}{h_n} \Big) W(F(y)) \dif y + \frac{1}{h_n^{3/2}} \int_{x-\Delta}^{x+\Delta} K'\Big( \frac{x - y}{h_n} \Big) V(y) \dif y + \op(1).
        \label{eq:In1dev2}
    \end{align}
    Here, the random variable $\Delta$ is positive and chosen so that $V$ has bounded weak derivative $V'$ on $[x-\Delta, x+\Delta]$ (see the remark after \cref{conj:ecdf}) and the second equality is due to Assumption~\ref{ass:regularity-kernel}, which ensures that
    \[
        \frac{1}{h_n^{3/2}} \int_{[0,1] \setminus [x-\Delta, x+\Delta]} K'\Big( \frac{x - y}{h_n} \Big) B(y) \dif y = \Op(h_n^{-3/2} K'(\Delta/h_n)) = \op(1),
    \]
    while the third equality is due to the assumed structure of the process $B$. We shall apply integration by parts to both integrals in~\cref{eq:In1dev2}, with $u(y) = W(F(y))$ in the first application and $u(y) = V(y)$ in the second; in both applications, the other integrand $v(y)$ is given by $-h_n K((x-y)/h_n)$. The first use of integration by parts is justified by the stochastic integration by parts formula~\citep[][Equation~(3.8) of Chapter 3]{karatzasBrownianMotionStochastic1998}, and the second is justified by the assumption that $V$ is Lipschitz continuous in a $\Delta$-neighborhood of $x$ (again, see the remark after \cref{conj:ecdf}). Upon these applications, the second term in \cref{eq:In1dev2} becomes
    \begin{align}
        &\phantom{{}={}}\frac{u(x+\Delta)v(x+\Delta) - u(x-\Delta)v(x-\Delta)}{h_n^{3/2}} - \frac{1}{h_n^{3/2}} \int_{x-\Delta}^{x+\Delta} v(y) \dif u(y) \nonumber
        \\
        &=\op(1) + \frac{1}{\sqrt{h_n}} \int_{x-\Delta}^{x+\Delta} K\Big( \frac{x-y}{h_n} \Big) V'(y) \dif y \nonumber
        \\
        &= \op(1) + \Op(h_n^{-1/2}), \label{eq:I1eq1}
    \end{align}
    where the first equality follows from our assumptions on the unimodality and the tails of $K$ (Assumption~\ref{ass:regularity-kernel}). Similarly, the first term in~\cref{eq:In1dev2} becomes
    \begin{align}
        &\frac{u(x+\Delta)v(x+\Delta) - u(x-\Delta)v(x-\Delta)}{h_n^{3/2}} - \frac{1}{h_n^{3/2}} \int_{x-\Delta}^{x+\Delta} v(y) \dif u(y)
        = \op(1) + \frac{1}{\sqrt{h_n}} \int_{x-\Delta}^{x+\Delta} K\Big( \frac{x-y}{h_n} \Big) \dif(W \circ F)(y). \label{eq:I1eq2}
    \end{align}
    By the assumed tails of $K$ (Assumption~\ref{ass:regularity-kernel}), we may now ``add back'' the set $[0,1] \setminus [x-\Delta, x+\Delta]$ to the domains of integration in~\cref{eq:In1dev2}, thereby obtaining from \cref{eq:I1eq1,eq:I1eq2}
    \[
        \sqrt{nh_n} I_n^{(1)}
        =
        \op(1) + \Op(h_n^{-1/2}) + \frac{1}{\sqrt{h_n}} \int_0^1 K\Big( \frac{x-y}{h_n} \Big) \dif(W \circ F)(y).
    \]
    The last term in the above display has a centered Gaussian distribution with variance
    \[
        \sigma_n^2 := \frac{1}{h_n} \int_0^1 K\Big( \frac{x - y}{h_n} \Big)^2 \dif [W \circ F](y),
    \]
    by the It\^{o} isometry \citep[][Proposition~2.10 of Chapter 3]{karatzasBrownianMotionStochastic1998}, where $[W \circ F]$ denotes the quadratic variation process of $W \circ F$. To prove \cref{eq:toprove1}, it therefore suffices to show that $\sigma^2_n \to \sigma^2$. To this end, note that since $W$ is a standard Wiener process, the process $[W \circ F]$ is deterministic and simply equal to $F$, so that
    \[
        \sigma_n^2 = \frac{1}{h_n} \int_0^1 K\Big( \frac{x - y}{h_n} \Big)^2 f(y) \dif y.
    \]
    Let $\delta > 0$ be sufficiently small so that $f$ is bounded on $[x-\delta, x+\delta]$, which is possible by the continuity of $f$ at $x$ implied by Assumption~\ref{ass:regularity-generator}. We then have the decomposition
    \begin{equation*}\label{eq:I1sigma_n}
        \sigma_n^2 = \frac{1}{h_n} \int_{[0,1] \setminus [x-\delta, x+\delta]} K\Big( \frac{x - y}{h_n} \Big)^2 f(y) \dif y + \frac{1}{h_n} \int_{x-\delta}^{x+\delta} K\Big( \frac{x - y}{h_n} \Big)^2 f(y) \dif y.
    \end{equation*}
    The first term in the above sum is bounded by
    \[
        \frac{K(\delta/h_n)^2}{h_n} \int_0^1 f(y) \dif y = O(h_n^{5+2\beta}) = o(1)
    \]
    by Assumption~\ref{ass:regularity-kernel}, while the second term is equal to
    \begin{equation}\label{eq:intK2f}
        \int_{-\delta/h_n}^{\delta/h_n} K(z)^2 f(x - h_n z) \dif z
    \end{equation}
    after a change of variable. By the continuity of $f$ at $x$, we have $K(z)^2 f(x - h_n z) \Ind{-\delta/h_n \leq z \leq \delta/h_n} \to K(z)^2 f(x)$ for every $z \in \R$, so by the boundedness of $f$ on $[x-\delta, x+\delta]$ and the dominated convergence theorem, the integral in \cref{eq:intK2f} converges to $f(x) \int_{-\infty}^{\infty} K(z)^2 \dif z = \sigma^2$. It follows that $\sigma_n^2 \to \sigma^2$.

    \noindent\textbf{Proof of \cref{eq:toprove2}.}
    As noted previously, \cref{conj:ecdf} ensures that $\sqrt{n} \big( \hat F_n^{-1} - F^{-1} \big)$ forms a uniformly tight sequence in the topology of uniform convergence on $[0,1]$. That is, $\sup_{0 \leq y \leq 1} \big( \hat F_n^{-1}(y) - F^{-1}(y) \big)^2 = \Op(1/n)$, and so
    \begin{align*}
        I_n^{(2)} &= \Op\Big(\frac{1}{n h_n^3}\Big) \int_0^1 K''\Big( \frac{x - F^{-1}(y) + u_n(y)}{h_n} \Big) \dif y.
    \end{align*}
    Let $\delta_n := h_n^{\beta/(1+\beta)}$. Since $h_n \to 0$, we may assume without loss of generality that $\delta_n$ is small enough that $f$ is bounded on $[x-\delta_n, x+\delta_n]$. The integral in the above display is bounded by
    \begin{equation}\label{eq:I3decomp}
        \int_{F(x-\delta_n)}^{F(x+\delta_n)} \Big| K''\Big( \frac{x - F^{-1}(y) + u_n(y)}{h_n} \Big) \Big| \dif y
        +
        \int_{[0,1] \setminus [F(x-\delta_n), F(x+\delta_n)]} \Big| K''\Big( \frac{x - F^{-1}(y) + u_n(y)}{h_n} \Big) \Big| \dif y.
    \end{equation}
    Because $|K''|$ is bounded by $L$, the first term in \cref{eq:I3decomp} is bounded by
    \[
        L (F(x+\delta_n) - F(x-\delta_n)) = O(\delta_n) = O\Big( h_n^{\frac{\beta}{1+\beta}} \Big).
    \]
    As for the second term, recall that $u_n(y)$ lies between $0$ and $F^{-1}(y) - \hat{F}_n^{-1}(y)$. Therefore, $u_n(y) = \Op(n^{-1/2}) = \op(\delta_n)$ uniformly on $[0,1]$. Therefore, for $y \in [0,1] \setminus [F(x-\delta_n), F(x+\delta_n)]$, we have
    \[
        \Big| \frac{x - F^{-1}(y) + u_n(y)}{h_n} \Big| \geq \frac{\delta_n}{h_n} (1 - \op(1)) = h_n^{-\frac{1}{1+\beta}} (1 - \op(1))
    \]
    so by Assumption~\ref{ass:regularity-kernel},
    \begin{equation*}
        \int_{[0,1] \setminus [F(x-\delta_n), F(x+\delta_n)]} \Big| K''\Big( \frac{x - F^{-1}(y) + u_n(y)}{h_n} \Big) \Big| \dif y
        \leq K''\Big( h_n^{-\frac{1}{1+\beta}} (1 - \op(1)) \Big) = \Op\Big( h_n^{\frac{\beta}{1+\beta}} \Big).
    \end{equation*}
    The sum in \cref{eq:I3decomp} is thus of order $\Op\Big( h_n^{\frac{\beta}{1+\beta}} \Big)$, from which we finally conclude that
    \[
        I_n^{(2)} 
        = 
        \Op\Big(\frac{1}{n h_n^3}\Big) \cdot \Op\Big( h_n^{\frac{\beta}{1+\beta}} \Big)
        =
        \Op\Big( n^{-1} h_n^{-3 + \frac{\beta}{1+\beta}} \Big),
    \]
    as desired.
\end{proof}

\section{Simulations and applications}\label{sec:simsandapps}

\subsection{Simulations}

To assess the practical performance of our inference methods, we conducted a series of simulations with three main objectives: to evaluate the accuracy of the data-driven signature selection procedure (\cref{sub:sigselection}), to study parameter recovery under maximum likelihood estimation (\cref{sub:parametric}), and to assess the performance of the generator KDE (\cref{sub:nonpara}). The simulations span the range of generator families introduced in \cref{sub:simpleexamples}, including the triangular distribution, several members of the beta family, and an asymmetric mixture of truncated normal distributions. For each objective, we systematically varied the sample size $n$ and the dimension $d$ to explore their effects on accuracy, robustness, and estimation quality.

For signature selection, we considered all combinations of $n \in \{50, 100, 200, 500, 1000\}$ and $d \in \{2,3,4,5\}$, and evaluated recovery using both the Kolmogorov--Smirnov and Cram\'er--von Mises criteria. The generators here were the $\mathrm{Triangular}(1,1)$,\footnote{The two-parameter triangular distribution on $[0,1]$ with mode $m$ and upper limit $b$ --- which we write as $\mathrm{Triangular}(b,m)$ --- is a special case of the three-parameter triangular distribution with its lower limit fixed to $0$, where $0 < m \leq b \leq 1$.} and $\mathrm{Beta}(1/2, 1)$ densities, as well as the $(1/4) \cTN_{[0,1]}(1/4, 1/100) + (3/4) \cTN_{[0,1]}(3/4, 1/100)$ mixture density used in \cref{sub:simpleexamples}. In each replicate, we simulated the data $\bU_1, \ldots, \bU_n$ directly from the copula using the true generator using \cref{algo:sampler}, with no marginal transformations applied; we then applied the signature selection criterion to the rank-based pseudo-observations. For each dimension $d$, we evaluated the performance across all $2^{d-1}$ possible signature configurations. As we show in \cref{fig:signaturesim}, across all settings, the recovery rate quickly reached 100\% even for modest sample sizes, with little sensitivity to dimension or the generator family.

\begin{figure}[ht]
    \centering
    \includegraphics[width=0.32\textwidth]{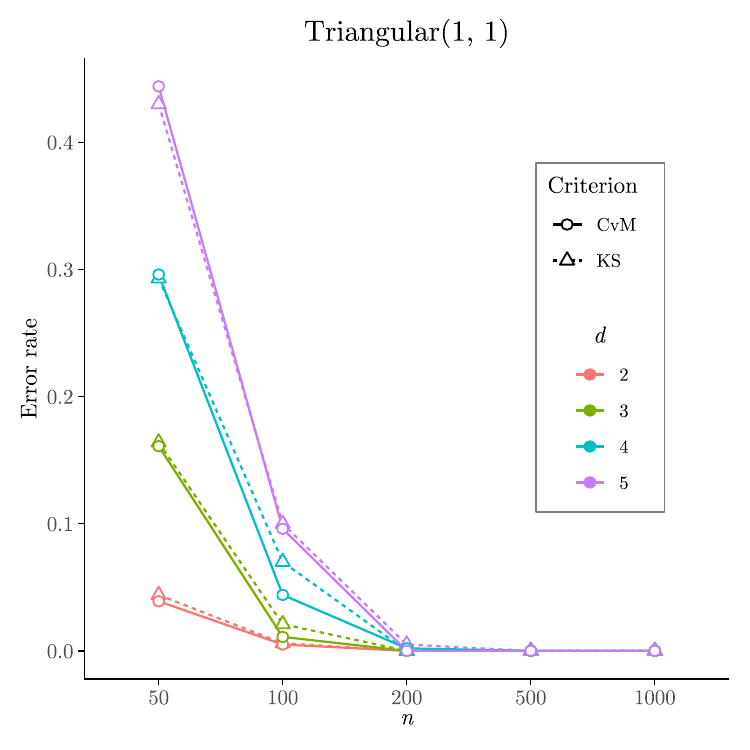}
    \includegraphics[width=0.32\textwidth]{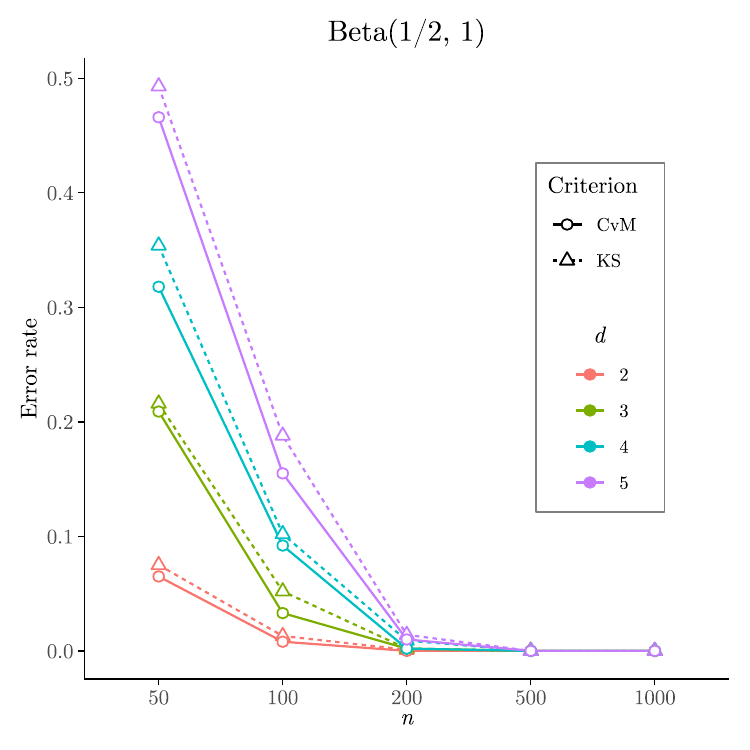}
    \includegraphics[width=0.32\textwidth]{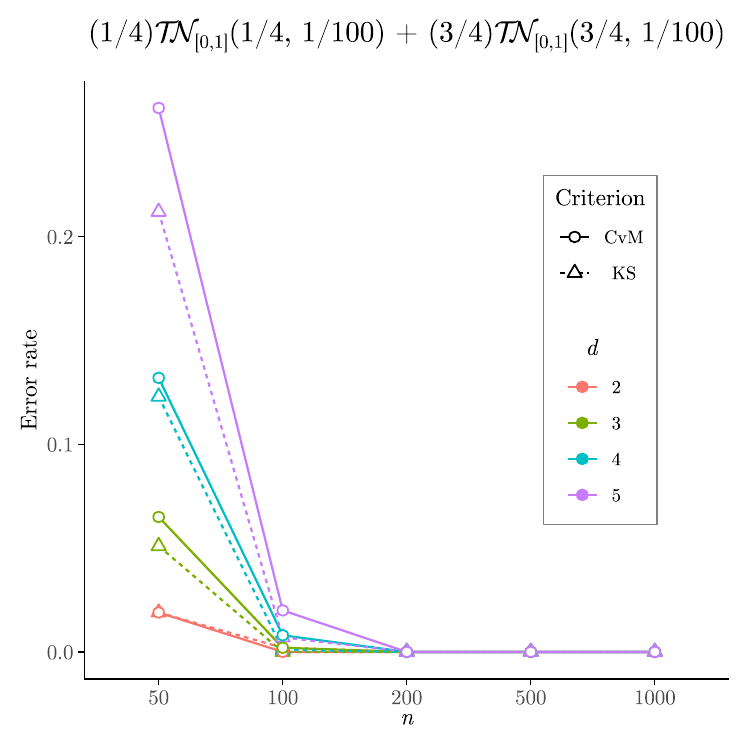}
    \caption{Proportion of replicates with incorrect signature recovery by method (Kolmogorov--Smirnov versus Cram\'er--von Mises), sample size $n \in \{50, 100, 200, 500, 1000\}$, and dimension $d \in \{2,3,4,5\}$, across several generator families used in \cref{sub:simpleexamples}. The recovery rate reaches $100\%$ quickly and is largely insensitive to $d$ and the generator family.}
    \label{fig:signaturesim}
\end{figure}

For parameter estimation, we used all combinations of $n \in \{100, 500, 1000, 5000, 10000\}$ and $d \in \{2,3,4,5\}$ and selected the same parametric families for the generator. We fixed the true signatures to be $(0,1)$ for $d=2$, $(0,1,0)$ for $d=3$, $(0,1,0,1)$ for $d=4$, and $(0,1,0,1,0)$ for $d=5$. For each $d$ and $n$, we simulated $d$-dimensional iid random vectors $\bX_1, \ldots, \bX_n$ with the given copula and $j$th marginal $X_{ij} \sim \cN(j, j)$, for $j \in \{1,\ldots,d\}$. We then estimated the generator parameters using both methods described in \cref{sub:parametric}: a two-stage estimator in which the marginals were estimated first (using the standard MLEs for the mean and variance under the assumption of normality), and the rank-based pseudo-observations without imposing any parametric assumptions on the marginals. In both cases, maximum likelihood was performed on the resulting wrapped sums using the \texttt{optim} and \texttt{nlm} functions in \textsf{R}. We independently replicated this process $100$ times in order to estimate the RMSE of each estimator. The results, displayed in \cref{fig:parametersim_triangular_1-1,fig:parametersim_beta_0.5-1,fig:parametersim_mixture_0.25-0.75}, show that RMSEs generally decrease with increasing $n$, and the shape and rate of convergence are comparable to those seen in classical univariate asymptotics. Interestingly, neither of the two methods for handling the marginals consistently outperformed the other across any factor in the simulation setup.

\captionsetup{format=plain, labelformat=simple, labelsep=colon}

\begin{figure}[ht]
    \ContinuedFloat*
    \centering
    \includegraphics[width=0.7\linewidth, keepaspectratio]{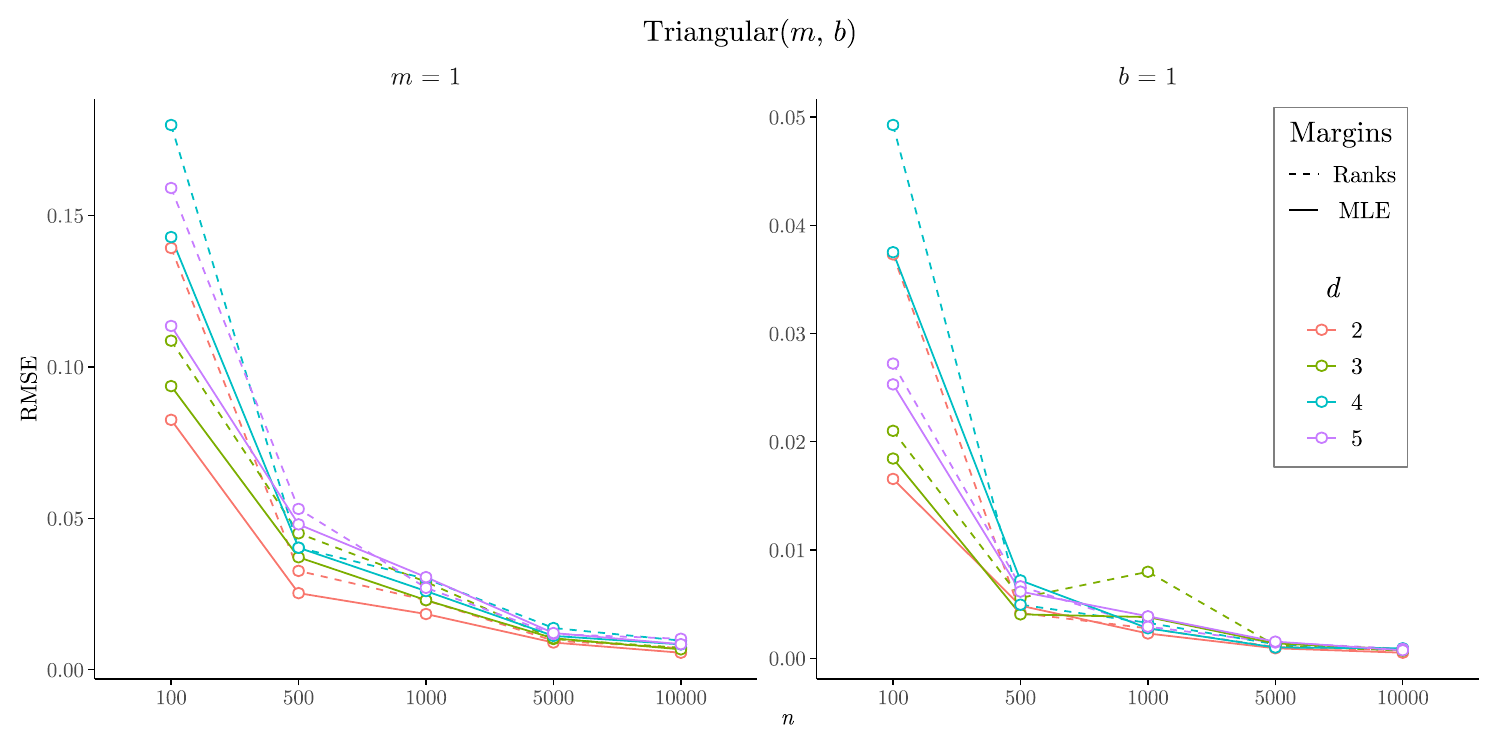}
    \caption{Parameter RMSEs for the triangular generator under maximum likelihood estimation, comparing two marginal estimation strategies: (i) two-stage parametric margins and (ii) rank-based pseudo-observations. RMSE decreases with $n$; neither marginal strategy uniformly dominates across settings.}
    \label{fig:parametersim_triangular_1-1}
\end{figure}

\begin{figure}[ht]
    \ContinuedFloat
    \centering
    \includegraphics[width=0.7\linewidth, keepaspectratio]{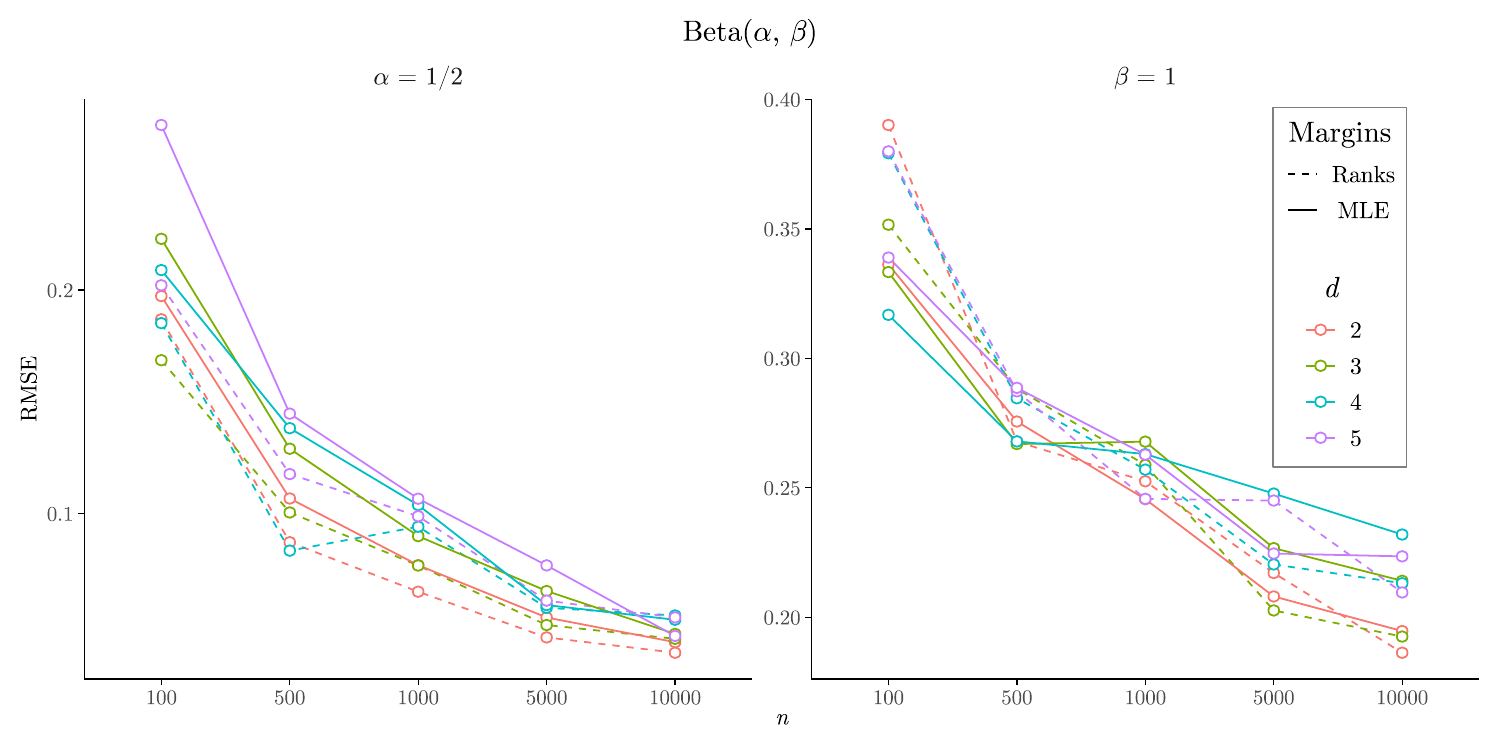}
    \caption{Parameter RMSEs for the $\mathrm{Beta}(1/2, 1)$ generator, with the same layout and comparison as \cref{fig:parametersim_triangular_1-1}. The RMSEs decline with $n$; both marginal estimation strategies perform similarly well.}
    \label{fig:parametersim_beta_0.5-1}
\end{figure}

\begin{figure}[ht]
    \ContinuedFloat
    \centering
    \includegraphics[width=0.7\linewidth, keepaspectratio]{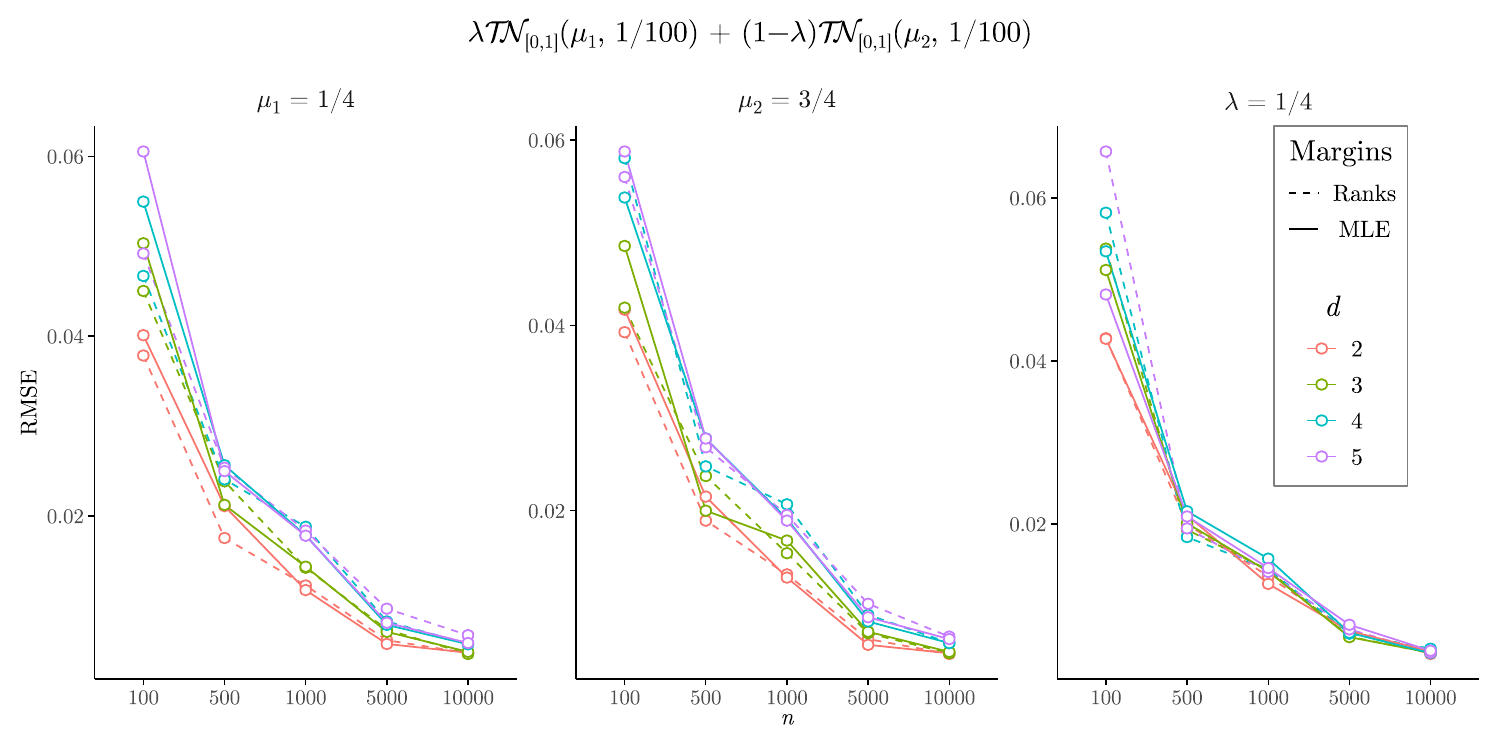}
    \caption{Parameter RMSEs for the $(1/4) \cTN_{[0,1]}(1/4, 1/100) + (3/4) \cTN_{[0,1]}(3/4, 1/100)$ generator; despite the more challenging target, RMSE improves steadily with $n$ and the two marginal estimation strategies remain competitive.}
    \label{fig:parametersim_mixture_0.25-0.75}
\end{figure}

For nonparametric estimation of the generator, we examined the ability of the rank-based kernel density estimator to recover the true generator density from pseudo-observations. We again simulated $d$-dimensional iid random vectors $\bX_1, \ldots, \bX_n$ using the same settings for $d \in \{2,3,4,5\}$ and $n \in \{100, 500, 1000, 5000, 10000\}$ as in the previous simulation study, with the same fixed signatures and normal margins. The underlying generator was again taken as the $(1/4) \cTN_{[0,1]}(1/4, 1/100) + (3/4) \cTN_{[0,1]}(3/4, 1/100)$ density --- an intentionally challenging choice, as KDEs often struggle with multimodal targets. Again, for each simulated dataset, we constructed pseudo-observations using both maximum likelihood-fitted margins and rank-based margins; we then applied univariate KDEs to the resulting wrapped sums to estimate the generator. Specifically, we used the \texttt{density} function in \textsf{R} with a Gaussian kernel and a plug-in bandwidth based on the Sheather--Jones method \citep{sheatherReliableDataBasedBandwidth1991}, evaluated over 200 evenly spaced points in $[0,1]$. The entire process was replicated 100 times to assess variability.

The results are summarized in \cref{fig:kde_6a,fig:kde_6b}, which plot the average estimated density across replicates with pointwise 95\% confidence bands and display the mean integrated squared error (MISE) for each setting. Even with as few as $n = 100$ observations, the KDE is able to capture the overall shape of the true generator quite well, with the primary discrepancy being an underestimation of the modal peaks. Performance improves substantially by $n = 500$, as mode recovery sharpens and confidence bands narrow visibly. This trend continues as $n$ increases, as the MISE values steadily decrease and the estimated densities approach the true generator density; performance degrades only slightly as $d$ increases. The parametric approach yields slightly lower MISE values overall, but the differences are modest, and rank-based estimation remains competitive even when the true margins are known only up to estimation of parameters.

\begin{figure}[ht]
    \ContinuedFloat*
    \centering
    \includegraphics[width=0.49\linewidth]{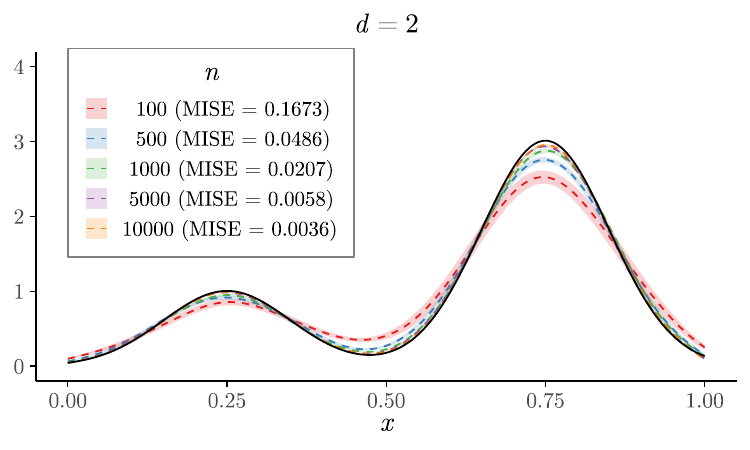}
    \includegraphics[width=0.49\linewidth]{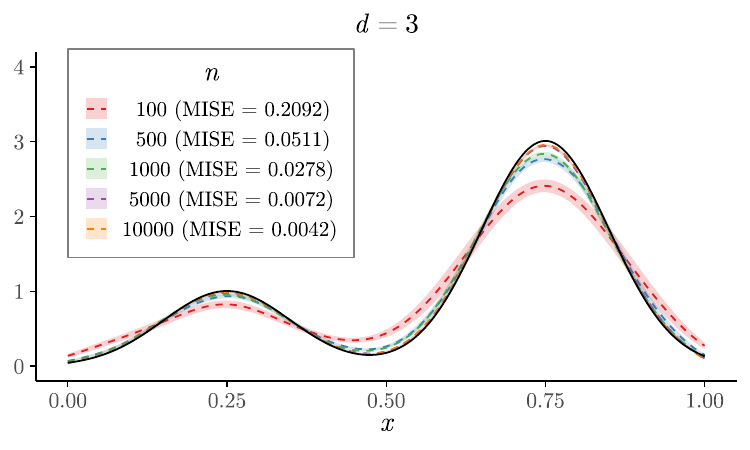} \\
    \includegraphics[width=0.49\linewidth]{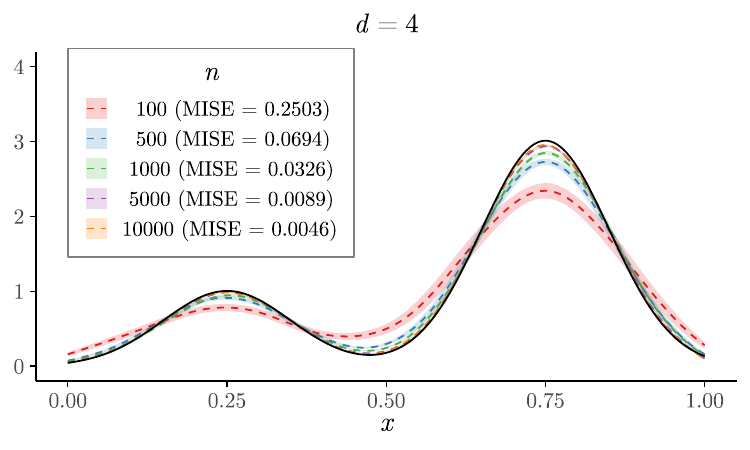}
    \includegraphics[width=0.49\linewidth]{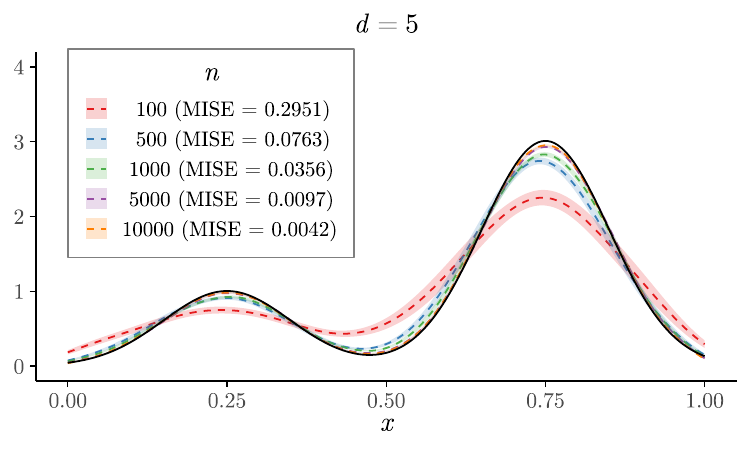}
    \caption{KDEs and corresponding MISEs of the generator for $d \in \{2,3,4,5\}$ based on parametric margins, for $n \in \{100, 500, 1000, 5000, 10000\}$. Dashed curves are means over $100$ replicates with pointwise $95\%$ confidence bands; the solid black curve is the true generator. Even at $n = 100$, the overall shape is captured, with rapid improvement by $n = 500$.}
    \label{fig:kde_6a}
\end{figure}

\begin{figure}[ht]
    \ContinuedFloat
    \centering
    \includegraphics[width=0.49\linewidth]{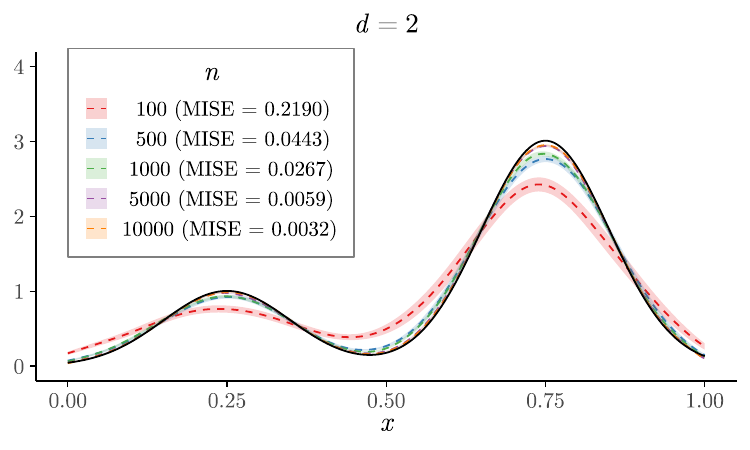}
    \includegraphics[width=0.49\linewidth]{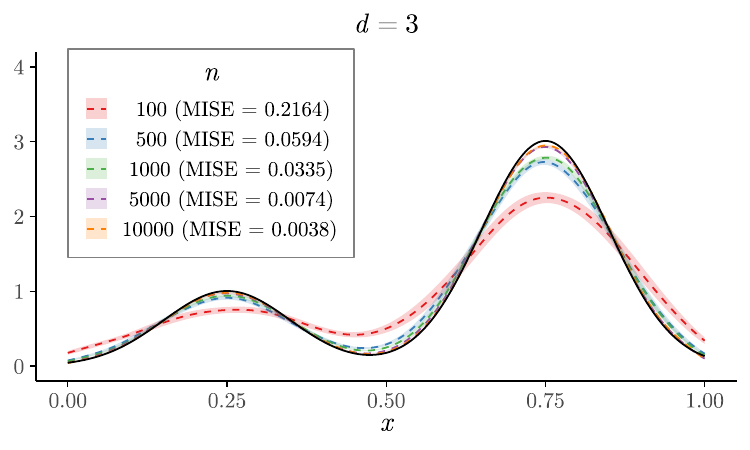} \\
    \includegraphics[width=0.49\linewidth]{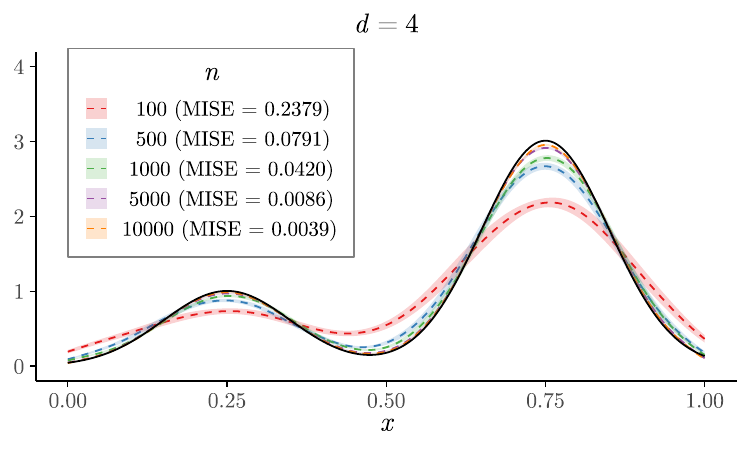}
    \includegraphics[width=0.49\linewidth]{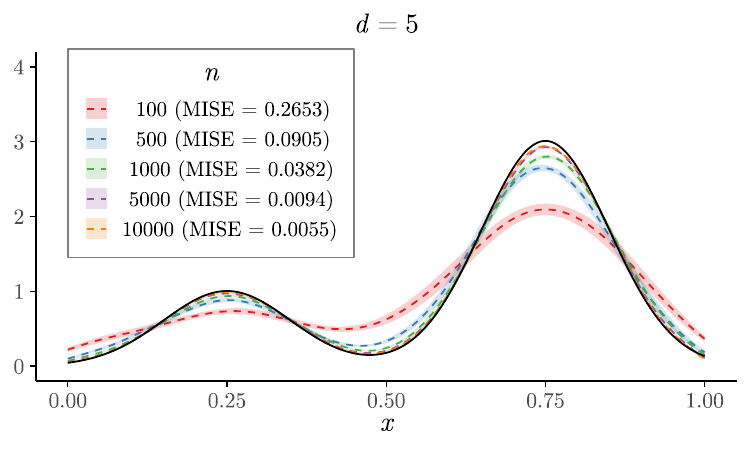}
    \caption{Generator KDEs and MISEs based on rank-based pseudo-observations under the same setup as \cref{fig:kde_6a}.}
    \label{fig:kde_6b}
\end{figure}

\subsection{Application: neural connectivity}\label{sub:neural}

The hippocampus, located within the medial temporal lobe of the brain, is intimately involved in many types of learning and memory, and makes extremely complex connections with other brain areas such as the prefrontal cortex. The hippocampus itself has been subdivided structurally and functionally into different subregions including the \emph{dentate gyrus}, the \emph{subiculum} and the \emph{CA3} area. The latter region is thought to operate as an auto-association network, meaning that it can integrate inputs from different areas of the cortex (for example, auditory and visual inputs), and then allow the recall of the entire set of associations --- the integrated memory --- when any one aspect of the memory is presented as a cue \citep{rollsDilutedConnectivityPattern2015, daltonSegmentingSubregionsHuman2017, przezdzikFunctionalOrganisationHippocampus2019}.

It is of great interest to understand the mechanisms and substructures involved in associative learning (for example, a language learner's associating the pronunciation of a new word with its written form). To study these mechanisms, \citet{brincatFrequencyspecificHippocampalprefrontalInteractions2015} analyze neurophysiological activity at various sites in the prefrontal cortex and hippocampus of rhesus monkeys during a paired associative memory task that requires the animal to remember which images from a given set are correctly paired. Analyzing electrophysiological spiking patterns in these sites, the authors observe that while the prefrontal cortex appears to be the principal site of rapid object associative learning, the hippocampus appears to be more involved in signalling feedback to the prefrontal cortex about the accuracy of answers. The data from their experiment consists of instantaneous phase angles, recorded from $24$ channels (electrodes) placed among the three aforementioned hippocampal regions and the prefrontal cortex across $840$ trials.

In a follow-up analysis of the same dataset, \citet{kleinTorusGraphsMultivariate2020} develop a new class of graphical models for data in the $d$-dimensional torus, which they call \emph{torus graphs}. Their framework is built to model the interactions between phase angles $X_1, \ldots, X_d$. A key feature of their model is the explicit distinction between two fundamental types of dependence between pairs of phase angles: rotational dependence, which is a function of the phase differences $X_j - X_k$, and reflectional dependence, which is a function of the phase sums $X_j + X_k$. This distinction allows them to define important subfamilies of torus graphs, including the \emph{phase difference submodel}, corresponding to random vectors $\bX \in [0, 2\pi)^d$ for which no $X_j$ is a rotation from $-X_k$ by some angle, and the \emph{uniform marginal submodel}, corresponding to the case where $\bX$ has uniform margins. Based on preliminary hypothesis tests \citep[Figure 8]{kleinTorusGraphsMultivariate2020}, the authors conclude that the data exhibit significant rotational dependence but lack evidence of reflectional dependence, leading them to adopt the more parsimonious model in the intersection of the two subfamilies (i.e., a phase difference model with uniform margins) for their analysis of neural connectivity. 

Using our framework to investigate the data, we first applied our data-driven signature selection procedure (see \cref{sub:sigselection}) to every possible pair of the $24$ neural phase channels. This procedure is model-agnostic, in the sense that it does not presume a specific type of dependence \emph{a priori}. The result was unequivocal: the signature $\bs = (0,1)$ was selected in every instance. Within our framework, a signature of $\bzero$ would correspond to dependence based on wrapped sums, while the $(0,1)$ signature corresponds to dependence based on wrapped differences. Our empirical finding therefore independently validates the conclusion reached by \citet{kleinTorusGraphsMultivariate2020}: the dependence structure in this dataset is overwhelmingly characterized by the consistency of phase differences.

The empirical agreement on the type of dependence present in the data motivates a more focused investigation into its structure. \citet{kleinTorusGraphsMultivariate2020} note that bivariate phase coupling measures generally fail to distinguish direct from indirect coupling between nodes of interest (in this case, the neural pathways operating within the hippocampus). However, under the phase difference model with uniform margins, the authors introduce a transformation of the pairwise parameters that yields a phase-locking-value-like conditional coupling measure bounded in $[0,1]$, enabling pairwise measures that condition on the remaining nodes. In this way, they can distinguish direct from indirect connections between nodes, both in simulated data and in experimental data gathered from the five linearly placed channels in the CA3 region. 

Consistent with our signature selection results, the analysis of \citet{kleinTorusGraphsMultivariate2020} isolates rotational (difference-based) coupling and identifies edges that are local along the CA3 probe. In particular, they show that there is no direct edge between the dentate gyrus and prefrontal cortex, but they each couple to the subiculum. Within the CA3 subregion, a reduced model fit to the five linearly arranged channels recovers a nearest-neighbor (chain-like) conditional independence pattern; in the full $24$-dimensional fit, the adjacency matrix (with channels ordered by physical position) again reveals a linear channel structure (see their Figure 8), even though the model does not include position as an input. These findings suggest that direct dependencies are primarily local, occurring between adjacent CA3 channels. We therefore focus on the five sequential CA3 pairs (i.e., channels $(j, j+1)$ for $j \in \{1,2,3,4\}$), which aligns our pairwise analysis with the graph-level edges and avoids confounding from indirect paths in non-adjacent pairs.

The data consists of four pairs of $840$ bivariate angle measurements $(x^{(j)}_{i1},x^{(j)}_{i2}) \in [-\pi, \pi]^2$, for $j \in \{1,\ldots,4\}$ and $i \in \{1, \ldots, 840\}$ (in this section, we present results for pairs $2$ and $3$, while the corresponding figures and tables for pairs $1$ and $4$ are provided in \cref{app:additionalplots}). To understand the dependence structure among the phase angles, we begin by examining the empirical beta copula density for each pair, as shown in \cref{fig:torus_empiricalbetas}. The densities are smooth and concentrate most of their mass around a single ridge along the main diagonal $u_1 = u_2$, indicating strong positive dependence between the two marginal distributions; peaks around the corners $(0,1)$ and $(1,0)$ likely explained by the rotational structure. This structure is consistent with the physical proximity of the pairs of recording sites, and suggests that the two corresponding neurons are driven by a common phase signal, possibly with negligible delay.

\begin{figure}[ht]
    \centering
    \includegraphics[width=0.32\textwidth]{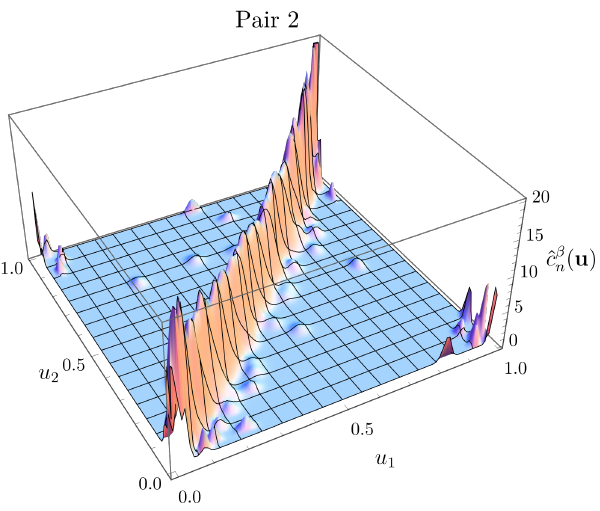}
    \includegraphics[width=0.32\textwidth]{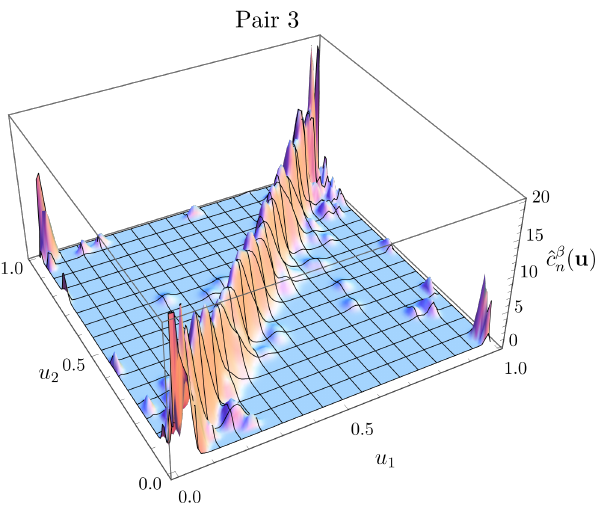}
    \caption{Empirical beta copula densities for pairs 2 and 3 of the CA3 channel data ($n = 840$ per pair). Each surface exhibits a pronounced ridge along the diagonal $u_1 = u_2$ (indicating strong positive dependence) with some additional mass around the corners $(0,1)$ and $(1,0)$; note that the view angle differs here from that in \cref{sec:examples}. Corresponding plots for pairs 1 and 4 are given in \cref{fig:torus_empiricalbetas14}.}
    \label{fig:torus_empiricalbetas}
\end{figure}

To gain further insight into the dependence between each pair, we examine the distributions of the wrapped differences $(\hat{u}_{i,1}^{(j)} - \hat{u}_{i,2}^{(j)}) \bmod{1}$, where $\hat{u}_{1,1}^{(j)}, \ldots, \hat{u}_{n,1}^{(j)}$ and $\hat{u}_{1,2}^{(j)}, \ldots, \hat{u}_{n,2}^{(j)}$ are the rank-based pseudo-observations of the marginal variables corresponding to pair $j$, for $j \in \{1, \ldots, 4\}$. Histograms of these data, shown in red in \cref{fig:torus_histograms}, exhibit prominent peaks appear near $0$ and $1$, and a relative trough in between. However, this apparent bimodality is purely an artifact of projecting a circular random variable onto the interval $[0,1)$; values near $0$ and $1$ are adjacent on the circle, and the apparent trough at $1/2$ arises from an arbitrary choice of origin. When we instead examine the shifted wrapped differences $(\hat{u}_{i,1}^{(j)} - \hat{u}_{i,2}^{(j)} + 1/2) \bmod{1}$, shown in blue in \cref{fig:torus_histograms}, the resulting histograms are cleanly unimodal across all pairs, with modes near $1/2$. In the language of our copulas, this reparameterization corresponds to shifting the argument of the generator $f$, and does not alter the copula model. However, it facilitates fitting by aligning the mode of $f$ with the center of the interval; in particular, under the signature $(0,1)$, the ridge along the diagonal $u_1 = u_2$ is well captured by a generator centered at $1/2$. We adopt the shifted representation throughout for clarity and computational convenience. 

\begin{figure}[ht]
    \centering
    \includegraphics[width=0.49\linewidth]{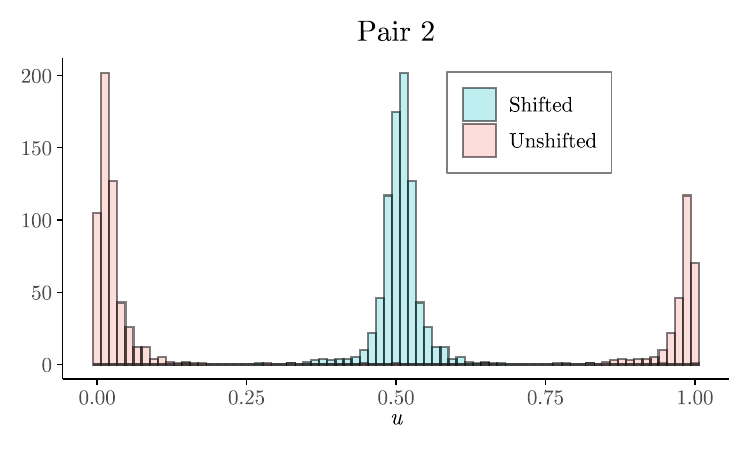}
    \includegraphics[width=0.49\linewidth]{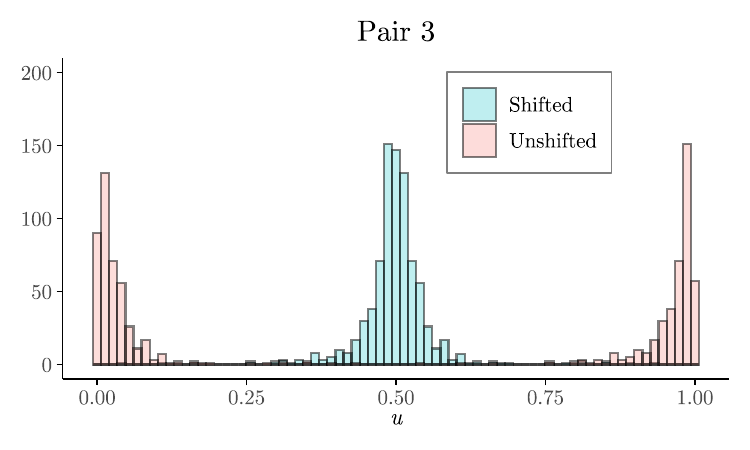}
    \caption{Histograms of wrapped differences for pairs 2 and 3 of the CA3 channel data  pair before (red) and after (blue) shifting by $+1/2 \bmod{1}$. The apparent pre-shift bimodality is a circular edge effect; the shift yields a clean unimodal distribution with mode near $1/2$, which aligns the generator with the $(0,1)$ signature. Corresponding histograms for pairs 1 and 4 are given in \cref{fig:torus_histograms14}.}
    \label{fig:torus_histograms}
\end{figure}

We now establish a formal connection between the \citet{kleinTorusGraphsMultivariate2020} submodel and our copula family. The density of the bivariate phase difference model with uniform margins takes the form \citep[][Equation 2.2]{kleinTorusGraphsMultivariate2020}
\[
    g(\bx) \propto \exp\left( \phi_1 \cos(x_1 - x_2) + \phi_2 \sin(x_1 - x_2)\right), \quad \bx \in [0,2\pi]^2.
\]
Applying the transformations $u_1 = x_1/2\pi$ and $u_2 = x_2/2\pi$ and using the fact that $2\pi x \bmod{2\pi} = 2\pi(x \bmod{1})$ for any $x \in \R$, we scale the support of $g$ to $[0,1]^2$ to obtain the scaled density
\begin{align*}
    \bar{g}(\bu) &\propto \exp\left(\phi_1 \cos(2\pi(u_1 - u_2 \bmod{1})) + \phi_2 \sin(2\pi(u_1 - u_2 \bmod{1}))\right), \quad \bu \in [0,1]^2
\end{align*}
which depends on $\bu$ through the wrapped sum $\tu_1 \oplus \tu_2$ under the signature $(0,1)$. 
It follows that $\bar{g} = c^{(0,1)}_{f_{\bphi}}$, where $f_{\bphi} \in \cF_{[0,1]}$ is the density of a von Mises distributed random variable scaled by a factor of $1/2\pi$:
\begin{equation}\label{eq:vonM}
    f_{\bphi}(x) := \frac{\exp(\phi_1 \cos(2\pi x) + \phi_2 \sin(2\pi x))}{I_0(\sqrt{\phi_1^2 + \phi_2^2})}
    = \frac{\exp(\kappa \cos(2\pi (x - \mu)))}{I_0(\kappa)}
    , \quad x \in [0,1]
\end{equation}
where $\kappa = \sqrt{\phi_1^2 + \phi_2^2}$, $\mu = \tan^{-1}(\phi_2/\phi_1)$, and $I_0$ is the modified Bessel function of the first kind of order $0$; we write $\mathrm{vonMises(\phi_1, \phi_2)}$ for this scaled von Mises distribution. Thus, up to rescaling, the distribution of the bivariate phase difference model with uniform margins is exactly $C_{f_{\bphi}}^{(0,1)} \in \cC$.

As a contrast to the $\mathrm{vonMises(\phi_1, \phi_2)}$ distribution, we first considered four other unimodal parametric families in $\cF_{[0,1]}$ that could plausibly serve as generators: the $\text{Beta}(\alpha, \beta)$ distribution, the $\cTN_{[0,1]}(\mu, \sigma^2)$ distribution, the $\mathrm{Kumaraswamy}(a,b)$ distribution, and the $\mathrm{Logit}\cN(\mu, \sigma^2)$ distribution.\footnote{$X \sim \mathrm{Kumaraswamy}(a,b)$ if and only if $X^a \sim \mathrm{Beta}(a,b)$, and $X \sim \mathrm{Logit}\cN(\mu, \sigma^2)$ if and only if $\log(X/(1-X)) \sim \cN(\mu, \sigma^2)$.} For each pair $j$, we fit these distributions to the shifted wrapped differences $(\hat{u}_{i,1}^{(j)} - \hat{u}_{i,2}^{(j)} + 1/2) \bmod{1}$ using maximum likelihood as developed in \cref{sub:parametric}, via the \texttt{optim} function in \textsf{R}. We also obtain a fully nonparametric estimate of the generator using the estimator in \cref{sub:nonpara}. The resulting fitted generators are illustrated in \cref{fig:torus_gens_uni}. While the KDE fits the data reasonably well, it is clear that none of the five unimodal parametric families are able to adequately capture the mode of the distributions.

\begin{figure}[ht]
    \centering
    \includegraphics[width=0.49\linewidth]{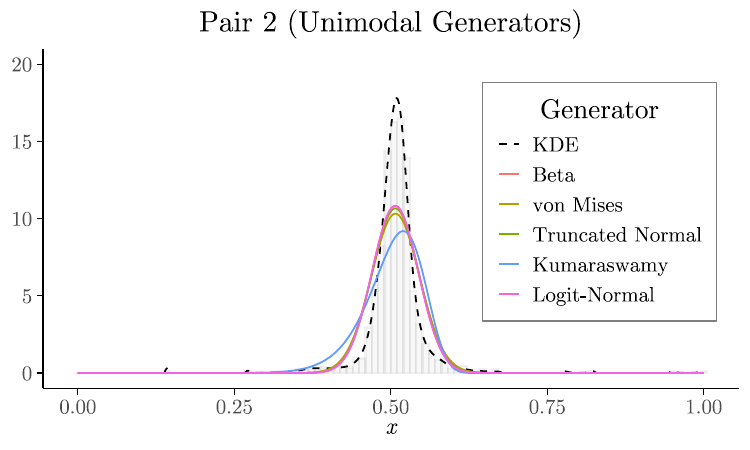}
    \includegraphics[width=0.49\linewidth]{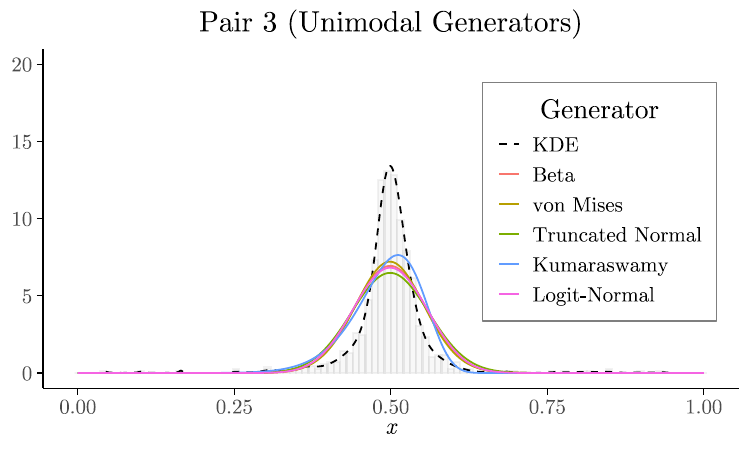}
    \caption{Unimodal parametric fits of the generator (beta, truncated normal, Kumaraswamy, logit-normal, and von Mises) overlaid on the shifted wrapped difference histograms for pairs 2 and 3. Corresponding fits for pairs 1 and 4 are given in \cref{fig:torus_gens_uni14}.}
    \label{fig:torus_gens_uni}
\end{figure}

Suspecting that this deficiency is due to local modes in the data away from $1/2$, we repeat the experiment with $2$-component mixtures of each of the five parametric distributions ($15$ in total). These fitted distributions, shown in \cref{fig:torus_gens_mix}, fare substantially better. \cref{tab:fit_v1,tab:fit_v2,tab:fit_v3,tab:fit_v4} show the resulting parameter estimates for each pair for all $20$ generators considered, along with corresponding estimates of Spearman's rho, Kendall's tau, and the Dette--Siburg--Stoimenov coefficient, with each computed using the formulas given by \cref{prop:spearman,prop:kendall,prop:chatterjee}; in the $2$-component mixtures, $\pi \in [0,1]$ is the mixture weight on the first component and $1-\pi$ is that on the second. We compare the parametric models using the Akaike information criterion (AIC) \citep{akaikeNewLookStatistical1974}, also shown in \cref{tab:fit_v1,tab:fit_v2,tab:fit_v3,tab:fit_v4}. In all cases, the AIC values suggest that of the ten parametric generators, the $2$-component mixture of von Mises distributions best fits the data.

\begin{figure}[ht]
    \centering
    \includegraphics[width=0.49\linewidth]{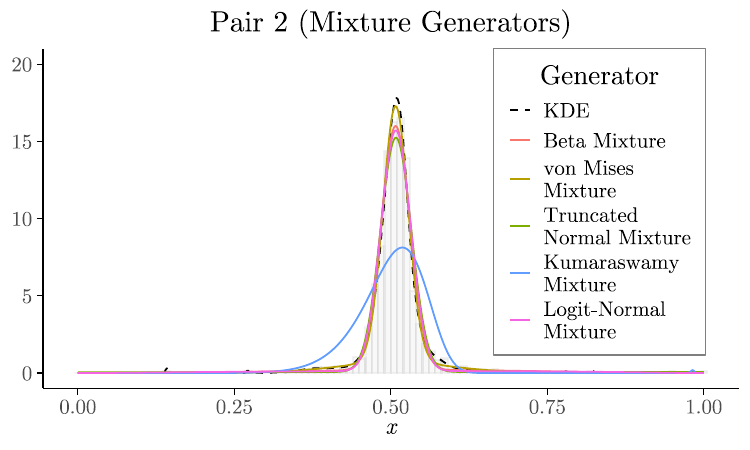}
    \includegraphics[width=0.49\linewidth]{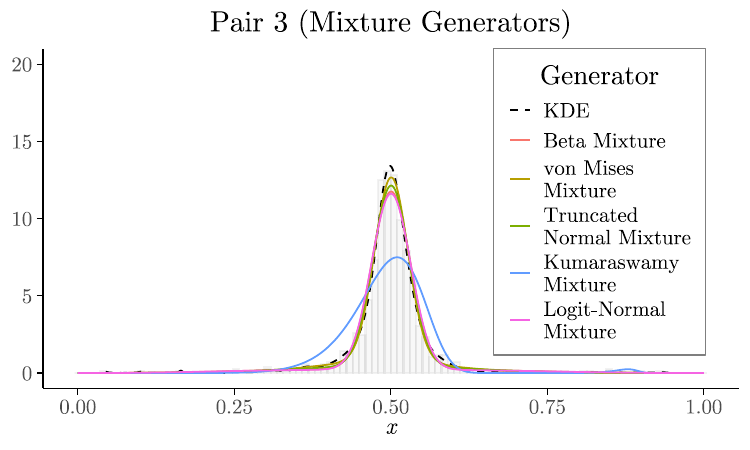}
    \caption{$2$-component mixture fits of the same parametric families overlaid on the histograms of the shifted wrapped differences for pairs 2 and 3.
    Corresponding plots for pairs 1 and 4 are given in \cref{fig:torus_gens_mix14}.
    }
    \label{fig:torus_gens_mix}
\end{figure}

\begin{table}[ht]
\centering
\resizebox{.94\linewidth}{!}{
\begin{tabular}{llcccc}
  \toprule
Generator & Estimated parameters & $\hat\rho$ & $\hat\tau$ & $\hat\xi$ & AIC \\ 
  \cmidrule(lr){1-1}\cmidrule(lr){2-2}\cmidrule(lr){3-5}\cmidrule(r){6-6}
\multicolumn{1}{@{}l}{\textit{Unimodal generators}} &  &  &  &  &  \\ 
  \quad $\mathrm{Beta}(\alpha,\beta)$ & $\hat\alpha = 93.69,\ \hat\beta = 91.13$ & 0.83 & 0.83 & 0.84 & -2874.99 \\ 
   \cmidrule(lr){1-6}
\quad $\mathrm{vonMises}(\phi_1,\phi_2)$ & $\hat\phi_1 = -17.19,\ \hat\phi_2 = -0.80$ & 0.82 & 0.83 & 0.84 & -2938.47 \\ 
   \cmidrule(lr){1-6}
\quad $\cTN_{[0,1]}(\mu,\sigma)$ & $\hat\mu = 0.51,\ \hat\sigma = 0.04$ & 0.83 & 0.83 & 0.84 & -2895.45 \\ 
   \cmidrule(lr){1-6}
\quad $\mathrm{Logit}\cN(\mu,\sigma)$ & $\hat\mu = 0.03,\ \hat\sigma = 0.15$ & 0.83 & 0.83 & 0.84 & -2869.76 \\ 
   \cmidrule(lr){1-6}
\quad $\mathrm{Kumaraswamy}(a,b)$ & $\hat a = 13.02,\ \hat b = 4641.32$ & 0.79 & 0.78 & 0.76 & -2768.56 \\ 
   \cmidrule(lr){1-6}
\addlinespace
\multicolumn{1}{@{}l}{\textit{Two-component mixtures}} &  &  &  &  &  \\ 
  \quad \makecell[l]{$\pi \cdot\mathrm{Beta}(\alpha_1,\beta_1) \,\, + $\\$(1-\pi)\cdot\,\mathrm{Beta}(\alpha_2,\beta_2)$} & \makecell[l]{$\hat\pi = 0.91\;\; \hat\alpha_1 = 241.51,\ \hat\beta_1 = 234.22$\\$\hat\alpha_2 = 4.09,\ \hat\beta_2 = 3.62$} & 0.84 & 0.83 & 0.80 & -3339.65 \\ 
   \cmidrule(lr){1-6}
\quad \makecell[l]{$\pi \cdot\mathrm{vonMises}(\phi_{11},\phi_{21}) \,\, + $\\$(1-\pi)\cdot\,\mathrm{vonMises}(\phi_{12},\phi_{22})$} & \makecell[l]{$\hat\pi = 0.16\;\; \hat\phi_{11} = -3.29,\ \hat\phi_{21} = -0.13$\\$\hat\phi_{12} = -62.60,\ \hat\phi_{22} = -3.00$} & 0.85 & 0.84 & 0.81 & \textbf{-3400.65} \\ 
   \cmidrule(lr){1-6}
\quad \makecell[l]{$\pi \cdot\cTN_{[0,1]}(\mu_1,\sigma_1) \,\, + $\\$(1-\pi)\cdot\,\cTN_{[0,1]}(\mu_2,\sigma_2)$} & \makecell[l]{$\hat\pi = 0.94\;\; \hat\mu_1 = 0.51,\ \hat\sigma_1 = 0.02$\\$\hat\mu_2 = 0.78,\ \hat\sigma_2 = 0.83$} & 0.83 & 0.81 & 0.79 & -3309.97 \\ 
   \cmidrule(lr){1-6}
\quad \makecell[l]{$\pi \cdot\mathrm{Logit}\cN(\mu_1,\sigma_1) \,\, + $\\$(1-\pi)\cdot\,\mathrm{Logit}\cN(\mu_2,\sigma_2)$} & \makecell[l]{$\hat\pi = 0.92\;\; \hat\mu_1 = 0.03,\ \hat\sigma_1 = 0.09$\\$\hat\mu_2 = 0.13,\ \hat\sigma_2 = 1.01$} & 0.83 & 0.82 & 0.80 & -3314.59 \\ 
   \cmidrule(lr){1-6}
\quad \makecell[l]{$\pi \cdot\mathrm{Kumaraswamy}(a_1,b_1) \,\, + $\\$(1-\pi)\cdot\,\mathrm{Kumaraswamy}(a_2,b_2)$} & \makecell[l]{$\hat\pi = 0.00\;\; \hat a_1 = 406.18,\ \hat b_1 = 1325.77$\\$\hat a_2 = 11.51,\ \hat b_2 = 1753.85$} & 0.77 & 0.76 & 0.74 & -2757.12 \\ 
   \cmidrule(lr){1-6}
\quad \makecell[l]{$\pi \cdot\mathrm{Beta}(\alpha_1,\beta_1) \,\, + $\\$(1-\pi)\cdot\,\mathrm{vonMises}(\phi_{12},\phi_{22})$} & \makecell[l]{$\hat\pi = 0.84\;\; \hat\alpha_1 = 313.25,\ \hat\beta_1 = 303.87$\\$\hat\phi_{12} = -3.31,\ \hat\phi_{22} = -0.14$} & 0.85 & 0.84 & 0.81 & -3400.25 \\ 
   \cmidrule(lr){1-6}
\quad \makecell[l]{$\pi \cdot\mathrm{Beta}(\alpha_1,\beta_1) \,\, + $\\$(1-\pi)\cdot\,\cTN_{[0,1]}(\mu_2,\sigma_2)$} & \makecell[l]{$\hat\pi = 0.94\;\; \hat\alpha_1 = 207.40,\ \hat\beta_1 = 200.72$\\$\hat\mu_2 = 18.77,\ \hat\sigma_2 = 46.86$} & 0.83 & 0.81 & 0.79 & -3304.86 \\ 
   \cmidrule(lr){1-6}
\quad \makecell[l]{$\pi \cdot\mathrm{Beta}(\alpha_1,\beta_1) \,\, + $\\$(1-\pi)\cdot\,\mathrm{Logit}\cN(\mu_2,\sigma_2)$} & \makecell[l]{$\hat\pi = 0.09\;\; \hat\alpha_1 = 4.04,\ \hat\beta_1 = 3.63$\\$\hat\mu_2 = 0.03,\ \hat\sigma_2 = 0.09$} & 0.84 & 0.82 & 0.80 & -3339.60 \\ 
   \cmidrule(lr){1-6}
\quad \makecell[l]{$\pi \cdot\mathrm{Beta}(\alpha_1,\beta_1) \,\, + $\\$(1-\pi)\cdot\,\mathrm{Kumaraswamy}(a_2,b_2)$} & \makecell[l]{$\hat\pi = 0.87\;\; \hat\alpha_1 = 157.56,\ \hat\beta_1 = 152.88$\\$\hat a_2 = 9.38,\ \hat b_2 = 288.42$} & 0.84 & 0.85 & 0.85 & -3138.14 \\ 
   \cmidrule(lr){1-6}
\quad \makecell[l]{$\pi \cdot\mathrm{vonMises}(\phi_{11},\phi_{21}) \,\, + $\\$(1-\pi)\cdot\,\cTN_{[0,1]}(\mu_2,\sigma_2)$} & \makecell[l]{$\hat\pi = 0.16\;\; \hat\phi_{11} = -2.51,\ \hat\phi_{21} = 0.06$\\$\hat\mu_2 = 0.51,\ \hat\sigma_2 = 0.02$} & 0.84 & 0.83 & 0.80 & -3393.61 \\ 
   \cmidrule(lr){1-6}
\quad \makecell[l]{$\pi \cdot\mathrm{vonMises}(\phi_{11},\phi_{21}) \,\, + $\\$(1-\pi)\cdot\,\mathrm{Logit}\cN(\mu_2,\sigma_2)$} & \makecell[l]{$\hat\pi = 0.16\;\; \hat\phi_{11} = -3.31,\ \hat\phi_{21} = -0.14$\\$\hat\mu_2 = 0.03,\ \hat\sigma_2 = 0.08$} & 0.85 & 0.84 & 0.81 & -3400.20 \\ 
   \cmidrule(lr){1-6}
\quad \makecell[l]{$\pi \cdot\mathrm{vonMises}(\phi_{11},\phi_{21}) \,\, + $\\$(1-\pi)\cdot\,\mathrm{Kumaraswamy}(a_2,b_2)$} & \makecell[l]{$\hat\pi = 0.00\;\; \hat\phi_{11} = 24.86,\ \hat\phi_{21} = -19.56$\\$\hat a_2 = 12.21,\ \hat b_2 = 2775.85$} & 0.77 & 0.76 & 0.74 & -2906.81 \\ 
   \cmidrule(lr){1-6}
\quad \makecell[l]{$\pi \cdot\cTN_{[0,1]}(\mu_1,\sigma_1) \,\, + $\\$(1-\pi)\cdot\,\mathrm{Logit}\cN(\mu_2,\sigma_2)$} & \makecell[l]{$\hat\pi = 0.13\;\; \hat\mu_1 = 0.51,\ \hat\sigma_1 = 0.12$\\$\hat\mu_2 = 0.03,\ \hat\sigma_2 = 0.08$} & 0.85 & 0.83 & 0.81 & -3379.95 \\ 
   \cmidrule(lr){1-6}
\quad \makecell[l]{$\pi \cdot\cTN_{[0,1]}(\mu_1,\sigma_1) \,\, + $\\$(1-\pi)\cdot\,\mathrm{Kumaraswamy}(a_2,b_2)$} & \makecell[l]{$\hat\pi = 0.06\;\; \hat\mu_1 = 3523.62,\ \hat\sigma_1 = 22022.81$\\$\hat a_2 = 21.60,\ \hat b_2 = 1566268.31$} & 0.81 & 0.79 & 0.75 & -3234.74 \\ 
   \cmidrule(lr){1-6}
\quad \makecell[l]{$\pi \cdot\mathrm{Logit}\cN(\mu_1,\sigma_1) \,\, + $\\$(1-\pi)\cdot\,\mathrm{Kumaraswamy}(a_2,b_2)$} & \makecell[l]{$\hat\pi = 0.91\;\; \hat\mu_1 = 0.03,\ \hat\sigma_1 = 0.09$\\$\hat a_2 = 2.60,\ \hat b_2 = 3.21$} & 0.83 & 0.82 & 0.80 & -3333.90 \\ 
   \cmidrule(lr){1-6}
\textit{Kernel density estimate} & — & 0.85 & 0.84 & 0.81 & — \\ 
   \bottomrule
\end{tabular}
}
\caption{Maximum likelihood fits of generator models for pair 2. As in \cref{tab:fit_v1}, we report parameter estimates, plug-in estimates for $\rho, \tau$, and $\xi$ (sample estimates: $\hat{\rho}_n = 0.86$, $\hat{\tau}_n = 0.83$, $\hat{\xi}_n = 0.79$), and AIC for the parametric generators. The two-component von Mises mixture attains the lowest AIC.} 
\label{tab:fit_v2}
\end{table}

\begin{table}[ht]
\centering
\resizebox{.94\linewidth}{!}{
\begin{tabular}{llcccc}
  \toprule
Generator & Estimated parameters & $\hat\rho$ & $\hat\tau$ & $\hat\xi$ & AIC \\ 
  \cmidrule(lr){1-1}\cmidrule(lr){2-2}\cmidrule(lr){3-5}\cmidrule(r){6-6}
\multicolumn{1}{@{}l}{\textit{Unimodal generators}} &  &  &  &  &  \\ 
  \quad $\mathrm{Beta}(\alpha,\beta)$ & $\hat\alpha = 37.98,\ \hat\beta = 38.13$ & 0.75 & 0.71 & 0.65 & -2193.89 \\ 
   \cmidrule(lr){1-6}
\quad $\mathrm{vonMises}(\phi_1,\phi_2)$ & $\hat\phi_1 = -8.54,\ \hat\phi_2 = 0.08$ & 0.75 & 0.72 & 0.66 & -2446.82 \\ 
   \cmidrule(lr){1-6}
\quad $\cTN_{[0,1]}(\mu,\sigma)$ & $\hat\mu = 0.50,\ \hat\sigma = 0.06$ & 0.73 & 0.70 & 0.63 & -2289.10 \\ 
   \cmidrule(lr){1-6}
\quad $\mathrm{Logit}\cN(\mu,\sigma)$ & $\hat\mu = -0.00,\ \hat\sigma = 0.23$ & 0.74 & 0.71 & 0.65 & -2167.77 \\ 
   \cmidrule(lr){1-6}
\quad $\mathrm{Kumaraswamy}(a,b)$ & $\hat a = 10.67,\ \hat b = 1157.06$ & 0.76 & 0.74 & 0.71 & -2144.80 \\ 
   \cmidrule(lr){1-6}
\addlinespace
\multicolumn{1}{@{}l}{\textit{Two-component mixtures}} &  &  &  &  &  \\ 
  \quad \makecell[l]{$\pi \cdot\mathrm{Beta}(\alpha_1,\beta_1) \,\, + $\\$(1-\pi)\cdot\,\mathrm{Beta}(\alpha_2,\beta_2)$} & \makecell[l]{$\hat\pi = 0.88\;\; \hat\alpha_1 = 135.02,\ \hat\beta_1 = 134.80$\\$\hat\alpha_2 = 3.98,\ \hat\beta_2 = 4.15$} & 0.80 & 0.78 & 0.74 & -2803.02 \\ 
   \cmidrule(lr){1-6}
\quad \makecell[l]{$\pi \cdot\mathrm{vonMises}(\phi_{11},\phi_{21}) \,\, + $\\$(1-\pi)\cdot\,\mathrm{vonMises}(\phi_{12},\phi_{22})$} & \makecell[l]{$\hat\pi = 0.21\;\; \hat\phi_{11} = -2.48,\ \hat\phi_{21} = 0.20$\\$\hat\phi_{12} = -36.25,\ \hat\phi_{22} = -0.17$} & 0.81 & 0.78 & 0.74 & \textbf{-2832.42} \\ 
   \cmidrule(lr){1-6}
\quad \makecell[l]{$\pi \cdot\cTN_{[0,1]}(\mu_1,\sigma_1) \,\, + $\\$(1-\pi)\cdot\,\cTN_{[0,1]}(\mu_2,\sigma_2)$} & \makecell[l]{$\hat\pi = 0.16\;\; \hat\mu_1 = 0.49,\ \hat\sigma_1 = 0.14$\\$\hat\mu_2 = 0.50,\ \hat\sigma_2 = 0.03$} & 0.80 & 0.78 & 0.74 & -2816.48 \\ 
   \cmidrule(lr){1-6}
\quad \makecell[l]{$\pi \cdot\mathrm{Logit}\cN(\mu_1,\sigma_1) \,\, + $\\$(1-\pi)\cdot\,\mathrm{Logit}\cN(\mu_2,\sigma_2)$} & \makecell[l]{$\hat\pi = 0.89\;\; \hat\mu_1 = 0.00,\ \hat\sigma_1 = 0.12$\\$\hat\mu_2 = -0.05,\ \hat\sigma_2 = 0.80$} & 0.80 & 0.78 & 0.74 & -2797.15 \\ 
   \cmidrule(lr){1-6}
\quad \makecell[l]{$\pi \cdot\mathrm{Kumaraswamy}(a_1,b_1) \,\, + $\\$(1-\pi)\cdot\,\mathrm{Kumaraswamy}(a_2,b_2)$} & \makecell[l]{$\hat\pi = 0.99\;\; \hat a_1 = 10.56,\ \hat b_1 = 1105.19$\\$\hat a_2 = 50.98,\ \hat b_2 = 712.67$} & 0.74 & 0.72 & 0.68 & -2472.90 \\ 
   \cmidrule(lr){1-6}
\quad \makecell[l]{$\pi \cdot\mathrm{Beta}(\alpha_1,\beta_1) \,\, + $\\$(1-\pi)\cdot\,\mathrm{vonMises}(\phi_{12},\phi_{22})$} & \makecell[l]{$\hat\pi = 0.79\;\; \hat\alpha_1 = 179.76,\ \hat\beta_1 = 179.22$\\$\hat\phi_{12} = -2.52,\ \hat\phi_{22} = 0.20$} & 0.81 & 0.78 & 0.74 & -2831.54 \\ 
   \cmidrule(lr){1-6}
\quad \makecell[l]{$\pi \cdot\mathrm{Beta}(\alpha_1,\beta_1) \,\, + $\\$(1-\pi)\cdot\,\cTN_{[0,1]}(\mu_2,\sigma_2)$} & \makecell[l]{$\hat\pi = 0.93\;\; \hat\alpha_1 = 105.58,\ \hat\beta_1 = 105.53$\\$\hat\mu_2 = -6.65,\ \hat\sigma_2 = 5.21$} & 0.78 & 0.77 & 0.74 & -2762.63 \\ 
   \cmidrule(lr){1-6}
\quad \makecell[l]{$\pi \cdot\mathrm{Beta}(\alpha_1,\beta_1) \,\, + $\\$(1-\pi)\cdot\,\mathrm{Logit}\cN(\mu_2,\sigma_2)$} & \makecell[l]{$\hat\pi = 0.12\;\; \hat\alpha_1 = 3.99,\ \hat\beta_1 = 4.17$\\$\hat\mu_2 = 0.00,\ \hat\sigma_2 = 0.12$} & 0.80 & 0.78 & 0.74 & -2802.73 \\ 
   \cmidrule(lr){1-6}
\quad \makecell[l]{$\pi \cdot\mathrm{Beta}(\alpha_1,\beta_1) \,\, + $\\$(1-\pi)\cdot\,\mathrm{Kumaraswamy}(a_2,b_2)$} & \makecell[l]{$\hat\pi = 0.96\;\; \hat\alpha_1 = 74.75,\ \hat\beta_1 = 73.97$\\$\hat a_2 = 5.40,\ \hat b_2 = 254.91$} & 0.79 & 0.79 & 0.78 & -2372.34 \\ 
   \cmidrule(lr){1-6}
\quad \makecell[l]{$\pi \cdot\mathrm{vonMises}(\phi_{11},\phi_{21}) \,\, + $\\$(1-\pi)\cdot\,\cTN_{[0,1]}(\mu_2,\sigma_2)$} & \makecell[l]{$\hat\pi = 0.21\;\; \hat\phi_{11} = -2.52,\ \hat\phi_{21} = 0.47$\\$\hat\mu_2 = 0.50,\ \hat\sigma_2 = 0.03$} & 0.80 & 0.78 & 0.74 & -2828.41 \\ 
   \cmidrule(lr){1-6}
\quad \makecell[l]{$\pi \cdot\mathrm{vonMises}(\phi_{11},\phi_{21}) \,\, + $\\$(1-\pi)\cdot\,\mathrm{Logit}\cN(\mu_2,\sigma_2)$} & \makecell[l]{$\hat\pi = 0.21\;\; \hat\phi_{11} = -2.53,\ \hat\phi_{21} = 0.20$\\$\hat\mu_2 = 0.00,\ \hat\sigma_2 = 0.11$} & 0.81 & 0.78 & 0.74 & -2831.43 \\ 
   \cmidrule(lr){1-6}
\quad \makecell[l]{$\pi \cdot\mathrm{vonMises}(\phi_{11},\phi_{21}) \,\, + $\\$(1-\pi)\cdot\,\mathrm{Kumaraswamy}(a_2,b_2)$} & \makecell[l]{$\hat\pi = 0.18\;\; \hat\phi_{11} = -2.49,\ \hat\phi_{21} = -0.14$\\$\hat a_2 = 18.95,\ \hat b_2 = 342874.95$} & 0.80 & 0.77 & 0.72 & -2800.18 \\ 
   \cmidrule(lr){1-6}
\quad \makecell[l]{$\pi \cdot\cTN_{[0,1]}(\mu_1,\sigma_1) \,\, + $\\$(1-\pi)\cdot\,\mathrm{Logit}\cN(\mu_2,\sigma_2)$} & \makecell[l]{$\hat\pi = 0.16\;\; \hat\mu_1 = 0.49,\ \hat\sigma_1 = 0.14$\\$\hat\mu_2 = 0.00,\ \hat\sigma_2 = 0.11$} & 0.80 & 0.78 & 0.74 & -2815.68 \\ 
   \cmidrule(lr){1-6}
\quad \makecell[l]{$\pi \cdot\cTN_{[0,1]}(\mu_1,\sigma_1) \,\, + $\\$(1-\pi)\cdot\,\mathrm{Kumaraswamy}(a_2,b_2)$} & \makecell[l]{$\hat\pi = 0.06\;\; \hat\mu_1 = -1886.75,\ \hat\sigma_1 = 618.82$\\$\hat a_2 = 15.30,\ \hat b_2 = 27397.34$} & 0.78 & 0.75 & 0.70 & -2732.35 \\ 
   \cmidrule(lr){1-6}
\quad \makecell[l]{$\pi \cdot\mathrm{Logit}\cN(\mu_1,\sigma_1) \,\, + $\\$(1-\pi)\cdot\,\mathrm{Kumaraswamy}(a_2,b_2)$} & \makecell[l]{$\hat\pi = 0.88\;\; \hat\mu_1 = 0.00,\ \hat\sigma_1 = 0.12$\\$\hat a_2 = 2.80,\ \hat b_2 = 4.67$} & 0.80 & 0.78 & 0.74 & -2800.84 \\ 
   \cmidrule(lr){1-6}
\textit{Kernel density estimate} & — & 0.81 & 0.78 & 0.71 & — \\ 
   \bottomrule
\end{tabular}
}
\caption{Maximum likelihood fits of generator models for pair 3. As in \cref{tab:fit_v1}, we report parameter estimates, plug-in estimates for $\rho, \tau$, and $\xi$ (sample estimates: $\hat{\rho}_n = 0.84$, $\hat{\tau}_n = 0.80$, $\hat{\xi}_n = 0.73$), and AIC for the parametric generators. The two-component von Mises mixture attains the lowest AIC.}
\label{tab:fit_v3}
\end{table}

\begin{figure}[ht]
    \centering
    \includegraphics[width=0.32\textwidth]{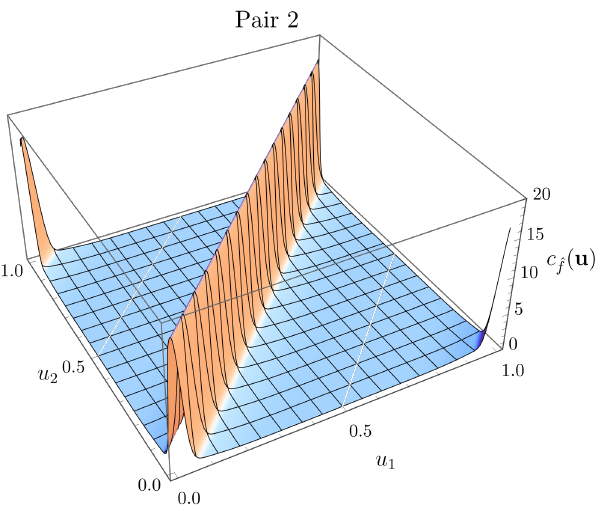}
    \includegraphics[width=0.32\textwidth]{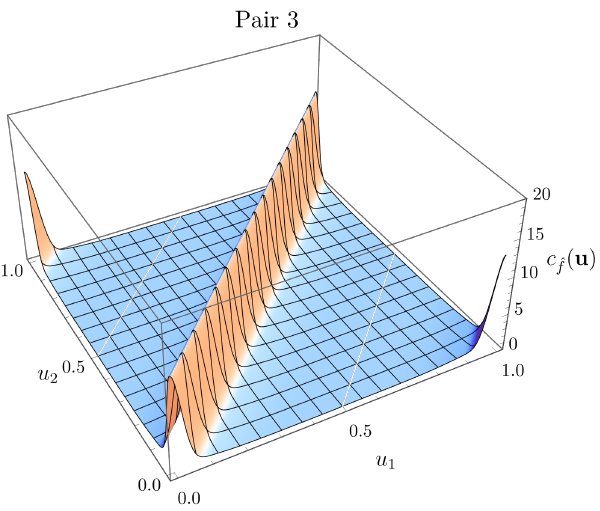}
    \caption{Fitted copula densities for pairs 2 and 3 using the best-fitting von Mises mixture generators from \cref{fig:torus_gens_mix}. These copulas recover both symmetry and asymmetries relative to the main diagonal suggested by the data. Corresponding densities for pairs 1 and 4 are given in \cref{fig:torus_vonMisesmixes14}.}
    \label{fig:torus_vonMisesmixes}
\end{figure}

\begin{figure}[ht]
    \centering
    \includegraphics[width=0.32\textwidth]{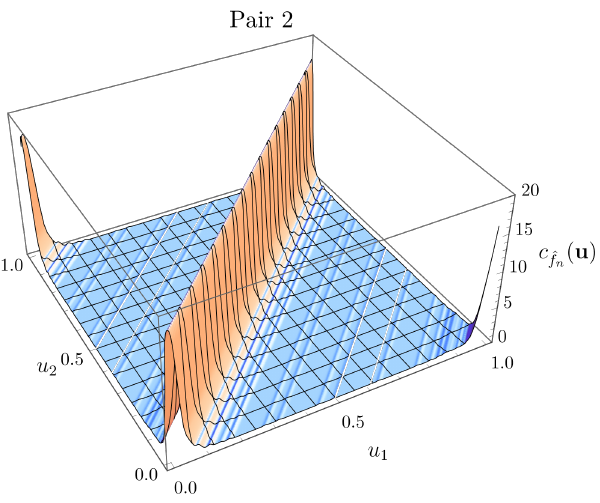} 
    \includegraphics[width=0.32\textwidth]{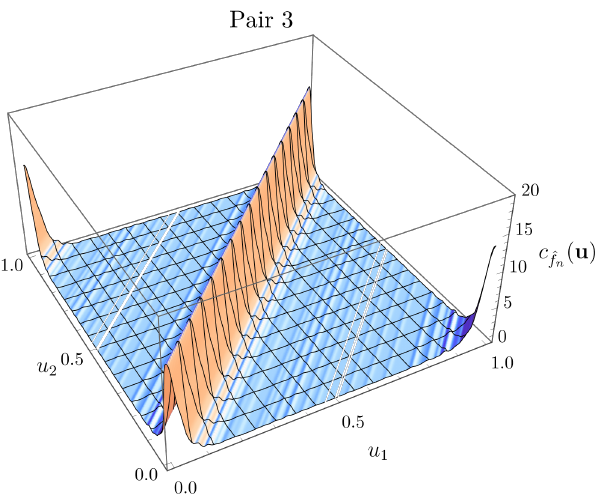}
    \caption{Fitted copula densities using KDE-based generators for pairs 2 and 3. Relative to \cref{fig:torus_vonMisesmixes}, KDE-based copulas capture diagonal asymmetry at least as well. Corresponding densities for pairs 1 and 4 are given in \cref{fig:torus_KDEs14}.}
    \label{fig:torus_KDEs}
\end{figure}

The densities of the fitted copulas themselves based on the fitted von Mises mixture generators are shown in \cref{fig:torus_vonMisesmixes} and based on the fitted KDEs in \cref{fig:torus_KDEs}. We note that in contrast to the copula based on the unimodal parametric generators, both the copulas based on the KDE generator --- and to a lesser extent the copula based on the $2$-component mixtures --- accurately recover some asymmetry with respect to the diagonal $[0,1]^2$. While our analysis assumes iid data, these results do capture what one would intuitively expect if neighboring recording sites (i.e., adjacent channels) were being activated in a temporal sequence during the performance of the memory task. In addition, the density estimate for pair 3 appears to be more symmetric and contains more isolated points off the diagonal than do the other pairs. The physiological significance of this difference is uncertain; it may simply be noise, or it may reflect special features of that particular pair, such as reception of inputs from elsewhere in the hippocampus or the rest of the brain, or perhaps a regulatory role yet to be determined.

It is also interesting to compare the values of Spearman's rho, Kendall's tau, and the Dette--Siburg--Stoimenov coefficient implied by our fitted copula models to their nonparametric counterparts. For each channel pair, we computed the values of all three measures from the best-fitting generators (in each case, the $2$-component von Mises mixture) via numerical integration. All of these computed values agreed closely with their sample versions (shown in the captions below \cref{tab:fit_v1,tab:fit_v2,tab:fit_v3,tab:fit_v4}), differing by no more than a few hundredths. The KDE-based copulas performed similarly well. These results demonstrate that our family is sufficiently flexible to capture the observed CA3 dependence structure, and illustrate the practical value of our framework in domains where circular dependence is of central interest. In particular, the consistently high values of all three coefficients --- across both parametric and nonparametric fits --- independently corroborate the presence of strong positive dependence among the neural phase signals.

\section{Discussion}\label{sec:discussion}

In this work, we introduce a new class $\cC$ of absolutely continuous copulas generated by univariate densities on $[0,1]$. Despite the structural rigidity of the construction, the copula family is surprisingly rich: it includes both simple and highly irregular copulas, admits transparent and easily computed formulas for standard concordance measures, and allows for straightforward statistical inference in both the parametric and nonparametric regimes. In the parametric regime, likelihood-based procedures and results for a copula $C^{\bs}_{f_{\theta}} \in \cC$ are directly inherited from those for the univariate model $f_\theta$. In the nonparametric regime, the existence of partial derivatives of $C^{\bs}_{f}$ allows for desirable large-sample properties for copula estimators constructed from kernel density estimates of the generator $f$. Our simulations show that both the signature and generator can be recovered reliably, and the neural connectivity application of \citet{kleinTorusGraphsMultivariate2020} illustrates that in practice the family can capture bivariate rotational dependence while remaining interpretable and computationally straightforward.

Beyond specific examples, our results show explicitly how the generator $f$ and signature $\bs$ jointly control dependence in the induced copula $C^{\bs}_f$: the signature acts as a pattern of componentwise reflections that fixes the direction of concordance (toggling between mainly positive versus negative association in the bivariate case), while the generator governs the strength and shape of dependence. Specifically, the concentration, symmetry, and modality of the generator translate directly into the magnitude of concordance and into the conditional behaviour that produces the characteristic rotational patterns: given any $d-1$ components of $\bU \sim C^{\bs}_f$, the remaining component (possibly rotated) is distributed according to $f$, so the qualitative features of $f$ propagate directly to $C^{\bs}_f$. A similar mechanism underlies our approach to signature selection: with the wrong signature $\bt \neq \bs$, the wrapped sum $\bigoplus_{j=1}^d (-1)^{t_j}U_j$ is uniform, yielding a powerful diagnostic tool. These features show that qualitative and quantitative aspects of dependence can be read directly from properties of the generator. As a counterpoint, the pathological construction  of $C^{\bs}_{f^*}$ in \cref{sub:univariate} --- which is $(d-1)$-times continuously differentiable everywhere despite violating virtually every regularity or boundedness principle one might reasonably expect of a copula --- demonstrates that in $\cC$, the extreme irregularity of a copula density can coexist with well-behaved copula-level properties.

\subsection{Extensions}

There are many worthwhile avenues to extend this work further. On the methodological side, although we provide closed-form expressions for standard concordance measures in terms of the generator, a more systematic catalogue of such results (including Blomqvist's beta, Gini's coefficient, and Spearman's footrule, among others; \citealp[see, e.g.,][Section 5.1.4]{nelsenIntroductionCopulas2006}) would be valuable. From an applied perspective, mixture constructions or embeddings within vine copula frameworks \citep{aas2009pair,joe2014dependence,czadoAnalyzingDependentData2019} could help extend the practical reach of the family to higher dimensions and more heterogeneous forms of dependence.

In the following subsections, we describe in more detail two additional future directions which illustrate how one might leverage the algebraic and geometric foundations of our construction to expand the family while still preserving its interpretability and richness.

\subsubsection{Hierarchical constructions via multivariate generators}\label{subsub:hierarchicalextension}

Because of its rather rigid dependence structures (e.g., \cref{prop:basicproperties}), the current version of our copula family may not be practically applicable beyond the trivariate case. However, we note here a natural and potentially far-reaching extension of the family: when choosing a generator, instead of a univariate density in $\cF_{[0,1]}$, we may consider a multivariate density $f$ on $[0,1]^{d'}$ and construct a copula density on $[0,1]^{d}$ via composition, where $d \geq 2d' \geq 2$. To describe this construction, we write $\mathfrak{A}(d,d')$ for the set of ordered partitions $\cA = \{A_1,\ldots,A_{d'}\}$ of $\{1,\ldots,d\}$ into $d'$ disjoint subsets with each $|A_j| \geq 2$.\footnote{There are exactly $d'! \sum_{j=0}^{d'} (-1)^j \binom{d}{k} S(d-j, d'-j)$ such ordered partitions, where $S(n,k)$ is a Stirling number of the second kind. When any $|A_j| = 1$, the construction will not yield a valid copula unless the $j$th marginal distribution of $f$ is $\stdunif$.} Let $\cA \in \mathfrak{A}(d,d')$, and let $\bs^{(j)} \in \{0,1\}^{|A_j|}$ be a slot-specific signature for the $j$th argument of $f$. Define the signature tuple $\bs := (\bs^{(1)}, \ldots, \bs^{(d')})$ and for each $j \in \{1,\ldots,d'\}$ and each index $k \in A_j$, let
\[
    \tu_k^{(j)} := u_k^{1 - s_k^{(j)}} (1 - u_k)^{s_k^{(j)}}
\]
and define the $d$-dimensional copula $C^{\cA, \bs}_f$ through its density
\[
    c^{\cA, \bs}_f(\bu) := f\left( \bigoplus_{k \in A_1} \tu_k^{(1)}, \ldots,  \bigoplus_{k \in A_{d'}} \tu_k^{(d')} \right).
\]
As in our original construction, this yields a valid copula density on $[0,1]^{d}$ regardless of the multivariate generator $f$. Our original family $\cC$ then arises as a special case when $d' = 1$, $A_1 = \{1,\ldots,d\}$, and $f$ is univariate.

This extension preserves the essential idea of using wrapped sums to induce dependence, while offering a much richer class of constructions. Each argument of $f$ may now depend on an arbitrary subset of the inputs, with signatures applied independently to each subset. Importantly, this family is closed under recursive application: we may take the multivariate generator $f$ itself to be a copula density defined via the original construction, thereby producing a new copula by using one defined copula as the generator of another. To formalize the idea, consider any positive integer sequence $\{d_m\}$ with $d_0 = d$ and $d_m \geq 2d_{m-1}$ for each $m \geq 1$, define  $\cC^{(0)} := \cC$ and
\[
    \cC^{(m)} := \left\{C^{\cA, \bs}_{c'}: C' \in \cC^{(m-1)}, \mbox{$\cA \in \mathfrak{A}(d_{m},d_{m-1})$}, \mbox{$\bs = (\bs^{(1)}, \ldots, \bs^{(d_{m-1})})$ is a signature tuple} \right\}.
\]
This recursive structure bears analogy to pair-copula constructions \citep{aas2009pair}, although the mechanism here is algebraic rather than conditional; dependence arises through componentwise wrapped sums and signature-based flips, rather than through conditioning or pseudo-observations. The result is a flexible, interpretable, and recursively extensible framework for copula construction. We leave the theoretical and practical development of the enlarged families $\cC^{(m)}$ for future work.

\subsubsection{Isometries on the circle group}\label{subsub:isometries}

An additional perspective on our framework recognizes the underlying group structure of the involved transformations. The operator $\oplus$ is addition in the circle group $\T$, and the binary signature element $s_j \in \{0,1\}$ corresponds to choosing for each input $u_j$ one of two specific isometries on $\T$: the identity $u_j \mapsto u_j$ or the reflection $u_j \mapsto -u_j \bmod{1}$. This viewpoint suggests another generalization. The isometry group on $\T$ equipped with the metric $d_\T(x,y) := \min\{|x-y|, 1 - |x-y|\}$ consists of both reflections and rotations. A complete generalization would therefore involve applying both a discrete reflection $\sigma_j \in \{-1,1\}$ and a continuous rotation $\phi_j \in [0,1)$ to each variable, yielding the transformation  $u_j \mapsto  (\sigma_j u_j + \phi_j) \bmod{1}$. 

This generalization defines a valid and richer class of copulas, which are conveniently parameterized by moving to a more natural representation of the circle group via the group isomorphism $z(u) := e^{2 \pi \imi u}$. This map connects the additive group $(\T, \oplus)$ to the multiplicative group $(S^1, \times)$, where $S^1 = \{z \in \C: |z| = 1\}$ is the unit circle in $\C$ and $\times$ denotes standard multiplication. In the complex representation, the transformation becomes clearer: rotations in $\T$ correspond to multiplication by a unit complex number $\psi_j = e^{2 \pi \imi \phi_j} \in S^1$, reflections in $\T$ correspond to complex conjugation $z_j \mapsto \bar{z}_j = z_j^{-1}$, and wrapped sums in $\T$ correspond to standard multiplication in $S^1$. Thus, the entire generalized construction can be reformulated on $S^1$, and the generalized copula density is given by
\[
    c_f^{\bsigma, \bpsi}(\bu)
    :=
    f \left( \frac{1}{2\pi}  \arg \left( \prod_{j=1}^d z(u_j)^{\sigma_j} \psi_j \right)\right),
\]
where $\bsigma = (\sigma_1, \ldots, \sigma_d) \in \{-1,1\}^d$, $\bpsi = (\psi_1, \ldots, \psi_d) \in (S^1)^d$, and $\arg{z}$ denotes the principal value of $z \in \C$ in the interval $[0, 2\pi)$. This reparameterization explicitly connects our copula family to the domain of Fourier analysis and signal processing, where transformations of phase and orientation are of fundamental importance.

\section*{Acknowledgments}

We are grateful to Yanbo Tang, Amy Goldwater, Radu Craiu and Stanislav Volgushev for their careful reading of an early unpublished version of this paper \citep{lalancetteNewFamilySmooth2022} and their constructive suggestions. Johanna Ne{\v{s}}lehov{\'a} provided insightful feedback that shaped an intermediate revision which appeared in \citet{zimmermanCopulasNewTheory2025}. A question by Stanislav Volgushev inspired the extension presented in \cref{subsub:hierarchicalextension}, and Amy Goldwater proofread a later version of the manuscript. ML gratefully acknowledges funding from the Natural Sciences and Engineering Research Council of Canada through a Postdoctoral Fellowship and Discovery Grant.

\FloatBarrier


\newpage
\appendix

\section{Additional plots and tables for Section~\ref{sub:neural}}\label{app:additionalplots}

\begin{figure}[ht]
    \centering
    \includegraphics[width=0.49\linewidth]{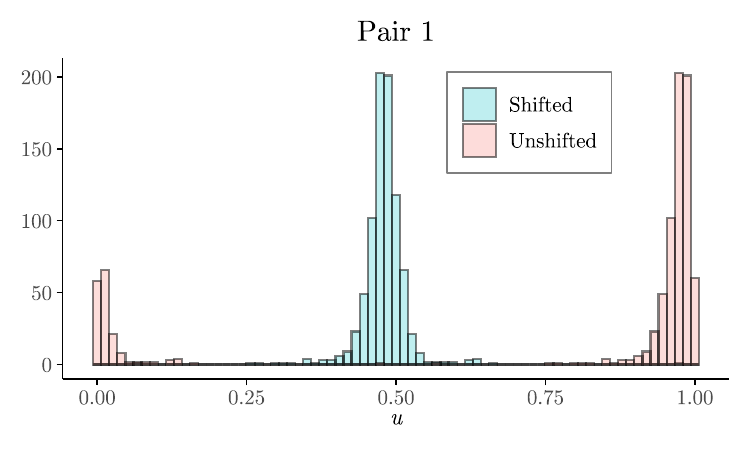}
    \includegraphics[width=0.49\linewidth]{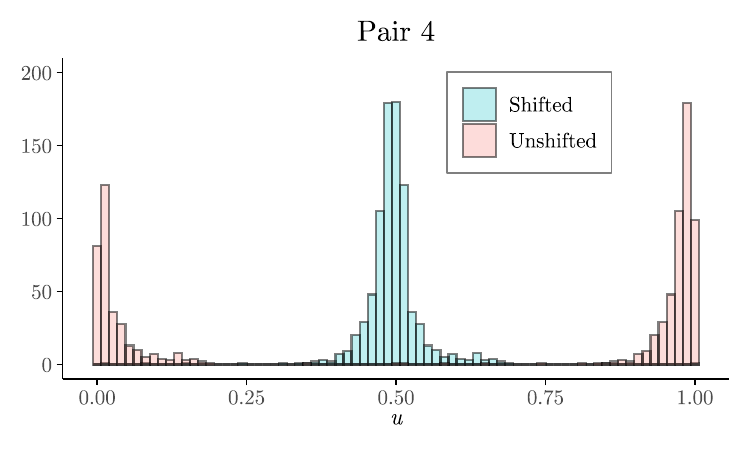}
    \caption{Histograms of wrapped differences for pairs 1 and 4 before (red) and after (blue) shifting by $+1/2 \bmod{1}$.}
    \label{fig:torus_histograms14}
\end{figure}

\begin{figure}[ht]
    \centering
    \includegraphics[width=0.49\linewidth]{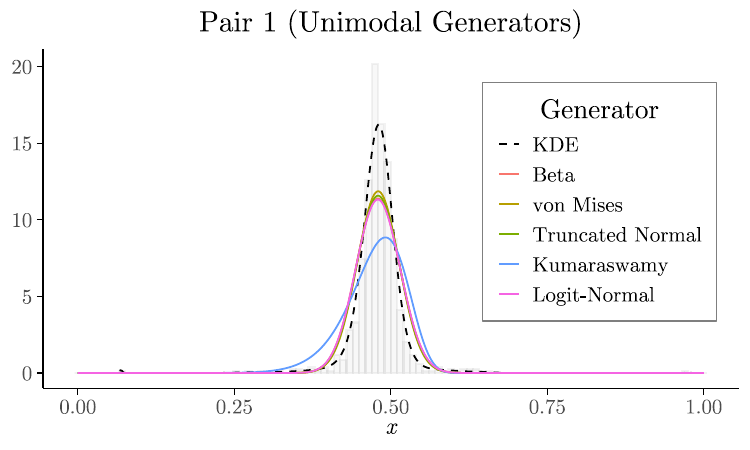}
    \includegraphics[width=0.49\linewidth]{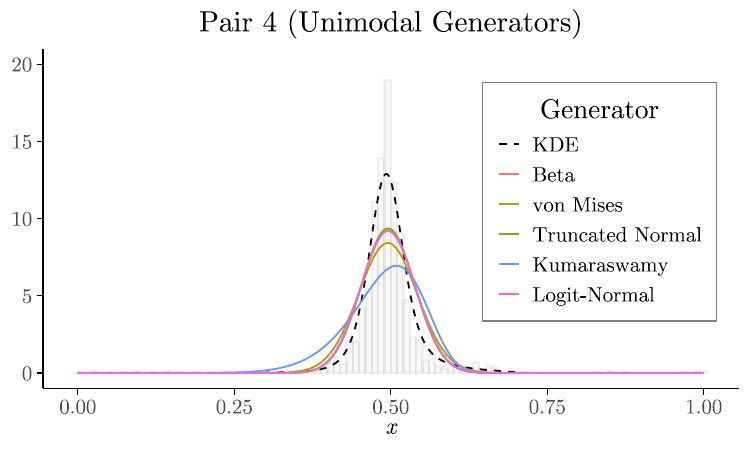}
    \caption{Unimodal parametric fits of the generator (beta, truncated normal, Kumaraswamy, logit-normal, and von Mises) overlaid on the shifted wrapped difference histograms for pairs 1 and 4.}
    \label{fig:torus_gens_uni14}
\end{figure}

\begin{figure}[ht]
    \centering
    \includegraphics[width=0.49\linewidth]{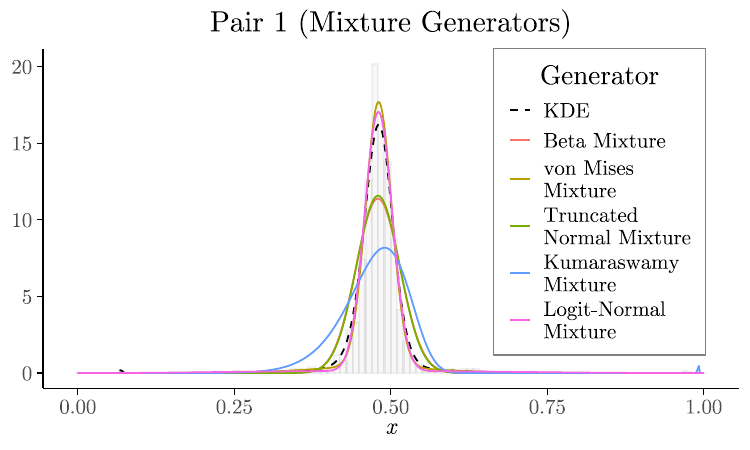}
    \includegraphics[width=0.49\linewidth]{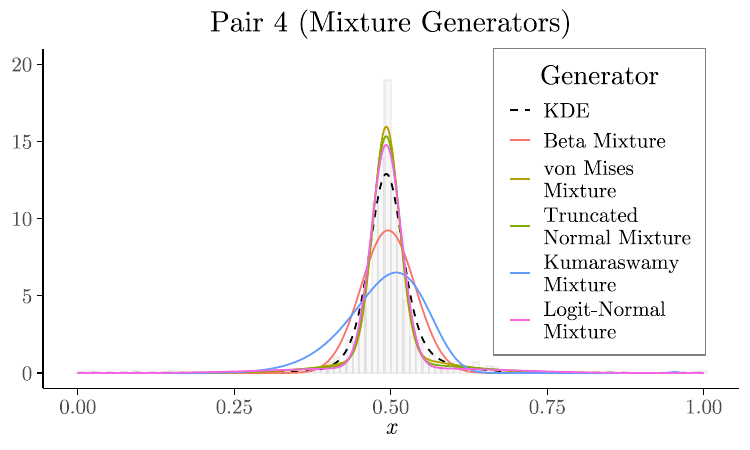}
    \caption{$2$-component mixture fits of the same parametric families overlaid on the histograms of the shifted wrapped differences for pairs 1 and 4.}
    \label{fig:torus_gens_mix14}
\end{figure}

\begin{figure}[ht]
    \centering
    \includegraphics[width=0.32\textwidth]{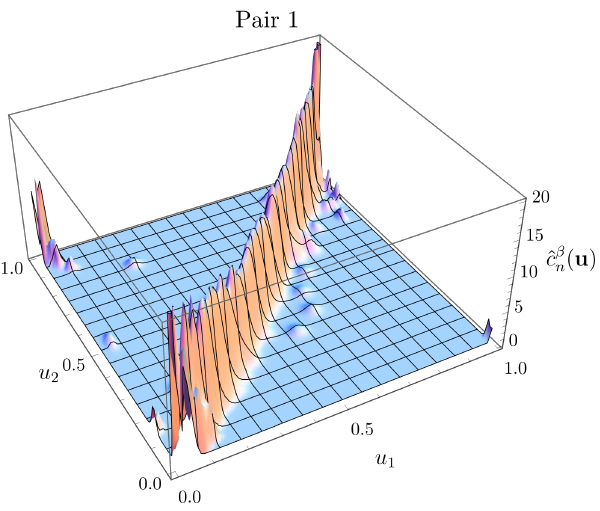}
    \includegraphics[width=0.32\textwidth]{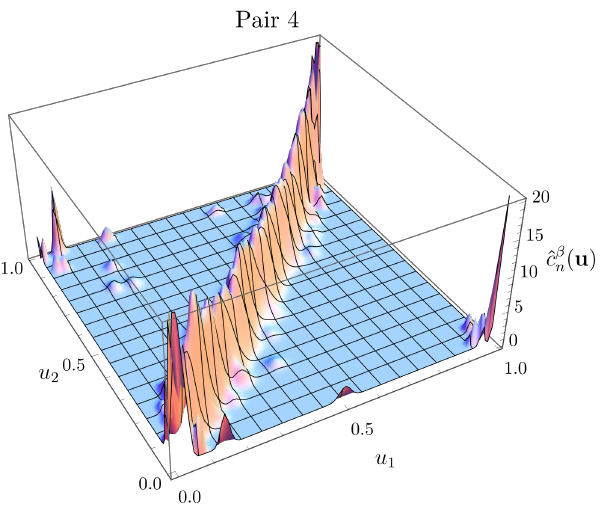}
    \caption{Empirical beta copula densities for CA3 channel pairs 1 and 4.}
    \label{fig:torus_empiricalbetas14}
\end{figure}

\begin{figure}[ht]
    \centering
    \includegraphics[width=0.32\textwidth]{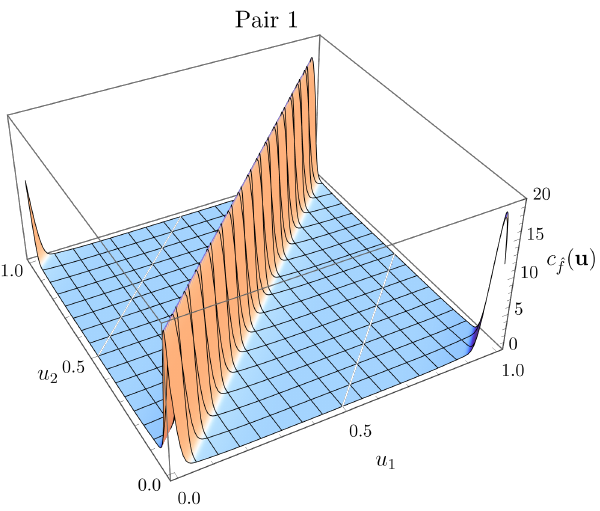}
    \includegraphics[width=0.32\textwidth]{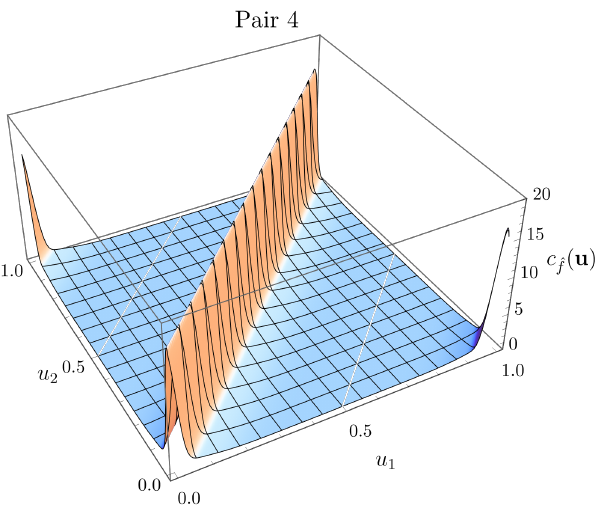}
    \caption{Fitted copula densities using the 2-component von Mises generator for pairs 1 and 4.}
    \label{fig:torus_vonMisesmixes14}
\end{figure}

\begin{figure}[ht]
    \centering
    \includegraphics[width=0.32\textwidth]{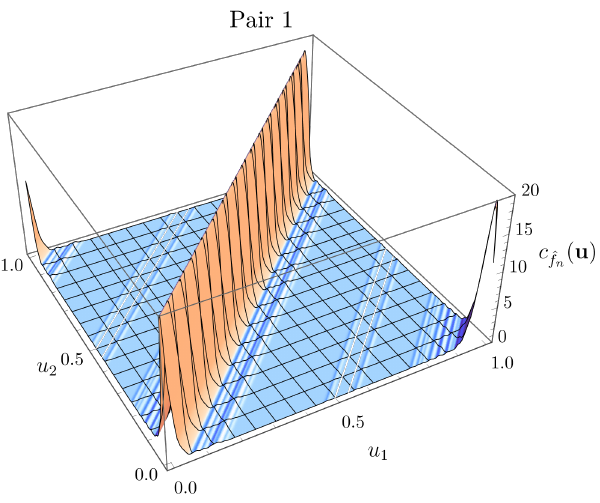}
    \includegraphics[width=0.32\textwidth]{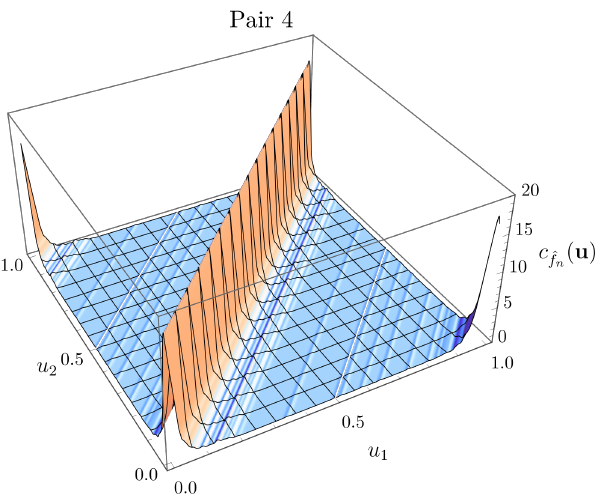}
    \caption{Fitted copula densities using the KDE generator for pairs 1 and 4.}
    \label{fig:torus_KDEs14}
\end{figure}

\begin{table}[ht]
\centering
\resizebox{.94\linewidth}{!}{
\begin{tabular}{llcccc}
  \toprule
Generator & Estimated parameters & $\hat\rho$ & $\hat\tau$ & $\hat\xi$ & AIC \\ 
  \cmidrule(lr){1-1}\cmidrule(lr){2-2}\cmidrule(lr){3-5}\cmidrule(r){6-6}
\multicolumn{1}{@{}l}{\textit{Unimodal generators}} &  &  &  &  &  \\ 
  \quad $\mathrm{Beta}(\alpha,\beta)$ & $\hat\alpha = 97.74,\ \hat\beta = 106.05$ & 0.81 & 0.83 & 0.87 & -3130.26 \\ 
   \cmidrule(lr){1-6}
\quad $\mathrm{vonMises}(\phi_1,\phi_2)$ & $\hat\phi_1 = -22.47,\ \hat\phi_2 = 2.87$ & 0.82 & 0.84 & 0.88 & -3183.10 \\ 
   \cmidrule(lr){1-6}
\quad $\cTN_{[0,1]}(\mu,\sigma)$ & $\hat\mu = 0.48,\ \hat\sigma = 0.03$ & 0.82 & 0.84 & 0.88 & -3152.56 \\ 
   \cmidrule(lr){1-6}
\quad $\mathrm{Logit}\cN(\mu,\sigma)$ & $\hat\mu = -0.08,\ \hat\sigma = 0.14$ & 0.81 & 0.83 & 0.87 & -3122.10 \\ 
   \cmidrule(lr){1-6}
\quad $\mathrm{Kumaraswamy}(a,b)$ & $\hat a = 11.86,\ \hat b = 4191.35$ & 0.77 & 0.75 & 0.72 & -2926.22 \\ 
   \cmidrule(lr){1-6}
\addlinespace
\multicolumn{1}{@{}l}{\textit{Two-component mixtures}} &  &  &  &  &  \\ 
  \quad \makecell[l]{$\pi \cdot\mathrm{Beta}(\alpha_1,\beta_1) \,\, + $\\$(1-\pi)\cdot\,\mathrm{Beta}(\alpha_2,\beta_2)$} & \makecell[l]{$\hat\pi = 0.00\;\; \hat\alpha_1 = 171.33,\ \hat\beta_1 = 224.05$\\$\hat\alpha_2 = 97.74,\ \hat\beta_2 = 106.05$} & 0.81 & 0.83 & 0.87 & -3124.26 \\ 
   \cmidrule(lr){1-6}
\quad \makecell[l]{$\pi \cdot\mathrm{vonMises}(\phi_{11},\phi_{21}) \,\, + $\\$(1-\pi)\cdot\,\mathrm{vonMises}(\phi_{12},\phi_{22})$} & \makecell[l]{$\hat\pi = 0.10\;\; \hat\phi_{11} = -2.96,\ \hat\phi_{21} = 0.57$\\$\hat\phi_{12} = -58.61,\ \hat\phi_{22} = 7.19$} & 0.83 & 0.85 & 0.87 & \textbf{-3589.16} \\ 
   \cmidrule(lr){1-6}
\quad \makecell[l]{$\pi \cdot\cTN_{[0,1]}(\mu_1,\sigma_1) \,\, + $\\$(1-\pi)\cdot\,\cTN_{[0,1]}(\mu_2,\sigma_2)$} & \makecell[l]{$\hat\pi = 1.00\;\; \hat\mu_1 = 0.48,\ \hat\sigma_1 = 0.03$\\$\hat\mu_2 = 0.75,\ \hat\sigma_2 = 0.01$} & 0.82 & 0.84 & 0.88 & -3146.56 \\ 
   \cmidrule(lr){1-6}
\quad \makecell[l]{$\pi \cdot\mathrm{Logit}\cN(\mu_1,\sigma_1) \,\, + $\\$(1-\pi)\cdot\,\mathrm{Logit}\cN(\mu_2,\sigma_2)$} & \makecell[l]{$\hat\pi = 0.94\;\; \hat\mu_1 = -0.08,\ \hat\sigma_1 = 0.09$\\$\hat\mu_2 = -0.10,\ \hat\sigma_2 = 0.77$} & 0.83 & 0.84 & 0.86 & -3551.95 \\ 
   \cmidrule(lr){1-6}
\quad \makecell[l]{$\pi \cdot\mathrm{Kumaraswamy}(a_1,b_1) \,\, + $\\$(1-\pi)\cdot\,\mathrm{Kumaraswamy}(a_2,b_2)$} & \makecell[l]{$\hat\pi = 0.00\;\; \hat a_1 = 905.52,\ \hat b_1 = 1858.48$\\$\hat a_2 = 10.96,\ \hat b_2 = 2245.17$} & 0.75 & 0.74 & 0.70 & -2928.24 \\ 
   \cmidrule(lr){1-6}
\quad \makecell[l]{$\pi \cdot\mathrm{Beta}(\alpha_1,\beta_1) \,\, + $\\$(1-\pi)\cdot\,\mathrm{vonMises}(\phi_{12},\phi_{22})$} & \makecell[l]{$\hat\pi = 0.90\;\; \hat\alpha_1 = 278.28,\ \hat\beta_1 = 300.75$\\$\hat\phi_{12} = -2.98,\ \hat\phi_{22} = 0.57$} & 0.83 & 0.85 & 0.87 & -3588.63 \\ 
   \cmidrule(lr){1-6}
\quad \makecell[l]{$\pi \cdot\mathrm{Beta}(\alpha_1,\beta_1) \,\, + $\\$(1-\pi)\cdot\,\cTN_{[0,1]}(\mu_2,\sigma_2)$} & \makecell[l]{$\hat\pi = 0.95\;\; \hat\alpha_1 = 227.78,\ \hat\beta_1 = 246.64$\\$\hat\mu_2 = -29.20,\ \hat\sigma_2 = 12.75$} & 0.82 & 0.83 & 0.85 & -3526.61 \\ 
   \cmidrule(lr){1-6}
\quad \makecell[l]{$\pi \cdot\mathrm{Beta}(\alpha_1,\beta_1) \,\, + $\\$(1-\pi)\cdot\,\mathrm{Logit}\cN(\mu_2,\sigma_2)$} & \makecell[l]{$\hat\pi = 0.07\;\; \hat\alpha_1 = 4.46,\ \hat\beta_1 = 4.90$\\$\hat\mu_2 = -0.08,\ \hat\sigma_2 = 0.09$} & 0.83 & 0.84 & 0.86 & -3559.58 \\ 
   \cmidrule(lr){1-6}
\quad \makecell[l]{$\pi \cdot\mathrm{Beta}(\alpha_1,\beta_1) \,\, + $\\$(1-\pi)\cdot\,\mathrm{Kumaraswamy}(a_2,b_2)$} & \makecell[l]{$\hat\pi = 0.01\;\; \hat\alpha_1 = 573.60,\ \hat\beta_1 = 335.40$\\$\hat a_2 = 10.14,\ \hat b_2 = 1377.09$} & 0.73 & 0.71 & 0.67 & -2959.96 \\ 
   \cmidrule(lr){1-6}
\quad \makecell[l]{$\pi \cdot\mathrm{vonMises}(\phi_{11},\phi_{21}) \,\, + $\\$(1-\pi)\cdot\,\cTN_{[0,1]}(\mu_2,\sigma_2)$} & \makecell[l]{$\hat\pi = 0.04\;\; \hat\phi_{11} = -4.13,\ \hat\phi_{21} = 5.90$\\$\hat\mu_2 = 0.48,\ \hat\sigma_2 = 0.02$} & 0.83 & 0.85 & 0.88 & -3405.34 \\ 
   \cmidrule(lr){1-6}
\quad \makecell[l]{$\pi \cdot\mathrm{vonMises}(\phi_{11},\phi_{21}) \,\, + $\\$(1-\pi)\cdot\,\mathrm{Logit}\cN(\mu_2,\sigma_2)$} & \makecell[l]{$\hat\pi = 0.10\;\; \hat\phi_{11} = -2.98,\ \hat\phi_{21} = 0.57$\\$\hat\mu_2 = -0.08,\ \hat\sigma_2 = 0.08$} & 0.83 & 0.85 & 0.87 & -3588.53 \\ 
   \cmidrule(lr){1-6}
\quad \makecell[l]{$\pi \cdot\mathrm{vonMises}(\phi_{11},\phi_{21}) \,\, + $\\$(1-\pi)\cdot\,\mathrm{Kumaraswamy}(a_2,b_2)$} & \makecell[l]{$\hat\pi = 0.01\;\; \hat\phi_{11} = -231.52,\ \hat\phi_{21} = -231.89$\\$\hat a_2 = 15.48,\ \hat b_2 = 63444.47$} & 0.80 & 0.79 & 0.76 & -3273.31 \\ 
   \cmidrule(lr){1-6}
\quad \makecell[l]{$\pi \cdot\cTN_{[0,1]}(\mu_1,\sigma_1) \,\, + $\\$(1-\pi)\cdot\,\mathrm{Logit}\cN(\mu_2,\sigma_2)$} & \makecell[l]{$\hat\pi = 0.08\;\; \hat\mu_1 = 0.47,\ \hat\sigma_1 = 0.12$\\$\hat\mu_2 = -0.08,\ \hat\sigma_2 = 0.09$} & 0.83 & 0.84 & 0.87 & -3576.23 \\ 
   \cmidrule(lr){1-6}
\quad \makecell[l]{$\pi \cdot\cTN_{[0,1]}(\mu_1,\sigma_1) \,\, + $\\$(1-\pi)\cdot\,\mathrm{Kumaraswamy}(a_2,b_2)$} & \makecell[l]{$\hat\pi = 0.03\;\; \hat\mu_1 = -146586.09,\ \hat\sigma_1 = 73195.44$\\$\hat a_2 = 19.95,\ \hat b_2 = 1552490.19$} & 0.81 & 0.79 & 0.76 & -3477.42 \\ 
   \cmidrule(lr){1-6}
\quad \makecell[l]{$\pi \cdot\mathrm{Logit}\cN(\mu_1,\sigma_1) \,\, + $\\$(1-\pi)\cdot\,\mathrm{Kumaraswamy}(a_2,b_2)$} & \makecell[l]{$\hat\pi = 0.94\;\; \hat\mu_1 = -0.08,\ \hat\sigma_1 = 0.09$\\$\hat a_2 = 2.79,\ \hat b_2 = 5.07$} & 0.83 & 0.84 & 0.86 & -3556.47 \\ 
   \cmidrule(lr){1-6}
\textit{Kernel density estimate} & — & 0.83 & 0.82 & 0.81 & — \\ 
   \bottomrule
\end{tabular}
}
\caption{Maximum likelihood fits of generator models for pair 1. Rows list unimodal families, $2$-component mixtures, and the KDE. Columns report MLEs of generator parameters, plug-in estimates of Spearman's rho, Kendall's tau, and the Dette--Siburg--Stoimenov coefficient (sample estimates: $\hat{\rho}_n = 0.85$, $\hat{\tau}_n = 0.84$, $\hat{\xi}_n = 0.82$), and the AIC for parametric models. By AIC, the best-fitting parametric generator is the $2$-component von Mises mixture.} 
\label{tab:fit_v1}
\end{table}

\begin{table}[ht]
\centering
\resizebox{.94\linewidth}{!}{
\begin{tabular}{llcccc}
  \toprule
Generator & Estimated parameters & $\hat\rho$ & $\hat\tau$ & $\hat\xi$ & AIC \\ 
  \cmidrule(lr){1-1}\cmidrule(lr){2-2}\cmidrule(lr){3-5}\cmidrule(r){6-6}
\multicolumn{1}{@{}l}{\textit{Unimodal generators}} &  &  &  &  &  \\ 
  \quad $\mathrm{Beta}(\alpha,\beta)$ & $\hat\alpha = 66.73,\ \hat\beta = 67.88$ & 0.80 & 0.81 & 0.82 & -2607.40 \\ 
   \cmidrule(lr){1-6}
\quad $\mathrm{vonMises}(\phi_1,\phi_2)$ & $\hat\phi_1 = -11.54,\ \hat\phi_2 = 0.35$ & 0.78 & 0.76 & 0.70 & -2715.85 \\ 
   \cmidrule(lr){1-6}
\quad $\cTN_{[0,1]}(\mu,\sigma)$ & $\hat\mu = 0.50,\ \hat\sigma = 0.04$ & 0.81 & 0.81 & 0.82 & -2625.86 \\ 
   \cmidrule(lr){1-6}
\quad $\mathrm{Logit}\cN(\mu,\sigma)$ & $\hat\mu = -0.02,\ \hat\sigma = 0.17$ & 0.80 & 0.81 & 0.82 & -2601.03 \\ 
   \cmidrule(lr){1-6}
\quad $\mathrm{Kumaraswamy}(a,b)$ & $\hat a = 9.66,\ \hat b = 609.33$ & 0.74 & 0.72 & 0.68 & -2414.32 \\ 
   \cmidrule(lr){1-6}
\addlinespace
\multicolumn{1}{@{}l}{\textit{Two-component mixtures}} &  &  &  &  &  \\ 
  \quad \makecell[l]{$\pi \cdot\mathrm{Beta}(\alpha_1,\beta_1) \,\, + $\\$(1-\pi)\cdot\,\mathrm{Beta}(\alpha_2,\beta_2)$} & \makecell[l]{$\hat\pi = 0.00\;\; \hat\alpha_1 = 119.02,\ \hat\beta_1 = 146.92$\\$\hat\alpha_2 = 66.73,\ \hat\beta_2 = 67.88$} & 0.80 & 0.81 & 0.82 & -2601.40 \\ 
   \cmidrule(lr){1-6}
\quad \makecell[l]{$\pi \cdot\mathrm{vonMises}(\phi_{11},\phi_{21}) \,\, + $\\$(1-\pi)\cdot\,\mathrm{vonMises}(\phi_{12},\phi_{22})$} & \makecell[l]{$\hat\pi = 0.22\;\; \hat\phi_{11} = -3.31,\ \hat\phi_{21} = -0.13$\\$\hat\phi_{12} = -59.13,\ \hat\phi_{22} = 2.72$} & 0.83 & 0.81 & 0.78 & \textbf{-3162.54} \\ 
   \cmidrule(lr){1-6}
\quad \makecell[l]{$\pi \cdot\cTN_{[0,1]}(\mu_1,\sigma_1) \,\, + $\\$(1-\pi)\cdot\,\cTN_{[0,1]}(\mu_2,\sigma_2)$} & \makecell[l]{$\hat\pi = 0.83\;\; \hat\mu_1 = 0.49,\ \hat\sigma_1 = 0.02$\\$\hat\mu_2 = 0.51,\ \hat\sigma_2 = 0.12$} & 0.83 & 0.81 & 0.78 & -3134.05 \\ 
   \cmidrule(lr){1-6}
\quad \makecell[l]{$\pi \cdot\mathrm{Logit}\cN(\mu_1,\sigma_1) \,\, + $\\$(1-\pi)\cdot\,\mathrm{Logit}\cN(\mu_2,\sigma_2)$} & \makecell[l]{$\hat\pi = 0.87\;\; \hat\mu_1 = -0.03,\ \hat\sigma_1 = 0.10$\\$\hat\mu_2 = 0.01,\ \hat\sigma_2 = 0.63$} & 0.82 & 0.80 & 0.77 & -3053.37 \\ 
   \cmidrule(lr){1-6}
\quad \makecell[l]{$\pi \cdot\mathrm{Kumaraswamy}(a_1,b_1) \,\, + $\\$(1-\pi)\cdot\,\mathrm{Kumaraswamy}(a_2,b_2)$} & \makecell[l]{$\hat\pi = 0.00\;\; \hat a_1 = 125.16,\ \hat b_1 = 339.32$\\$\hat a_2 = 9.06,\ \hat b_2 = 405.88$} & 0.72 & 0.70 & 0.67 & -2419.25 \\ 
   \cmidrule(lr){1-6}
\quad \makecell[l]{$\pi \cdot\mathrm{Beta}(\alpha_1,\beta_1) \,\, + $\\$(1-\pi)\cdot\,\mathrm{vonMises}(\phi_{12},\phi_{22})$} & \makecell[l]{$\hat\pi = 0.78\;\; \hat\alpha_1 = 287.98,\ \hat\beta_1 = 296.52$\\$\hat\phi_{12} = -3.33,\ \hat\phi_{22} = -0.12$} & 0.83 & 0.81 & 0.78 & -3161.97 \\ 
   \cmidrule(lr){1-6}
\quad \makecell[l]{$\pi \cdot\mathrm{Beta}(\alpha_1,\beta_1) \,\, + $\\$(1-\pi)\cdot\,\cTN_{[0,1]}(\mu_2,\sigma_2)$} & \makecell[l]{$\hat\pi = 0.93\;\; \hat\alpha_1 = 160.27,\ \hat\beta_1 = 165.15$\\$\hat\mu_2 = 31.72,\ \hat\sigma_2 = 163.24$} & 0.80 & 0.79 & 0.76 & -3043.33 \\ 
   \cmidrule(lr){1-6}
\quad \makecell[l]{$\pi \cdot\mathrm{Beta}(\alpha_1,\beta_1) \,\, + $\\$(1-\pi)\cdot\,\mathrm{Logit}\cN(\mu_2,\sigma_2)$} & \makecell[l]{$\hat\pi = 0.11\;\; \hat\alpha_1 = 3.46,\ \hat\beta_1 = 3.18$\\$\hat\mu_2 = -0.03,\ \hat\sigma_2 = 0.10$} & 0.81 & 0.80 & 0.76 & -3078.34 \\ 
   \cmidrule(lr){1-6}
\quad \makecell[l]{$\pi \cdot\mathrm{Beta}(\alpha_1,\beta_1) \,\, + $\\$(1-\pi)\cdot\,\mathrm{Kumaraswamy}(a_2,b_2)$} & \makecell[l]{$\hat\pi = 0.89\;\; \hat\alpha_1 = 119.83,\ \hat\beta_1 = 124.23$\\$\hat a_2 = 9.96,\ \hat b_2 = 242.41$} & 0.82 & 0.82 & 0.81 & -2902.35 \\ 
   \cmidrule(lr){1-6}
\quad \makecell[l]{$\pi \cdot\mathrm{vonMises}(\phi_{11},\phi_{21}) \,\, + $\\$(1-\pi)\cdot\,\cTN_{[0,1]}(\mu_2,\sigma_2)$} & \makecell[l]{$\hat\pi = 0.01\;\; \hat\phi_{11} = 5.00,\ \hat\phi_{21} = 4.98$\\$\hat\mu_2 = 0.50,\ \hat\sigma_2 = 0.04$} & 0.80 & 0.80 & 0.80 & -2853.92 \\ 
   \cmidrule(lr){1-6}
\quad \makecell[l]{$\pi \cdot\mathrm{vonMises}(\phi_{11},\phi_{21}) \,\, + $\\$(1-\pi)\cdot\,\mathrm{Logit}\cN(\mu_2,\sigma_2)$} & \makecell[l]{$\hat\pi = 0.22\;\; \hat\phi_{11} = -3.34,\ \hat\phi_{21} = -0.12$\\$\hat\mu_2 = -0.03,\ \hat\sigma_2 = 0.08$} & 0.83 & 0.81 & 0.78 & -3161.87 \\ 
   \cmidrule(lr){1-6}
\quad \makecell[l]{$\pi \cdot\mathrm{vonMises}(\phi_{11},\phi_{21}) \,\, + $\\$(1-\pi)\cdot\,\mathrm{Kumaraswamy}(a_2,b_2)$} & \makecell[l]{$\hat\pi = 0.01\;\; \hat\phi_{11} = 8.86,\ \hat\phi_{21} = 1.29$\\$\hat a_2 = 9.88,\ \hat b_2 = 694.77$} & 0.73 & 0.71 & 0.67 & -2554.03 \\ 
   \cmidrule(lr){1-6}
\quad \makecell[l]{$\pi \cdot\cTN_{[0,1]}(\mu_1,\sigma_1) \,\, + $\\$(1-\pi)\cdot\,\mathrm{Logit}\cN(\mu_2,\sigma_2)$} & \makecell[l]{$\hat\pi = 0.17\;\; \hat\mu_1 = 0.51,\ \hat\sigma_1 = 0.12$\\$\hat\mu_2 = -0.03,\ \hat\sigma_2 = 0.09$} & 0.83 & 0.81 & 0.78 & -3133.49 \\ 
   \cmidrule(lr){1-6}
\quad \makecell[l]{$\pi \cdot\cTN_{[0,1]}(\mu_1,\sigma_1) \,\, + $\\$(1-\pi)\cdot\,\mathrm{Kumaraswamy}(a_2,b_2)$} & \makecell[l]{$\hat\pi = 0.08\;\; \hat\mu_1 = 20192.22,\ \hat\sigma_1 = 375986.40$\\$\hat a_2 = 19.99,\ \hat b_2 = 956081.16$} & 0.79 & 0.76 & 0.69 & -3003.76 \\ 
   \cmidrule(lr){1-6}
\quad \makecell[l]{$\pi \cdot\mathrm{Logit}\cN(\mu_1,\sigma_1) \,\, + $\\$(1-\pi)\cdot\,\mathrm{Kumaraswamy}(a_2,b_2)$} & \makecell[l]{$\hat\pi = 0.89\;\; \hat\mu_1 = -0.03,\ \hat\sigma_1 = 0.10$\\$\hat a_2 = 2.62,\ \hat b_2 = 3.33$} & 0.81 & 0.80 & 0.76 & -3077.10 \\ 
   \cmidrule(lr){1-6}
\textit{Kernel density estimate} & — & 0.82 & 0.80 & 0.74 & — \\ 
   \bottomrule
\end{tabular}
}
\caption{Maximum likelihood fits of generator models for pair 4. As in \cref{tab:fit_v1}, we report parameter estimates, plug-in estimates for $\rho, \tau$, and $\xi$ (sample estimates: $\hat{\rho}_n = 0.86$, $\hat{\tau}_n = 0.83$, $\hat{\xi}_n = 0.77$), and AIC for the parametric generators. The two-component von Mises mixture attains the lowest AIC.}
\label{tab:fit_v4}
\end{table}

\end{document}